\documentclass[10pt]{amsart}
\usepackage{amsmath}
\usepackage{amssymb}
\usepackage{amsfonts}
\usepackage{amsthm}
\usepackage{enumerate}
\usepackage[all]{xy}
\usepackage[latin1]{inputenc}
\usepackage{graphicx}
\usepackage{latexsym}
\usepackage{tikz}
\usepackage{color}
\input xy

\newtheorem{Theo}{Theorem}[section]
\newtheorem{Prop}[Theo]{Proposition}
\newtheorem{Cor}[Theo]{Corollary}
\newtheorem{Lemma}[Theo]{Lemma}

\theoremstyle{definition}

\newtheorem{Exam}[Theo]{Example}
\newtheorem{Remark}[Theo]{Remark}
\newtheorem{Remarks}[Theo]{Remarks}

\newtheorem{Defn}[Theo]{Definition}

\newcommand{\Rep}{{\rm Rep}}
\newcommand{\Hom}{{\rm Hom}}
\newcommand{\Ext}{{\rm Ext}}
\newcommand{\End}{{\rm End}}
\newcommand{\mmod}{{\rm mod}}

\newcommand{\Z}{\mathbb{Z}}
\newcommand{\C}{\mathcal{C}}

\newcommand{\cala}{\mathcal{A}}
\newcommand{\calf}{\mathcal{F}}
\newcommand{\calx}{\mathcal{X}}
\newcommand{\calc}{\mathcal{C}}
\newcommand{\calo}{\mathcal{O}}

\newcommand{\za}{\alpha}
\newcommand{\zb}{\beta}

\newcommand{\zD}{\Delta}

\newcommand{\zg}{\gamma}
\newcommand{\zG}{\Gamma}
\newcommand{\zl}{\lambda}

\begin{document}

\title{Group actions on cluster algebras and cluster categories}

\author{Charles Paquette}\address{Department of Mathematics and Computer Science, Royal Military College of Canada,
Kingston, ON K7K 7B4, Canada}
\email{charles.paquette@usherbrooke.ca}
\author{Ralf Schiffler}\thanks{ The authors were supported by the NSF CAREER grant DMS-1254567. The first author was also supported by the Natural Sciences and Engineering Research Council of Canada, while the second was also supported by NSF Grant DMS-1800860. They are thankful to Thomas Lam, Pavel Tumarkin, Michael Gekhtman and Pierre-Guy Plamondon for useful discussions. They also want to thank an anonymous referee, as the comments made led to an improved version of the paper}
\address{Department of Mathematics, University of Connecticut,
Storrs, CT 06269-3009, USA}
\email{schiffler@math.uconn.edu}
\begin{abstract}
 We introduce admissible group actions on cluster algebras, cluster categories and quivers with potential and study the resulting orbit spaces.
The orbit space of the cluster algebra has the structure of a generalized cluster algebra. This generalized cluster structure is different from those introduced by Chekhov-Shapiro and Lam-Pylyavskyy.
For group actions on cluster algebras from surfaces, we describe the generalized cluster structure of the orbit space in terms of a triangulated orbifold. In this case, we give a complete list of exchange polynomials, and we classify the algebras of rank 1 and 2.
We also show that every admissible group action on a cluster category induces a precovering from the cluster category to the cluster category of orbits. Moreover this precovering is dense if the categories are of finite type.
\end{abstract}

\maketitle

\section{Introduction}
Cluster algebras were introduced by Fomin and Zelevinsky in \cite{FZ} in the context of canonical bases in Lie theory and total positivity. A cluster algebra $\cala=\cala(Q)$ is a subalgebra of a field of rational functions in $n$-variables whose generators, the cluster variables, are constructed recursively from an initial seed of $n$-variables. This construction, and hence the cluster algebra, is determined by a quiver $Q$ with $n$ vertices.
A strong connection between cluster algebras and representation theory was realized via cluster categories, which were introduced in \cite{BMRRT,CCS}.
The cluster character of \cite{CC,Palu} is a map from the set of objects of the cluster category $\calc$ to a ring of Laurent polynomials which provides a direct formula for the cluster variables
and gives a bijection between reachable cluster-tilting objects in $\calc$ and clusters in $\cala$. Cluster categories have been generalized in \cite{Amiot} using the theory of quivers with potential developed in \cite{DWZ}.

In this paper, we study certain group actions on cluster algebras, cluster categories and quivers with potential. We say that a group of automorphisms $G$ is admissible if it acts freely on a given cluster in $\cala$, or, equivalently, on a given cluster-tilting object in $\calc$.  On the level of quivers with potential this means that the group acts freely on the vertices of the quiver.

We define and study the corresponding orbit spaces in each of these settings. On the level of quivers with potential, we obtain a $G$-covering from the Jacobian algebras of the quiver with potential to the Jacobian algebras of the orbit quiver, see Proposition \ref{prop 21} and Corollary \ref{CovJacobian}. On the level of cluster categories, we have a $G$-precovering from the cluster category $\calc$ of the quiver with potential to the cluster category $\calc_G$ of $G$-orbits. Recall that a covering functor is a precovering functor that is also dense.
In particular, we show that $\calc$ is of finite type if and only if $\calc_G$ is of finite type, and that in this case our functor is a $G$-covering that preserves Auslander-Reiten triangles, see Propositions \ref{FiniteType} and \ref{PropARQUivers}.

The orbit space of the cluster algebra can be defined in at least two ways. On the one hand, we can take the quotient of the cluster algebra $\cala$ by identifying the cluster variables that lie in the same $G$-orbit. On the other hand, we can take the algebra $\cala_G$ generated by the images under the cluster character of all summands of cluster-tilting objects obtained from the initial cluster by $G$-orbit mutations. We will see that there are two natural ways to define a cluster character, however, one seems better behaved with respect to the precovering functor.
These two constructions yield the same algebra under some mild conditions. In general, the algebra $\cala_G$ is not an honest cluster algebra but rather a generalized cluster algebra. We point out that our generalized cluster structure is not the same as the one constructed by Chekhov and Shapiro in \cite{ChSh} and also not the one of Lam and Pylyavskyy in \cite{LP}, see Remark \ref{rem LP}.

\medskip
We devote particular attention to group actions on cluster algebras from surfaces. In this case, the initial cluster of $\cala$ corresponds to a triangulation of a surface with marked points, and the elements of $G$ are elements of the mapping class group of the surface that map the triangulation to itself. The admissibility condition translates to $G$ acting freely on the arcs of the triangulation.

The orbit space of such a group action is an orbifold.
In this situation, we give an explicit list of the exchange polynomials of the orbit cluster algebra $\cala_G$ in terms of the orbifold. We show that the algebra generated by all variables obtained by finite sequences of generalized mutations with respect to these exchange polynomials is indeed equal to the generalized cluster algebra $\cala_G$ of orbits. Some of the exchange polynomials that we get are similar to the exchange polynomials of the quasi-cluster algebras as defined in \cite{DP}. In general, the quasi-cluster algebras are different from the orbit cluster algebras $\cala_G$ coming from the action of a group $G$ on a surface; see Remark \ref{RemDP}.

We also point out that our orbifolds are different from the orbifolds considered by Felikson, Shapiro and Tumarkin \cite{FeShTu}.

The paper is organized as follows. In Section \ref{sect 2}, we recall background on quivers with potential and define admissible group actions. Our results on $G$-coverings follow in Section \ref{sect 3}. In Section \ref{sect 4}, we study admissible actions on the level of triangulated surfaces and introduce the orbifolds. Section \ref{TechnicalSection} is devoted to the computation of the exchange polynomials for the orbifolds. We use these computations in Section~\ref{sect 6} in order to define the generalized cluster algebra of an orbifold. We classify the four algebras of rank 1 and the six algebras of rank 2 in the Subsections \ref{sect 61} and \ref{sect 62}, respectively. Finally, in Section \ref{sect 7}, we come back to the study of cluster categories. We show that, in the surface case, the generalized cluster algebra of the orbifold is equal to the cluster algebra $\cala_G$, and in the finite type case,  the precovering of cluster categories is actually a covering. In order to study the cluster algebra in case $\C$ is the cluster category of a Dynkin quiver, we introduce a cluster character in $\C_G$ that gives all cluster variables of $\cala_G$.

\section{Preliminaries}\label{sect 2}

In this paper, $k$ denotes an algebraically closed field and $G$ a finite group whose order is not divisible by the characteristic of $k$. Also, $Q=(Q_0, Q_1)$ denotes a finite quiver. We compose paths like functions, that is, from right to left.

\subsection{Quivers with potential and automorphisms}\label{Quivers} Let $Q$ be a quiver. If $p, p'$ are two oriented cycles in $Q$, we write $p \sim p'$ if one can get $p'$ by cyclically rotating $p$. In other words, if $p = \alpha_r \cdots \alpha_2  \alpha_1$, then there exists $1 \le i \le r$ such that $p' = \alpha_{i-1}\cdots\alpha_1\alpha_r\cdots\alpha_{i+1}\alpha_i$. This relation is clearly an equivalence relation and the class of a cycle $p$ is denoted $[p]$. We define ${\rm cyc}(Q)$ to be the set of all equivalence classes of cycles of $Q$. Recall that a \emph{potential} for $Q$ is a (possibly infinite) linear combination of distinct elements in cyc$(Q)$. In this paper, $W$ always denotes a potential for $Q$. The pair $(Q,W)$ is called a quiver with potential \cite{DWZ}.

\medskip

An oriented cycle of length one is called a \emph{loop} and an oriented cycle of length two is called a \emph{$2$-cycle}. If $a$ is a vertex of $Q$ such that there are no loops and no $2$-cycles at $a$, then we can define the mutation $\mu_a(Q, W) = (Q', W')$ of $(Q, W)$ which is the \emph{mutation in direction} $a$ of the quiver with potential $(Q, W)$; see \cite{DWZ}. In particular, $Q'$ is a quiver with the same vertex set as the one for $Q$ and $W'$ is a potential for $Q'$. In general, the quiver $Q'$ may have $2$-cycles at $a$ (but no loops at $a$). There is a notion of \emph{right-equivalence} of quivers with potentials and even if $Q'$ has $2$-cycles at $a$, it could be possible to find a quiver with potential $(Q'', W'')$ that is right-equivalent to $(Q', W')$ so that $Q''$ has no $2$-cycles at $a$. Some authors are interested in the case where $W$ is \emph{non-degenerate}, which means that the quivers obtained from $(Q, W)$ by a finite sequence of mutations do not have $2$-cycles (up to right-equivalence). In particular, the original quiver $Q$ has no loops and no $2$-cycles. Having no $2$-cycles (and no loops) at a vertex $a$ of a quiver $Q$ is generally needed to define mutation in direction $a$ of $Q$. So in the non-degenerate setting, one can iteratively perform mutations of $(Q, W)$ in all possible directions, and at the quiver level, this is the usual quiver mutation as defined by Fomin-Zelevinsky in \cite{FZ}.

\medskip

Let $\varphi$ be an automorphism of $Q$. Clearly, $\varphi$ induces a permutation on cyc$(Q)$. We say that $\varphi$ is an \emph{automorphism} of $(Q,W)$ provided that whenever $\lambda [p]$ is a summand of $W$, with $\zl \in k$, then $\lambda \varphi[p]$ is also a summand of $W$. Let $G$ be a group of automorphisms of $(Q,W)$.
We call $G$ \emph{admissible} if each $\varphi \in G$ acts freely on $Q_0$, that is, if $\varphi(x)=x$ for some $x \in Q_0$ then $\varphi$ has to be the identity automorphism. Note that the generators of a group $G$ of automorphisms of $(Q,W)$ may act freely on $Q_0$ without $G$ being admissible.

\medskip

Since each element of $G$ acts freely on the vertices of $Q$, clearly, each element of $G$ also acts freely on the arrows of $Q$. For $a \in Q_0 \cup Q_1$, we denote by $G a$ the $G$-orbit of $a$. By the above observation, one has $|G a| = |G|$. In particular, $|G|$ divides both $|Q_0|, |Q_1|$. We define a quiver $Q_G$, called the \emph{orbit quiver} of $Q$, by $$(Q_G)_0 = \{Gx \mid x \in Q_0\}\;\; \text{and} \;\; (Q_G)_1 = \{G\alpha \mid \alpha \in Q_1\}.$$ For an illustration, see Example \ref{ex 2.2} below.

\subsection{Jacobian algebras and automorphisms}\label{sect Jac}
Let $(Q,W)$ be a quiver with potential. We recall the construction of the Jacobian algebra of $(Q,W)$. Given an arrow $\alpha$ in $Q$, consider $\partial_\alpha$ the partial differential operator on $kQ$ such that if $p=\alpha_r\cdots\alpha_1$, then
$$\partial_\alpha(p)=\sum_{i=1}^r\alpha_{i-1}\cdots\alpha_1\alpha_r\cdots\alpha_{i+1}\delta_{\alpha_i,\alpha}$$
where $\delta$ stands for the Kronecker symbol.
One can define $\partial_\alpha$ on an element $[p] \in {\rm cyc}(Q)$ by defining $\partial_\alpha[p] = \partial_\alpha(p)$. Take $I$ the ideal of $kQ$ generated by all $\partial_\alpha(W)$ where $\alpha$ runs through the set of arrows of $Q$. The \emph{Jacobian algebra} of $(Q,W)$, denoted $J(Q,W)$, is defined to be $\widehat{kQ}/\widehat{I}$ where $\widehat{kQ}$ is the completed path algebra of $Q$ and $\widehat{I}$ is the completion of $I$ in $\widehat{kQ}$. This algebra is not always finite dimensional. In case it is finite dimensional, the pair $(Q,W)$ is called \emph{Jacobi-finite}.

\medskip

Now, let $G$ be an admissible group of automorphisms of $(Q,W)$. Given an element $[p]$ of cyc$(Q)$, we denote by $G[p]$ its $G$-orbit, which is a subset of cyc$(Q)$. Let $\mathcal{E}$ be the set of all $G$-orbits in cyc$(Q)$. We can decompose $W$ as
\begin{equation}\label{eq 1}
 W = \sum_{e \in \mathcal{E}}\lambda_{e}\left(\sum_{[p] \in e}[p]\right).
\end{equation}

\begin{Lemma}
Any $\varphi \in G$ induces an automorphism of $J(Q,W)$.
\end{Lemma}

\begin{proof}
Let $\varphi \in G$. Clearly, we can extend $\varphi$ to a continuous automorphism of $\widehat{kQ}$, still denoted $\varphi$.
Observe that for all $\alpha \in Q_1$ and $[p] \in {\rm cyc}(Q)$, we have $\varphi(\partial_\alpha[p]) = \partial_{\varphi(\alpha)}\varphi([p])$.
Therefore, equation (\ref{eq 1}) implies that  $\varphi(\partial_\alpha W) = \partial_{\varphi(\alpha)}W$. This yields $\varphi(I) = I$.
Therefore, $\widehat{I} = \widehat{\varphi(I)}=\varphi(\widehat{I})$, since $\varphi$ is continuous. Thus, we get an automorphism $\varphi$ at the level of the quotient $\widehat{kQ}/\widehat{I}$.
\end{proof}

The next lemma guarantees that the equivalence classes of cycles in $Q_G$ coincide with the $G$-orbits of equivalence classes of cycles in $Q$.

\begin{Lemma} Let $[p],[q] \in {\rm cyc}(Q)$ with $p=\alpha_r \cdots \alpha_1$ and $q = \beta_r\cdots\beta_1$. If we have $[G\alpha_r \cdots G\alpha_1] = [G\beta_r\cdots G\beta_1]$, then $G[p] = G[q]$.
\end{Lemma}

\begin{proof}We are given that
$$[G\alpha_r \cdots G\alpha_1] = [G\beta_r\cdots G\beta_1].$$
By cyclically permuting $q$ if necessary, we may assume that, for each $i$, the arrows $\alpha_i, \beta_i$ lie in the same $G$-orbit. Let $g \in G$ with $g\alpha_1 = \beta_1$. Observe that the arrows $g\alpha_2, \beta_2$ both start at the same vertex of $Q$ and lie in the same $G$-orbit. Therefore, since $G$ is admissible, we have $g\alpha_2 = \beta_2$. By induction, we have $g \alpha_i = \beta_i$ for $1 \le i \le r$, that is, $gp=q$.
\end{proof}

Observe that we have a $k$-linear functor $\pi: kQ \to kQ_G$ of the corresponding $k$-categories such that for $a \in Q_0 \cup Q_1$, $\pi(a)=Ga$. Later, we will study this functor in more details.
Recall that since $G$ is an admissible group of automorphisms of $(Q,W)$, we can decompose the potential $W$ as
$$W = \sum_{G[p] \in {\rm cyc}(Q_G)}\lambda_{G[p]}\left(\sum_{[q] \in G[p]}[q]\right).$$
We define the following potential on the orbit quiver $Q_G$
$$W_G = \sum_{G[p] \in {\rm cyc}(Q_G)}\left(\lambda_{G[p]}\big|G[p]\big|\right)G[p].$$ Observe that
$$\partial_{G\alpha}(G[p])=|{\rm stab}(G,[p])|\pi\left(\sum_{[q] \in G[p]}\partial_{\alpha}[q]\right),$$
where ${\rm stab}(G,[p]) = \{g \in G \mid g[p]=[p]\}$ is the stabilizer subgroup of $[p]$. Since $$|{\rm stab}(G,[p])||G[p]| = |G|,$$
we see that
\begin{eqnarray*}\partial_{G\alpha}W_G &=& \sum_{G[p] \in {\rm cyc}(Q_G)}\left(\lambda_{G[p]}\big|G[p]\big|\right)\partial_{G\alpha}(G[p])\\ &=& \sum_{G[p] \in {\rm cyc}(Q_G)}\left(\lambda_{G[p]}\big|G\big| \right)\pi\left(\sum_{[q] \in G[p]}\partial_{\alpha}[q]\right)\\ &=& |G|\pi(\partial_\alpha(W)).\end{eqnarray*}
Define $I_G$ to be the ideal of $\widehat{kQ_G}$ generated by the elements $\partial_{G\alpha}(W_G)$. Since the characteristic of $k$ does not divide $|G|$, we see that $\pi$ sends the generator $\partial_\alpha(W)$ of $I$ to a scalar multiple of the generator $\partial_{G\alpha}(W_G)$ of $I_G$. We define the Jacobian algebra of the orbit as $J(Q_Q,W_G) = \widehat{kQ_G}/\widehat{I_G}$.

\begin{Exam}\label{ex 2.2} Consider the following quiver $Q$:
$$\xymatrix{
 & & b_1 \ar[rr]^{\alpha_1}& & c_1 \ar[dl]^{\beta_1} & & \\
&&&a_1 \ar[dr]^{\delta_1} \ar[ul]^{\gamma_1}&&&\\
c_3 \ar[rr]^{\beta_3} & & a_3 \ar[ur]^{\delta_3} \ar[dl]^{\gamma_3} & & a_2 \ar[rr]^{\gamma_2} \ar[ll]^{\delta_2} & & b_2 \ar[dl]^{\alpha_2}\\
& b_3 \ar[ul]^{\alpha_3} & & & & c_2 \ar[ul]^{\beta_2} &
}$$
Consider the cyclic group $G$ of order $3$ with generator $g$ such that $g$ acts on $Q_0 \cup Q_1$ by increasing by $1$, modulo $3$, the indices of the symbols. Clearly, $G$ is admissible. Take $W = \delta_3\delta_2\delta_1 + \sum_{i=1}^3\gamma_i\beta_i\alpha_i$. Then $G$ is an admissible group of automorphisms of $(Q,W)$. The quiver $Q_G$ is
$$\xymatrix{&b \ar[dr]^\alpha & \\ a \ar[ur]^\gamma \ar@(ld,lu)^\delta & & c \ar[ll]^\beta}$$
where $\delta = G\delta_1, \alpha = G\alpha_1, \beta = G\beta_1, \gamma = G\gamma_1, a = Ga_1, b = Gb_1$ and $c = Gc_1$. Now,
$$W_G = \delta^3 + 3\gamma\beta\alpha$$
The generators of $I_G$ are $3\delta^2, 3\gamma\beta, 3\beta\alpha, 3\alpha\gamma$. We have
$$J(Q_G,W_G)=kQ_G/\langle \partial_{G\alpha} W_G \mid G\alpha \in (Q_G)_1\rangle = kQ_G/\langle 3\delta^2, 3\gamma\beta, 3\beta\alpha, 3\alpha\gamma \rangle.$$
\end{Exam}

\subsection{Ginzburg DG-algebras}\label{sect DG}
Now, let $\Gamma(Q,W)$ be the (completed) Ginzburg DG-algebra of $(Q, W)$. Recall that as a graded algebra, $\Gamma(Q,W)$ is generated in non-positive degrees and is the completed graded quiver algebra $\widehat{k\overline{Q}}$ where $\overline{Q}$ is obtained from the quiver $Q$ by adding the following arrows: for each arrow $\alpha: i \to j$ in $Q$, we add an arrow $\alpha^*: j \to i$; and for each vertex $i$ in $Q$, we add a loop $t_i : i \to i$. To make $\widehat{k\overline{Q}}$ a graded algebra, we need to define the degree of the arrows of $\overline{Q}$. The arrows from $Q_1$ as well as the trivial paths $\{e_i \mid i \in Q_0\}$ are declared to be of degree zero. The arrows in $\{\alpha^* \mid \alpha \in Q_1\}$ are declared to be of degree $-1$ and the loops $\{t_i \mid i \in Q_0\}$ are of degree $-2$.
The DG-algebra $\Gamma(Q,W)$ is equipped with a continuous differential $\mathfrak{d}$ defined on $\alpha^*$ and $t_i$ by
$$\mathfrak{d}\alpha^* = \partial_\alpha(W)$$
and
$$\mathfrak{d}t_i = e_i\left(\sum_{\alpha \in Q_1}(\alpha \alpha^* - \alpha^*\alpha)\right)e_i,$$
and extended by the Leibniz rule to all of $\Gamma(Q,W)$. In particular, $\mathfrak{d}$ vanishes on $kQ$.
Given an automorphism $\varphi$ of $(Q,W)$, we extend its action to a unique (graded) automorphism of the graded algebra $k\overline{Q}$ as follows. We set $\varphi(\alpha^*) = (\varphi(\alpha))^*$ and $\varphi(t_i) = t_{\varphi(i)}$. This clearly extends to a continuous automorphism of $\widehat{ k\overline{Q}}$.

\section{The cluster category of $G$-orbits}\label{sect 3}
In this section we define the cluster category of $G$-orbits as the cluster category of the quiver $Q_G$ with its corresponding potential. When $G$ is an admissible group of automorphisms of $(Q,W)$ such that $(Q,W)$ is Jacobi-finite, we will see that we have two Hom-finite $2$-Calabi-Yau triangulated categories $\C(Q,W), \C(Q_G, W_G)$ associated to the quivers with potentials $(Q,W), (Q_G, W_G)$, respectively. These categories will be called cluster categories and we will show that we have a $G$-precovering functor $F: \C(Q,W) \to \C(Q_G, W_G)$ (see Proposition \ref{prop 4.1}) and this functor is compatible with mutations (see Subsection \ref{sect cto}). Precoverings of cluster categories together with mutations are also studied in \cite{OO} with the purpose of mutating some quivers of endomorphism algebras of $2$-Calabi-Yau tilted algebras having loops or $2$-cycles.

\subsection{Coverings of $k$-categories}\label{Covering}

In this subsection, we introduce the notion of $G$-(pre)covering of algebras or categories. A \emph{skeletal} category is a category for which different objects are not isomorphic. Let $\mathcal{A}, \mathcal{B}$ be two skeletal $k$-categories and $G$ be a group of automorphisms of $\mathcal{A}$. A $k$-linear functor $F: \mathcal{A} \to \mathcal{B}$ is called a $G$-\emph{precovering} if we have functorial isomorphisms $$\bigoplus_{g \in G}\mathcal{A}(a,gb) \cong \mathcal{B}(Fa,Fb)$$
and
$$\bigoplus_{g \in G}\mathcal{A}(ga,b) \cong \mathcal{B}(Fa,Fb)$$
induced by the sum of the images of $F$.
If, moreover, the functor $F$ is surjective, then $F$ is called a $G$-\emph{covering}. We refer the reader to Bongartz-Gabriel \cite{BG} for more details on $G$-coverings.

\medskip

These definitions can be adapted to the differential graded cases. Assume now that $\mathcal{A}, \mathcal{B}$ are skeletal DG $k$-categories with differentials $\mathfrak{d}_\mathcal{A}, \mathfrak{d}_\mathcal{B}$, respectively. Let $F: \mathcal{A} \to \mathcal{B}$ be a $k$-linear functor that is graded (that is, respect the grading of morphisms) and commutes with the differentials. The functor $F$ is called a $G$-\emph{precovering} of DG-categories if, for $i \in \Z$, we have functorial isomorphisms
$$\bigoplus_{g \in G}\mathcal{A}(a,gb)_i \cong \mathcal{B}(Fa,Fb)_i$$
and
$$\bigoplus_{g \in G}\mathcal{A}(ga,b)_i \cong \mathcal{B}(Fa,Fb)_i$$
of degree $i$ maps induced by $F$.
If, moreover, the functor $F$ is surjective, then $F$ is called a $G$-\emph{covering}.

\medskip

Let $A$ be a $k$-algebra having a complete set $e_1, \ldots, e_n$ of pairwise orthogonal primitive idempotents. It will be convenient for us to think of $A$ as a (skeletal) $k$-category, also denoted $A$. The objects of $A$ are the idempotents of $A$ and the morphisms from $e_i$ to $e_j$ are given by the elements in $e_jAe_i$. This is a Hom-finite category if and only if $A$ is finite dimensional. If $A$ is a DG algebra, then the corresponding category is a DG category.
Observe that if $x \in e_jAe_i$ and $y \in e_lAe_k$ with $k \ne j$, then $yx$ is not defined in the category $A$ while it is zero in the algebra $A$.

Let $G$ be a group of admissible automorphisms of $(Q,W)$. Recall that we have a $k$-linear functor $\pi: kQ \to kQ_G$ of $k$-categories such that for $a \in Q_0 \cup Q_1$, $\pi(a)=Ga$. This functor $\pi$ is clearly a $G$-covering. Moreover, it extends to a $k$-linear continuous functor $\pi: \widehat{kQ} \to \widehat{kQ_G}$, which is also a $G$-covering.

\begin{Prop}\label{prop 21}
We have a $G$-covering $$\pi: J(Q,W) \to J(Q_G, W_G)$$
induced by the $G$-covering $\pi: \widehat{kQ} \to \widehat{kQ_G}.$
\end{Prop}

\begin{proof} For vertices $a,b$ in $Q_0$, we have a functorial isomorphism $$p:\bigoplus_{g \in G}\widehat{kQ}(a,gb) \cong \widehat{kQ_G}(Ga,Gb)$$
which, by the results in \ref{sect Jac}, restricts to an isomorphism
$$f:\bigoplus_{g \in G}\widehat{I}(a,gb) \cong \widehat{I_G}(Ga,Gb).$$
Now, consider the commutative diagram
$$\xymatrix{0 \ar[r] & \bigoplus_{g \in G}\widehat{I}(a,gb) \ar[r] \ar[d]^f & \bigoplus_{g \in G}\widehat{kQ}(a,gb) \ar[r] \ar[d]^p & \bigoplus_{g \in G}J(Q,W)(a,gb) \ar[r] & 0\\ 0 \ar[r] & \widehat{I_G}(Ga,Gb) \ar[r] & \widehat{kQ_G}(Ga,Gb) \ar[r] & J(Q_G,W_G)(Ga,Gb) \ar[r] & 0}$$
There is an induced isomorphism
$$h:\bigoplus_{g \in G}J(Q,W)(a,gb) \to J(Q_G,W_G)(Ga,Gb),$$
which is functorial. Similarly, there is a functorial isomorphism
$$\bigoplus_{g \in G}J(Q,W)(ga,b) \to J(Q_G,W_G)(Ga,Gb).$$
\end{proof}

\begin{Lemma} \label{LemmaDiff}Let $\varphi$ be an automorphism of $(Q,W)$ and extend $\varphi$ to a graded automorphism of $\widehat{k\overline{Q}}$ as previously. Then $\varphi$ induces an automorphism of $\Gamma(Q,W)$, that is, $\varphi$ commutes with the differential $\mathfrak{d}$.
\end{Lemma}

\begin{proof}
It suffices to check the compatibility on the arrows of degree $-1,-2$. We have $$\mathfrak{d}\varphi(\alpha^*) = \mathfrak{d}(\varphi(\alpha)^*) = \partial_{\varphi(\alpha)}(W) = \partial_{\varphi(\alpha)}(\varphi(W)) = \varphi(\partial_\alpha(W)) = \varphi\mathfrak{d}(\alpha^*)$$
and
\begin{eqnarray*}\mathfrak{d}\varphi(t_i) &=& \mathfrak{d}(t_{\varphi(i)})\\& = & e_{\varphi(i)}\left(\sum_{\alpha \in Q_1}(\alpha \alpha^* - \alpha^*\alpha)\right)e_{\varphi(i)}\\ &=& e_{\varphi(i)}\left(\sum_{\alpha \in Q_1}(\varphi(\alpha) \varphi(\alpha)^* - \varphi(\alpha)^*\varphi(\alpha))\right)e_{\varphi(i)}\\ &=& \varphi\left(e_i\left(\sum_{\alpha \in Q_1}(\alpha \alpha^* - \alpha^*\alpha)\right)e_i\right)\\& =& \varphi\mathfrak{d}(t_i).\end{eqnarray*}
\end{proof}

Now, consider the (completed) Ginzburg orbit DG-algebra $\Gamma(Q_G, W_G)$ with differential $\mathfrak{d}_G$. As seen in Lemma \ref{LemmaDiff}, we have $\mathfrak{d}\varphi = \varphi\mathfrak{d}$ whenever $\varphi \in G$. This means that the differential $\mathfrak{d}_G$ in $\Gamma(Q_G, W_G)$ comes from the differential $\mathfrak{d}$ of $\Gamma(Q,W)$. In order to consider $G$-coverings of DG-algebras, we can naturally think of the Ginzburg DG-algebras $\Gamma(Q,W), \Gamma(Q_G, W_G)$ as DG-categories. We get a graded $G$-covering
$$\pi: \Gamma(Q,W) \to \Gamma(Q_G, W_G),$$
of DG-categories, that is, for each $i \le 0$, we have natural isomorphisms
$$\bigoplus_{g \in G}\widehat{k\overline{Q}}(a,gb)_i \cong \widehat{k\overline{Q_G}}(Ga,Gb)_i$$
and
$$\bigoplus_{g \in G}\widehat{k\overline{Q}}(ga,b)_i \cong \widehat{k\overline{Q_G}}(Ga,Gb)_i$$
of degree $i$ maps. Moreover, $\pi$ commutes with the differentials $\mathfrak{d}_G$ and $\mathfrak{d}$.

\begin{Remark}
Observe that using the $G$-covering functor $\pi: \Gamma(Q,W) \to \Gamma(Q_G, W_G)$ together with the fact that the Jacobian algebra is the zero-th cohomology of the Ginzburg DG-algebra, we can recover Proposition \ref{prop 21}.
\end{Remark}

\subsection{Perfect derived categories}

Let $G$ be an admissible group of automorphisms of $(Q,W)$ and consider the graded $G$-covering
$\pi: \Gamma(Q,W) \to \Gamma(Q_G, W_G)$ as obtained above. Given a DG $k$-algebra (or category) $\Lambda$, we let $\mathcal{H}(\Lambda)$ denote the homotopy category of the category of DG $\Lambda$-modules and we let per$\Lambda$ denote the full subcategory of $\mathcal{H}(\Lambda)$ of the perfect DG $\Lambda$-modules: it is the smallest full triangulated subcategory of $\mathcal{H}(\Lambda)$ containing $\Lambda$ that is closed under isomorphisms and direct summands. Finally, we let f.d.$\Lambda$ denote the full subcategory of $\mathcal{H}(\Lambda)$ of the DG-modules whose total homology is finite dimensional. When $\Lambda$ is a Ginzburg DG-algebra of a Jacobi-finite quiver with potential, the subcategory f.d.$\Lambda$ is a triangulated subcategory of per$\Lambda$, and consequently, the quotient per$\Lambda/{\rm f.d.}\Lambda$ is a Hom-finite $2$-Calabi-Yau triangulated $k$-category; see \cite{Amiot}.

\medskip

We want to define a functor
$$F: {\rm per}\Gamma(Q,W) \to {\rm per}\Gamma(Q_G, W_G)$$
at the level of the perfect derived categories of DG-modules. Let $M^\bullet = (M_i)_{i \in \Z}$ be a DG-module in ${\rm per}\Gamma(Q,W)$ with differential $(d_i: M_i \to M_{i+1})_{i \in \Z}$. Observe that each $M_i$ is a $kQ$-modules and each $d_i$ is a morphism of $kQ$-modules. Consider the $G$-covering $\pi: kQ \to kQ_G$. There is an induced push-down functor $\pi_\lambda: {\rm Rep}(Q) \to {\rm Rep}(Q_G)$. For $x \in Q_0$, we have $(\pi_\lambda M)(Gx) = \oplus_{g \in G}M(gx)$ and for $\alpha \in Q_1$, we have $(\pi_\lambda M)(G\alpha) = \oplus_{g \in G}M(g\alpha)$. This functor $\pi_\lambda$ is a $G$-precovering. We define $F M^\bullet$ to be the complex $(\pi_\lambda M_i)_{i \in \Z}$ with differentials $(\pi_\lambda d_i)_{i\in \Z}$. We need to check that this is well defined. First of all, since $\pi_\lambda$ is a functor, it is clear that $(\pi_\lambda d_i)_{i \in \Z}$ is a differential. Fix $i \in \Z$. We have
$$\pi_\lambda M_i = \bigoplus_{Gx \in (Q_G)_0}(\pi_\lambda M_i) e_{Gx},$$
where $(\pi_\lambda M_i) e_{Gx} = \bigoplus_{y \in Gx}M_i e_y$. Assume that $\alpha: t \to s$, so that $\alpha^*: s \to t$. Then $\alpha^*$ induces a linear map $(M^\bullet(\alpha^*))_i: M_ie_s \to M_{i-1}e_t$. Therefore, for $g \in G$, we have a linear map $(M^\bullet(g\alpha^*))_i = (M^\bullet((g\alpha)^*))_i: M_ie_{gs} \to M_{i-1}e_{gt}$. As we have $(\pi_\lambda M_i) e_{Gs} = \bigoplus_{g \in G}M_i e_{gs}$ and $(\pi_\lambda M_{i-1}) e_{Gt} = \bigoplus_{g \in G}M_{i-1} e_{gt}$, this induces a linear map $(\pi_\lambda M_i) e_{Gs} \to (\pi_\lambda M_{i-1}) e_{Gt}$ that we define to be the action of $G\alpha^*$ on $(F M^\bullet)_i = \pi_\lambda M_i$. Similarly, we can define the action of $Gt_i$ on $(F M^\bullet)_i = \pi_\lambda M_i$. This makes $F M^\bullet$ a graded $k\overline{Q_G}$-module. Since $\pi$ sends the ideal generated by the arrows of $\bar Q$ to the ideal generated by arrows of $\bar Q_G$, we have that $F M^\bullet$ is actually a $\widehat{k\overline{Q_G}}$-module. One has to check that the differential $(\pi_\lambda d_i)_{i \in \Z}$ satisfies the Leibniz rule and one needs to define $F$ on morphisms. For this purpose, let $f^\bullet = (f_i)_{i \in \Z}: M^\bullet \to N^\bullet$ be a morphism of DG-modules. We define $F f$ to be the morphism $(\pi_\lambda f_i)_{i \in \Z}$. For a homogeneous element $a$ in a DG-algebra, we let $|a|$ denote its degree.

\begin{Lemma}  The differential $(\pi_\lambda d_i)_{i \in \Z}$ defined above satisfies the Leibniz rule and if $f^\bullet: M^\bullet \to N^\bullet$ is a morphism (of degree zero) of DG-modules, then $F(f^\bullet)=(\pi_\lambda f_i)_{i \in \Z}: F(M^\bullet) \to F(N^\bullet)$ is a morphism of DG-modules.
\end{Lemma}

\begin{proof}
Let $a \in \Gamma(Q_G, W_G): Gx \to Gy$ be an arrow of degree $-2 \le j \le 0$ and $z = (z_g)_{g \in G} \in (\pi_\lambda M_i) e_{Gx} = \bigoplus_{g \in G}M_i e_{gx}$. We may assume that $a = Gb$ for some arrow $b:x \to y$ of degree $j$ in $\Gamma(Q,W)$. We have
$(\pi_\lambda d_i) ((z_g)_{g \in G}) = (d_i(z_g))_{g \in G}$. Also, $|gb|=|a|$ for all $g \in G$. Therefore,

\begin{eqnarray*}(\pi_\lambda d_{i+j})(za) &=& (\pi_\lambda d_{i+j})((z_g(gb))_{g \in G})\\& =& (d_{i+j}(z_g(gb)))_{g \in G} \\ & = & (d_i(z_g)(gb)+ (-1)^{|a|}z_g\mathfrak{d}_j(gb))_{g \in G} \\ & = & (d_i(z_g)gb)_{g \in G} + (-1)^{|a|}(z_g\mathfrak{d}_j(gb))_{g \in G}\\ & = & \pi_\lambda d_i(z)a + (-1)^{|a|}z\mathfrak{d}_{G,j}(a)\end{eqnarray*}
which shows that the differential satisfies the Leibniz rule. Now, let $f: M^\bullet \to N^\bullet$ be a morphism of DG-modules. We have

\begin{eqnarray*}\pi_\lambda f_{i+j}(za) & = & \pi_\lambda f_{i+j}((z_g(gb))_{g \in G})\\ & = & (f_{i+j}(z_g(gb)))_{g \in G})\\
& = & (f_{i}(z_g)(gb))_{g \in G}\\ & = &(f_{i}(z_g))_{g \in G}a\\ & = & (\pi_\lambda f_i)(z)a,  \end{eqnarray*}
which shows that $(\pi_\lambda f_i)_{i \in \Z}$ induces a morphism $F f: F M^\bullet \to F N^\bullet$ of DG-modules.
\end{proof}

Observe finally that $F$ is additive and $F (\Gamma(Q,W))$ lies in the additive hull of $\Gamma(Q_G, W_G)$ so that $F$ is a well-defined functor at the level of the perfect derived categories. Consider now the functor $\pi^\lambda: {\rm Rep}(Q_G) \to {\rm Rep}(Q)$ which is right adjoint to $\pi_\lambda$. For $M \in \Rep(Q)$, we have $\pi^\lambda \pi_\lambda  M = \oplus_{g \in G} gM$. We first need to extend $\pi^\lambda$ to a functor $\bar F: {\rm per}\Gamma(Q_G,W_G) \to {\rm per}\Gamma(Q,W)$. Let $M^\bullet = (M_i)_{i \in \Z}$ be a DG-module in ${\rm per}\Gamma(Q_G,W_G)$ with differential $(d_i)_{i \in \Z}$. We define $\bar F M^\bullet$ to be the complex $(\pi^\lambda M_i)_{i \in \Z}$ with differential $(\pi^\lambda d_i)_{i\in \Z}$. One can check that this defines a DG-module in $\mathcal{H}(\Gamma(Q,W))$. One also needs to define $\bar F$ on morphisms on the natural way: if $f^\bullet = (f_i)_{i\in\Z}: M^\bullet \to N^\bullet$ is a morphism of DG-module, then $(\pi^\lambda f_i)_{i \in \Z}$ is a morphism of DG-$\Gamma(Q, W)$-modules. One can check that $\bar F$ defines a functor from per$\Gamma(Q_G, W_G)$ to $\mathcal{H}(\Gamma(Q,W))$.

\begin{Lemma} \label{Adjoint}We have an adjoint pair $(F, \bar F)$. Moreover, for $M^\bullet \in {\rm per}\Gamma(Q,W)$, we have $\bar F F M^\bullet \cong \oplus_{g \in G} gM^\bullet$. In particular, since $G$ is finite, the functor $\bar F$ is from ${\rm per}\Gamma(Q_G,W_G)$ to ${\rm per}\Gamma(Q,W)$.
\end{Lemma}

\begin{proof}
This follows from the analogous properties for the functors $\pi_\lambda, \pi^\lambda$.
\end{proof}

\subsection{Cluster categories}\label{sect cluster cat}
The cluster category $\C(Q,W)$ of the quiver with potential $(Q,W)$ is defined in \cite{Amiot} as follows.
$$\C(Q,W) = {\rm per}\Gamma(Q,W)/{\rm f.d.}\Gamma(Q,W).$$
In this short subsection, we will study the category
$$\C(Q_G,W_G) = {\rm per}\Gamma(Q_G,W_G)/{\rm f.d.}\Gamma(Q_G,W_G).$$ Observe that
f.d.$\Gamma(Q,W)$ is clearly sent to f.d.$\Gamma(Q_G,W_G)$ by $F$. Therefore, the exact functor $F$ induces a functor
$$F: \C(Q,W) \to \C(Q_G,W_G).$$
This is an exact functor of triangulated categories. In general, this functor is neither full nor dense. We have the following.

\begin{Prop} \label{prop 4.1} The functor $F: {\rm per}\Gamma(Q,W) \to {\rm per}\Gamma(Q_G,W_G)$ is a $G$-precovering. It induces a $G$-precovering $F: \C(Q,W) \to \C(Q_G, W_G)$.
\end{Prop}

\begin{proof} The first part of the proof is an adaptation of Asashiba's proof \cite[Theorem 4.3 and 4.4]{Asa}. Let $M^\bullet, N^\bullet \in {\rm per}\Gamma(Q,W)$. Since $G$ is finite, we have a functorial isomorphism
$$\oplus_{g \in G}\Hom_{{\rm per}\Gamma(Q,W)}(M^\bullet, gN^\bullet) \cong \Hom_{{\rm per}\Gamma(Q,W)}(M^\bullet, \oplus_{g \in G}gN^\bullet).$$
The latter is functorially isomorphic to
$$\Hom_{{\rm per}\Gamma(Q,W)}(M^\bullet, \bar F F N^\bullet)$$
which, by the adjunction property, is functorially isomorphic to
$$\Hom_{{\rm per}\Gamma(Q_G,W_G)}(F M^\bullet, F N^\bullet).$$
Similarly, we have a functorial isomorphism $$\oplus_{g \in G}\Hom_{{\rm per}\Gamma(Q,W)}(gM^\bullet, N^\bullet) \cong \Hom_{{\rm per}\Gamma(Q_G,W_G)}(F M^\bullet, F N^\bullet).$$ This shows the first part of the lemma.

For the second part, we only need to observe that the functorial isomorphism
$$\oplus_{g \in G}\Hom_{{\rm per}\Gamma(Q,W)}(M^\bullet, gN^\bullet) \cong \Hom_{{\rm per}\Gamma(Q_G,W_G)}(F M^\bullet, F N^\bullet)$$
induces a functorial isomorphism
$$\oplus_{g \in G}\Hom_{\C(Q,W)}(M^\bullet, gN^\bullet) \cong \Hom_{\C(Q_G,W_G)}(F M^\bullet, F N^\bullet).$$
Similarly, we get a functorial isomorphism
$$\oplus_{g \in G}\Hom_{\C(Q,W)}(gM^\bullet, N^\bullet) \cong \Hom_{\C(Q_G,W_G)}(F M^\bullet, F N^\bullet). \qedhere$$

\end{proof}

When $(Q,W)$ is Jacobi-finite, since we have a $G$-covering $J(Q,W) \to J(Q_G, W_G)$, the pair $(Q_G, W_G)$ is also Jacobi-finite, so that by \cite[Theorem 3.5]{Amiot} again, $\C(Q_G,W_G)$ is a $2$-Calabi-Yau triangulated Hom-finite Krull-Schmidt $k$-category. The category $\C(Q_G,W_G)$ is then called the \emph{cluster category} of $(Q_G, W_G)$. Note that $Q_G$ may have loops and $2$-cycles. As a consequence, the potential $W_G$ need not be non-degenerate, even when $W$ is.

\subsection{Cluster-tilting objects and mutations}\label{sect cto}

In this subsection, we assume that $(Q,W)$ is Jacobi-finite and we let $G$ be an admissible group of automorphisms of $(Q,W)$. In particular, both $\C(Q,W), \C(Q_G, W_G)$ are $2$-Calabi-Yau triangulated Hom-finite Krull-Schmidt $k$-category. Let $T$ be a basic cluster-tilting object in $\C(Q,W)$. Equivalently, $\Hom_{\C(Q,W)}(T, T[1])=0$ and $T$ has exactly $n$ non-isomorphic direct summands, where $n = |Q_0|$. We call such a $T$ a \emph{$G$-cluster-tilting object} if $gT \cong T$ for all $g \in G$.
Clearly, the projective module $\Gamma(Q,W)$ is a $G$-cluster-tilting object.
If $U$ is an indecomposable direct summand of $T$ and $T$ is $G$-cluster-tilting, then for $g \in G$, we have that $gU$ is isomorphic to a direct summand of $T$. We will denote by $\overline{U}$ the direct sum of all the non-isomorphic such $gU$ and by $T_U$ the rigid object $T/\overline{U}$.

We recall some notions from Iyama-Yoshino; see \cite{IY}. Let $\mathcal{D}$ be an additive subcategory of $\C(Q,W)$ which is closed under taking direct summands and such that for $D_1, D_2 \in \mathcal{D}$, we have $\Hom(D_1, D_2[1])=0$. Assume also that $\mathcal{D}$ is functorially finite in $\C(Q,W)$. Let $\mathcal{X}$ be an additive subcategory of $\C(Q,W)$ which is closed under taking direct summands, contains $\mathcal{D}$, and is such that for $D \in \mathcal{D}$ and $X \in \mathcal{X}$, we have $\Hom(D,X[1])=0$. Given an object $X \in \mathcal{X}$, take a left $\mathcal{D}$-approximation $X \to D'$ and consider a triangle
$$X \stackrel{f}{\to} D' \to C_{X,f} \to X[1].$$
Then $\Hom(C_{X,f}, D[1])=0$ for all $D \in \mathcal{D}$ and $C_{X,f}$ is nonzero if $X$ is not in $\mathcal{D}$. Consider the additive subcategory $\mathcal{Y}$ of $\C(Q,W)$ generated by all such $C_{X,f}$. Clearly, $\mathcal{Y}$ contains $\mathcal{D}$ (the approximations above are not necessarily minimal) and for $Y \in \mathcal{Y}, D \in \mathcal{D}$, we have $\Hom(Y, D[1])=0$. By Proposition 2.1(1) in \cite{IY}, the category $\mathcal{Y}$ is closed under direct summands. Following the terminology in \cite{IY}, the pair $(\mathcal{X},\mathcal{Y})$ is called a $\mathcal{D}$-\emph{mutation pair}. It follows from Proposition 5.1 in \cite{IY} that $\mathcal{X}$ is a cluster-tilting subcategory if and only if so is $\mathcal{Y}$.

\medskip

As an application, we consider the following. Let $T$ be a $G$-cluster-tilting object in $\C(Q,W)$ and $U$ an indecomposable direct summand of $T$. Let $\mathcal{D}$ be the additive subcategory generated by the indecomposable direct summands of $T_U$ and let $\mathcal{X}$ be the one generated by the indecomposable direct summands of $T$. Clearly, $\mathcal{D}, \mathcal{X}$ are as above. Let $f_U: U \to D_U$ be a minimal left $\mathcal{D}$-approximation of $U$ in $\C(Q,W)$ and let $C_U$ be the cone of $f_U$.
Since each $g \in G$ can be seen as an automorphism of $\C(Q,W)$, the triangle $$U \stackrel{f_U}{\to} D_U \to C_U \to U[1]$$ is sent to the triangle $$gU \stackrel{gf_U}{\to} gD_U \to gC_U \to gU[1]$$
as $(gU)[1] \cong g(U[1])$. Now, $gU \in \mathcal{X}, gD_U \in \mathcal{D}$ and $gf_U$ is a minimal left $\mathcal{D}$-approximation of $gU$, so $gC_U \cong C_{gU}$. Now, let $f_{\overline{U}}: \overline{U} \to D_{\overline{U}}$ be a minimal left $\mathcal{D}$-approximation of $\overline{U}$ in $\C(Q,W)$.

\begin{Lemma} We have $C_{\overline{U}} \cong \overline{C_U}$, where $\overline{C_U}$ is the direct sum of the non-isomorphic objects in $\{gC_U \mid g \in G\}$.
\end{Lemma}

\begin{proof}
It is easy to check that the direct sum of the $g f_U$ for $g \in G$ forms a minimal left $\mathcal{D}$-approximation of $\overline{U}$ in $\C(Q,W)$. Therefore, we just need to check that $gU \cong U$ if and only if $C_U \cong C_{gU}$. The necessity follows from the left-approximation property. For the sufficiency, we just need to observe that if we have a left $\mathcal{D}$-approximation $f_U$ of $U$ with the corresponding triangle
$$U \stackrel{f_U}{\to} D_U \stackrel{f_U'}{\to} C_U \to U[1],$$
then $f_U'$ is a right $\mathcal{D}$-approximation of $C_U$.
\end{proof}

Now, we can set $\mu(T,\overline{U}) = (T/\overline{U}) \oplus C_{\overline{U}}$ and by construction, $\mathcal{Y}$ is the additive subcategory generated by the indecomposable direct summands of $\mu(T,\overline{U})$. In particular, $\mathcal{Y}$ is a cluster-tilting subcategory, meaning that $\mu(T,\overline{U})$ is a cluster-tilting object. It is clear that $\mu(T,\overline{U})$ is also $G$-cluster-tilting. We denote by $\mathcal{D}_G$ the full additive subcategory in $\C(Q_G, W_G)$ generated by the indecomposable direct summands of $F(T_U)$ and by $D_G$ the basic object of $F(T_U)$.

\begin{Prop}
Assume that $(Q,W)$ is Jacobi-finite. Let $h: U \to D$ be a minimal left $\mathcal{D}$-approximation of $U$ in $\mathcal{C}(Q,W) $. Then $F h$ is a left $\mathcal{D}_G$-approximation of $F U$ in $\C(Q_G, W_G)$.
\end{Prop}

\begin{proof} Since $(Q,W)$ is Jacobi-finite, the cluster categories $\C(Q,W), \C(Q_G, W_G)$ are Hom-finite. Let $D' \in \mathcal{D}$ be arbitrary, so that $F D'$ is arbitrary in $\mathcal{D}_G$. Since $F$ is a $G$-precovering, for each $g \in G$, there exists a natural isomorphism $\phi_g:  F\circ g \to F$ such that
$$(*): \quad \Phi_{U,D'}: \oplus_{g \in G}\Hom(U, gD') \to \Hom(F U, F D')$$
is given by $(f_g)_{g \in G} \to \sum_{g \in G}(\phi_g D') \circ F(f_g)$.
Let $f: F U \to F D'$ be any morphism. Since $(F, \bar F)$ is an adjoint pair and since $\bar F F D' \cong \oplus_{g \in G}gD'$, there is a morphism $\bar f \in \Hom_{\C(Q,W)}(U, \oplus_{g \in G}gD')$ corresponding to $f$ through the adjunction isomorphism $$\Hom_{\C(Q,W)}(U, \oplus_{g \in G}gD') \cong \Hom_{\C(Q_G,W_G)}(F U, F D').$$ Decompose $\bar f$ as $\bar f = (f_g)_{g \in G}$. Since $h$ is a left $\mathcal{D}$-approximation of $U$, there is a morphism $\eta: D \to \oplus_{g \in G}gD'$ such that $\bar f = \eta h$. Now, we have $(F f_g)_{g \in G} = F \eta F h$. Now, the diagram
$$\xymatrix{& F D \ar[dr]^{F \eta}&&&\\ F U \ar[ur]^{F h} \ar[rr]^{(F f_g)_{g \in G}} & & \oplus_{g \in G}F gD' \ar[rr]^{(\phi_g D')_{g \in G}} && F D'}$$
yields
$$f = (\phi_g D')_{g \in G}(F f_g)_{g \in G} = ((\phi_g D')_{g \in G}F \eta) F h$$
which shows that $F h$ is a left $\mathcal{D}_G$-approximation of $F U$ in $\C(Q_G, W_G)$.
\end{proof}

In the above setting, the process of replacing $\overline{U}$ in $T$ by the cone $C_{\overline{U}}$ of a minimal left $\mathcal{D}$-approximation $\overline{U} \to D_{\overline{U}}$ is called the (Iyama-Yoshino) \emph{orbit mutation} of $U$ in $T$. Note that $gC_{\overline{U}} \cong C_{\overline{U}}$ for all $g \in G$ and hence $C_{\overline{U}} \cong \overline{C_U}$, that is, the indecomposable direct summands of $C_{\overline{U}}$ are precisely the non-isomorphic objects of $\{gC_U \mid g \in G\}$.

\begin{Cor}
Let $T$ be a $G$-cluster-tilting object in $\C(Q,W)$.
Let $U$ be an indecomposable direct summand of $T$. The orbit mutation of $U$ in $T$ corresponds to the classical mutation of $F U$ inside the cluster-tilting object $F T$ of $\C(Q_G, W_G)$.
\end{Cor}

In the above corollary, the assumption that $T$ is $G$-cluster-tilting is necessary. In general, an indecomposable rigid object in $\C(Q,W)$ is not sent by $F$ to a rigid object in $\C(Q_G, W_G)$. In particular, a cluster-tilting object in $\C(Q,W)$ is not necessarily sent to a cluster-tilting object in $\C(Q_G, W_G)$ through $F$.

\begin{Cor}
Let $T$ be an object in $\C(Q,W)$ obtained by a sequence of orbit mutations of the rigid object $\Gamma(Q,W)$ seen as an object in $\C(Q,W)$. Then $F T$ is (not necessarily basic) cluster-tilting in $\C(Q_G,W_G)$.
\end{Cor}

\begin{Cor} \label{CovJacobian}
Let $T$ be a $G$-cluster-tilting object in $\C(Q,W)$.
Then there is a $G$-covering $\End_{\C(Q, W)}(T) \to \End_{\C(Q_G, W_G)}(F T)$.
\end{Cor}

\section{Surfaces and orbifolds}\label{sect 4}

Building on work of Fock and Goncharov \cite{FG1, FG2}, and of
Gekhtman, Shapiro and Vainshtein \cite{GSV},
Fomin, Shapiro and Thurston \cite{FST} associated a cluster algebra
to any {\it bordered surface with marked points}. Oriented Riemann orbifolds have been considered in \cite{FeShTu, ChSh, IwNa} in the context of cluster algebras. The triangulated orbifolds considered in \cite{FeShTu} is the geometric framework which allowed the same authors to complete the classification of skew-symmetrizable cluster algebras of finite type, in \cite{FeShTu2}. In \cite{ChSh}, the authors have also studied orbifolds, defined in a similar way, in the context of Teichm\"{u}ller theory. They have shown that the $\lambda$-lengths relation for the arcs in an orbifold behave like a three-term exchange relation of a generalized-cluster algebra, which is defined there.

\medskip

We fix the following notation.
\begin{itemize}
\item $S$ is a connected oriented Riemann surface with
(possibly empty)
boundary $\partial S$.
\item $M\subset S$ is a finite set of {\it marked points} with at least one marked point on each connected component of the boundary.
\end{itemize}
We will refer to the
pair $(S,M)$ simply as a \emph{surface}. A surface is called \emph{closed} if the boundary is empty. A connected component of $\partial S$ is called a \emph{boundary component}. Marked points in the interior of $S$ are called \emph{punctures}.

\medskip

An {orbifold} is a surface   with  additional data. Each puncture $b$ comes with a positive integer $m_b$ attached to it, called its \emph{isotropy}, and there is also a finite set of points $\mathcal{O}$ on $S \backslash (\partial S \cup M)$ called \emph{orbifold points}. More precisely, an \emph{orbifold} is a triple $(S,M,\mathcal{O})$ together with a function $m:M \to \Z_{\ge 1}$ such that $m_b:=m(b)=1$ whenever $b \in \partial S$. A puncture $b$ with isotropy $m_b$ will be called an $m_b$-\emph{puncture} and a $1$-\emph{puncture} is often called an \emph{ordinary puncture}.

\medskip

For technical reasons, when $\mathcal{O}$ is empty, we require that $(S,M)$ is not
a sphere with $1, 2$ or $3$ punctures;
a monogon with $0$ or $1$ puncture;
or a bigon or triangle without punctures.

\medskip

An orbifold point is denoted by a cross $\times$ in the surface, a marked point with isotropy one is denoted by a dot {\tiny$^\bullet$} while a puncture with isotropy greater than one is denoted by $\otimes$.

\subsection{Arcs and triangulations}

An \emph{arc} $\gamma$ in $(S,M)$ is a curve in $S$, considered up
to isotopy,
 such that
\begin{itemize}
\item[(a)] the endpoints of $\gamma$ are in $M$;
\item[(b)] $\gamma$ is disjoint from $\mathcal{O}$ and, except for the endpoints, $\gamma$ is disjoint from $M$ and
  from $\partial S$,
\item[(c1)] $\gamma$ does not cut out an unpunctured monogon, unless there is exactly one orbifold point in the monogon;
\item[(c2)] $\gamma$ does not cut out an unpunctured bigon;
\item [(d)] $\gamma$ does not cross itself, except that its endpoints may coincide.
\end{itemize}
If $\gamma$ is an arc with endpoints $a,b$, we will often indicate this by $\gamma: a - b$ or by $\gamma: b - a$. Curves that connect two marked points and lie entirely on the boundary of $S$ without passing
through a third marked point are called \emph{boundary segments}. By (c1) and (c2), boundary segments are not arcs.
A \emph{closed loop} is a closed curve in $S$ which is disjoint from the boundary of $S$.

\smallskip

For any two arcs $\gamma,\gamma'$ in $S$, define
\[e(\gamma,\gamma') = \min\{
\textup{number of crossings of
 $\alpha$ and $\alpha'\mid \alpha\simeq\gamma,\alpha'\simeq\gamma'$}\},
\]
 where $\alpha$
and $\alpha'$ range over all arcs isotopic to
$\gamma$ and $\gamma'$, respectively.
We say that arcs $\gamma$ and $\gamma'$ are \emph{compatible} if $e(\gamma,\gamma')=0$.

An \emph{ideal triangulation} is a maximal collection of
pairwise compatible arcs (together with all boundary segments).
The arcs of a
triangulation cut the surface into \emph{ideal triangles}.
Triangles that have exactly two distinct sides are called \emph{self-folded} triangles.  Note that a self-folded triangle consists of
a loop $\ell$, together with an arc $r$ to an enclosed puncture which we call a
\emph{radius}. If $m$ denotes the isotropy of the puncture inside the self-folded triangle, then the triangle is called $m$-\emph{self-folded}. A triangle that has only one arc has to be a loop enclosing exactly one orbifold point. Such a triangle is called an \emph{orbifold triangle}. A triangle that is neither self-folded nor an orbifold triangle is called a \emph{standard triangle}. A triangle is called \emph{internal} if no edge of the triangle is a boundary segment. A self-folded or orbifold triangle is always internal.
Examples of ideal triangulations are given in Figure \ref{figtriangulations}.

\begin{figure}
\begin{center}
\scriptsize\scalebox{1.2}{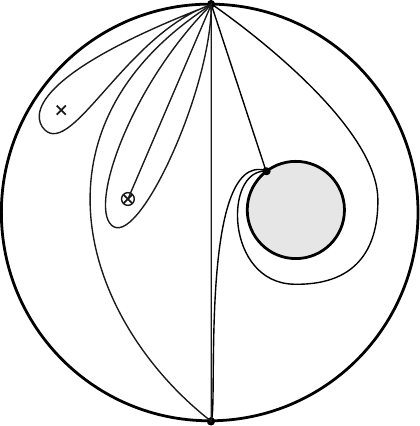}
\caption{An ideal triangulations of an annulus with one $m$-puncture and one orbifold point. The arc $6$ is the loop of an $m$-self-folded triangle whose radius is the arc 1. The arc 8 is the loop of an orbifold triangle}
\label{figtriangulations}
\end{center}
\end{figure}

The following is well known when $\mathcal{O}=\emptyset$.

\begin{Lemma}\label{LemmaRank}
\label{Ideal-Tri} The number of arcs in an ideal triangulation is exactly
\[n=6g+3b+3p + 2x +c-6,\]
where $g$ is the
genus of $S$, $b$ is the number of boundary components, $p$ is the number of punctures, $x$ is the number of orbifold points and $c=|M|-p$ is the
number of marked points on the boundary of $S$.
The number $n$ is called the \emph{rank} of $(S,M)$.
\end{Lemma}

\begin{proof} Consider the surface $(S,M')$ obtained by taking $M' = M \cup \mathcal{O}$. To get a triangulation of the ordinary surface $(S,M')$, we only need to add an arc for each point in $\mathcal{O}$. Therefore, $n+x=6g+3b+3(p + x) +c-6$, which gives the wanted expression for $n$.
\end{proof}

Ideal triangulations are connected to each other by sequences of
{\it flips}.  Each flip replaces a single arc $\gamma$
in $T$ by a unique new arc $\gamma' \neq \gamma$
such that
 \[T'=(T\setminus\{\zg\})\cup\{\zg'\}\]
is a triangulation.

\subsection{Tagged arcs}
 Note that an arc $\gamma$ that lies inside a self-folded triangle
in $T$ cannot be flipped.
In order to rectify this problem, the authors of \cite{FST}
were led to introduce the slightly more general notion
of {\it tagged arcs}. We adapt the notion for triangulations of orbifolds.

A {\it tagged arc} is obtained by taking an arc that does not
cut out a once-punctured monogon
and marking (``tagging")
each of its ends in one of two ways, {\it plain} or {\it notched},
so that the following conditions are satisfied:
\begin{itemize}
\item an endpoint lying on the boundary of $S$ must be tagged plain
\item both ends of a loop must be tagged in the same way.
\end{itemize}
Thus there are four ways to tag an arc between two distinct punctures and there are two ways to tag a loop at a puncture; see Figure \ref{taggedarcs}.
The notching is indicated by a bow tie.

\begin{figure}[h]
\begin{center}
 {\tiny \scalebox{1.15}{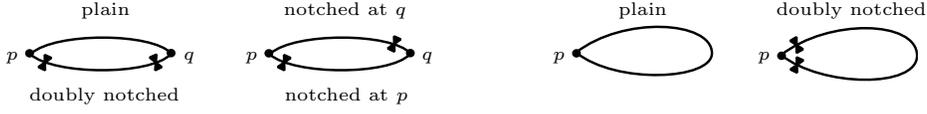}}
 \caption{Four ways to tag an arc between two punctures (left); two ways to tag a loop at a puncture (right)}
 \label{taggedarcs}
\end{center}
\end{figure}

\medskip

One can represent an ordinary arc $\beta$ by
a tagged arc $\iota(\beta)$ as follows.  If $\beta$
does not cut out a once-punctured monogon, then $\iota(\beta)$
is simply $\beta$ with both ends tagged plain.
Otherwise, $\beta$ is a loop based at some marked point $q$
and cutting out
a punctured monogon with the sole puncture $p$ inside it.
Let $\alpha$ be the unique arc connecting $p$ and $q$ and compatible
with $\beta$.  Then $\iota(\beta)$
is obtained by tagging $\alpha$ plain at $q$ and notched at $p$.

Tagged arcs $\alpha$ and
$\beta$ are called {\it compatible} if and only if the following
properties hold:
\begin{itemize}
\item the arcs $\alpha^0$ and $\beta^0$ obtained from
   $\alpha$ and $\beta$ by forgetting the taggings are compatible;
\item if $\alpha^0=\beta^0$ then at least one end of $\alpha$
  must be tagged in the same way as the corresponding end of $\beta$;
\item if $\alpha^0\neq \beta^0$ but they share an endpoint $a$,
 then the ends of $\alpha$ and $\beta$ connecting to $a$ must be tagged in the
same way.
\end{itemize}

A maximal collection
of pairwise compatible tagged arcs is called a {\it tagged triangulation}. Assume that $T$ is a tagged triangulation of $(S,M)$. We define a triangulation $\tau(T)$ without tags as follows. As a first case, assume that there is a puncture $b$ having two arcs $\alpha, \beta$ of $T$ connected to $b$ such that $\alpha$ is tagged plain at $b$ while $\beta$ is tagged notched at $b$. Then $\alpha^0 = \beta^0 : a - b$ and $\alpha, \beta$ are tagged the same way at $a$. Moreover, there is no other arcs having $b$ as endpoint. In this case, let $\tau(\alpha)$ be the arc $\alpha^0$ and $\tau(\beta)$ be the loop $a - a$ enclosing the puncture $b$ and tagged plain. If $\gamma$ is a tagged arc not as in the latter case, we let $\tau(\gamma) = \gamma^0$. It is easy to check that $\tau(T):=\{\tau(\gamma) \mid \gamma \in T\}$ is an ideal triangulation of $(S,M)$. Also, if $T$ is an ideal triangulation, then $\tau(\iota(T))=T$.

\subsection{Quivers and cluster categories}

In this subsection, $(S,M)$ is an ordinary surface, that is, $\mathcal{O}=\emptyset$. Given an ideal triangulation $T=\{\tau_1,\tau_2,\ldots,\tau_n\}$, the associated quiver $Q_T$ introduced in \cite{FST} can be defined as follows.
The vertices of $Q_T$ are in bijection with the arcs of $T$, and we denote the vertex of $Q_T$ corresponding to the arc $\tau_i$ simply by $i$. The arrows of $Q_T$ are defined as follows.
For any triangle $\Delta$ in $T$ which is not self-folded, we add an arrow $i\to j$
whenever
\begin{itemize}
\item[(a)] $\tau_i$ and $\tau_j$ are sides of
  $\Delta$ with  $\tau_j$ following $\tau_i$  in the
  clockwise order;
\item[(b)] $\tau_j$ is a radius in a self-folded triangle enclosed by a loop $\tau_\ell$, and $\tau_i$ and $\tau_\ell$ are sides of
  $\Delta$ with  $\tau_\ell$ following $\tau_i$  in the
clockwise order;
\item[(c)] $\tau_i$ is a radius in a self-folded triangle enclosed by a loop $\tau_\ell$, and $\tau_\ell$ and $\tau_j$ are sides of
  $\Delta$ with  $\tau_j$ following $\tau_\ell$  in the
clockwise order;
\item[(d)] $\tau_i, \tau_j$ are radii of self folded triangles with respective loops $\tau_{\ell}, \tau_{m}$ where $\tau_\ell, \tau_m$ are sides of
  $\Delta$ with $\tau_\ell$ following $\tau_m$ in the
clockwise order;
\end{itemize}
Then we remove all 2-cycles. If $T$ is tagged, then the quiver $Q_T$ of $T$ coincides with the quiver $Q_{\tau(T)}$ of the ideal triangulation $\tau(T)$.

\medskip

One can attach a cluster category, defined by a quiver with potential, to any triangulation $T$ of the ordinary surface $(S,M)$; see \cite{Amiot}. Let us recall the main ingredients of this construction. We let $W_T$ denote a potential in $\widehat{kQ_T}$. An example of a potential is the \emph{canonical potential} (or \emph{Labardini potential}) attached to $T$; see \cite{Labardini}. In case where there is no self-folded triangle in $T$, this potential $W_{T, {\rm ca}}$ is a sum of cycles, where a given cycle in $W_{T, {\rm ca}}$ is either a cycle of length $3$ corresponding to an internal triangle of $T$ or else is a cycle corresponding to surrounding once a puncture.
 In particular, the number of terms in $W_{T, {\rm ca}}$ is the number of internal triangles in $T$ plus the number of punctures in $M$. The Labardini potential can also be defined in the cases where $T$ has self-folded triangles (see \cite{Labardini}), but the definition is slightly more involved.

\medskip

Recall from Section \ref{sect cluster cat} that to the pair $(Q_T, W_T)$, one can attach the cluster category $\C(Q_T, W_T)$. In this category, one can perform mutations at any summand of a cluster-tilting object, regardless of the local properties of the quiver of that cluster-tilting object. Since we are mainly working with cluster categories, we will generally not assume that the potential $W_T$ is non-degenerate. Let us just mention the following fact.

\begin{Prop}
\cite{Labardini2}
Let $S$ be a surface with non-empty boundary. Then $W_{T,\textup{ca}}$ is non-degenerate. Moreover, for every mutation $\mu_a$, the potential $\mu_a W_{T,\textup{ca}}$ is right equivalent to the potential $W_{\mu_a( T),\textup{ca}}$. In particular, there is an isomorphism of Jacobian algebras
$J( \mu_a(Q_T, W_{T,{\textup{ca}}}))
 \cong J(Q_{\mu_a(T)}, W_{\mu_a(T), {\rm ca}})$.
\end{Prop}

\subsection{Group actions on triangulations}

Now, fix a tagged triangulation $T$ of $(S, M)$. For us, a \emph{homeomorphism} of $(S,M)$ is always an orientation-preserving homeomorphism of $S$ that maps $M$ to $M$. Two homeomorphisms $\varphi_1, \varphi_2$ of $(S,M)$ are \emph{isotopic} if their actions on $M$ coincide and if there is an isotopy $h: S \times [0,1] \to S$ such that $h(-,0)=\varphi_1$, $h(-,1)=\varphi_2$ and for $t \in (0,1)$, $h(-,t)$ has the same action on $M$ as $\varphi_1$. Following \cite{ASS}, we consider $\textup{MCG}(S,M)$ the mapping class group of $(S,M)$. The elements of $\textup{MCG}(S,M)$ are the homeomorphisms of $(S,M)$ up to the above-defined isotopy relation. This is a group under composition. We define $\textup{MCG}(S,M,T)$ to be the subgroup of $\textup{MCG}(S,M)$
of those elements $g$ that map $\tau(T)$ to $\tau(T)$ and preserve the tagging of arcs in the following way. If $\alpha: a - b \in T$, we require that the tagged arc $g\alpha: ga - gb$ is such that $\alpha, g\alpha$ are tagged the same way at $a, ga$, respectively; and  $\alpha, g\alpha$ are tagged the same way at $b, gb$, respectively. Since $T$ is finite, the group $\textup{MCG}(S,M,T)$ is always finite. Indeed, any element of $\textup{MCG}(S,M,T)$ fixing each arc of $T$ and each marked point of $M$ has to be the identity element.
An element in $\textup{MCG}(S,M,T)$ is called a \emph{$T$-automorphism} of $(S,M)$.

\medskip

An \emph{admissible group} is a  group $G$ of $T$-automorphisms that acts \emph{freely} on $T$, that is, if $g \in G$ fixes an arc of $T$ (but not necessarily its endpoints), then $g$ is the identity automorphism.
From now on, let $G$ be an admissible group. Let $b$ be a triangle from $T$, an arc of $T$, a boundary segment or a marked point of $M$.
The subgroup $G_b$ of all  $g \in G$ that map the set $b$ to itself will be called the \emph{isotropy group of $b$}. We sometimes say that $b$ has trivial isotropy if $G_b$ is trivial, that is, if $g(b)=b$ then $g$ is the identity.
Notice that the isotropy group of an arc is always trivial, since $G$ is admissible.

\begin{Lemma}\label{lem 22} Let $G$ be a non-trivial admissible group of $T$-automorphisms of $(S,M)$ and $b$ be a triangle, marked point or boundary segment with  non-trivial isotropy group $G_b$.
\begin{enumerate}[$(1)$]
    \item If $b$ is a triangle, then $b$ is not self-folded and $G_b$ has order $3$.
    \item Otherwise, $b$ is a puncture and $G_b$ is a cyclic group whose order is a divisor of the number of arcs incident to $b$, and of the number of loops incident to $b$.
    \end{enumerate}
\end{Lemma}

\begin{proof}
We first claim that $b$ cannot be a marked point on the boundary or a boundary segment. Assume otherwise. Assume further that $B$ is a boundary component of $S$ with $b \in B$.
Let $g \in G_b$ be non-trivial. Then $g$ maps $B$ to $B$ and hence $g$ permutes the marked points of $B$. Suppose first that $b$ is a boundary segment in $B$. Then $b$ is the bounding curve of a unique  standard triangle $\delta$ of $T$. Since $b$ is fixed by $g$, we see that $\delta$ is fixed by $g$. Now, $\delta$ has at least one internal arc. If it has exactly one, say $a$, then $g$ fixes $a$, a contradiction to $G$ being admissible. If $\delta$ has two internal arcs, then $g$ permutes these internal arcs. But then, $g$ reverses the orientation of $\delta$, a contradiction. Suppose now that $b$ is a marked point of $B$. If $b$ is the unique marked point of $B$, then we take $c$ to be the unique boundary segment of $B$ and the above argument applies. Otherwise, let $c_1, c_2$ be the two boundary segments attached to $b$. If each $c_i$ is fixed by $g$, then the above argument applies. Otherwise, $g$ permutes $c_1, c_2$ but then reverse the orientation on $B$, a contradiction.

Suppose now that $b$ is a puncture. If $b$ lies inside a self-folded triangle, then clearly, the loop of that self-folded triangle is fixed by $G_b$, a contradiction. Let $c_0, \ldots, c_{m-1}$ be the arcs of $T$ incident to $b$ in cyclic order around $b$. Any $h \in G_b$ induces a permutation $\sigma_h$ of $c_0, \ldots, c_{m-1}$. Since $G_b$ preserves the orientation of $S$, every $h$ is uniquely determined by its action on $c_0$: if $h(c_0) = c_i$, then $h(c_j) = c_{j+i}$, where the indices are taken modulo $m$. Take $g_0 \in G_b$ with $g_0(c_0) = c_i$ where $i > 0$ is minimal. We claim that $G_b$ is the cyclic group generated by $g_0$. Let $h \in G_b$ and assume that $h(c_0) = c_j$ where $j \ge i$. Then for $t \in \Z$, the element $g_0^{-t}h$ is such that $g_0^{-t}h(c_0) = c_{j-ti}$. There exists $t \ge 1$ such that $0 \le j-ti < i$. By minimality of $i$, we have $j=ti$ and $g_0^{-t}h$ fixes $c_0$, showing that $h = g_0^t$. This shows that $G_b$ is cyclic. Since $g_0^m = 1$, we see that the order of $G_b$ divides $m$. Since an element of $G$ sends a loop of $T$ to a loop of $T$ and a non-loop of $T$ to a non-loop of $T$, the second statement of the proposition follows.

The only case left is when $b$ is a triangle from $T$. As observed above, no arc of $b$ is a boundary segment. Also, $b$ is not self-folded, as otherwise, its loop would be fixed by $G_b$, which is impossible. Every non-identity element $h$ in $G_b$ induces a rotation of order $3$ of $b$. Using the fact that $G$ acts freely on $T$, we see that $G_b$ is generated by any non-identity element $h \in G_b$ and hence, $G_b$ has order three.
\end{proof}

\medskip

Since $G$ is finite and admissible, it acts properly discontinuously on $(S,M)$ and the orbit space $S_G:=S/G$ is a Riemann surface. Moreover, since $G$ consists only of orientation preserving homeomorphisms, $S_G$ is actually oriented (with the induced orientation from $S$) with finitely many isolated singular points. We refer the reader to W. Thurston's notes \cite[Chapter 13]{WThurston} for results in this direction and also for more details concerning these orbit spaces.

One way to study the orbit space is through a fundamental domain. For each internal and standard triangle having a non-trivial isotropy group, consider the unique point in its interior which is a singular point. Let $\overline{\mathcal{O}}$ denote the set of all of these points. Consider the oriented Riemann surface $(S,M')$ where $M' = M \cup \overline{\mathcal{O}}$ and let $T^{\mathcal{O}}$ be the triangulation obtained from $\tau(T)$ by adding three arcs to each marked point of $\overline{\mathcal{O}} \subseteq M'$. It is clear that $G$ is an admissible group of $T^{\mathcal{O}}$-automorphisms of $(S,M')$ and where each triangle from $T^{\mathcal{O}}$ now has a trivial isotropy group. We construct a collection $\mathfrak{C}$ of triangles from $T^{\mathcal{O}}$ as follows. Start with a triangle $\delta_1$ of $T^{\mathcal{O}}$ and set $\mathfrak{C}_1 = \{\delta_1\}$. In general, suppose that we have constructed $\mathfrak{C}_t$ for $t \ge 1$. If for any given $\delta \in \mathfrak{C}_t$, all triangles adjacent to $\delta$ are in the $G$-orbit of some triangle of $\mathfrak{C}_t$, then we set $\mathfrak{C}:= \mathfrak{C}_t$. Otherwise, there is a triangle $\delta_{t+1}$ adjacent to a triangle from $\mathfrak{C}_t$ which is not in the $G$-orbit of any triangle of $\mathfrak{C}_t$. We set $\mathfrak{C}_{t+1}:=\mathfrak{C}_t \cup \{\delta_{t+1}\}$. Continuing this way, since $T^{\mathcal{O}}$ has finitely many triangles, we get a final collection $\mathfrak{C} := \mathfrak{C}_s$ for some $s$ of triangles from $T^{\mathcal{O}}$ having the property that for any $\delta \in \mathfrak{C}$, all triangles adjacent to $\delta$ are in the $G$-orbit of some triangle of $\mathfrak{C}$. We denote by $\mathfrak{F}$ the union of all triangles from $\mathfrak{C}$. Observe that $\mathfrak{F}$ is connected by construction. Note also that $\mathfrak{F}$ forms an oriented Riemann surface, and the arcs and boundary segments of $T^{\mathcal{O}}$ bounding a triangle of $\mathfrak{C}$ induce a triangulation $T_{\mathfrak{F}}$ of $\mathfrak{F}$.

\begin{Lemma}
The surface $\mathfrak{F}$ is (the closure of) a fundamental domain for $S$ under the action of $G$.
\end{Lemma}

\begin{proof}
We use the notations in the above paragraph. Let $\Delta = \Delta_0$ be a triangle in $T^{\mathcal{O}}$. Since $S$ is connected, there is a sequence of triangles $\Delta_0, \Delta_1, \ldots, \Delta_m$ of $T^{\mathcal{O}}$ such that $\Delta_i$ shares an edge with $\Delta_{i+1}$ for all $0 \le i \le m-1$ and $\Delta_m$ lies in $\mathfrak{C}$. By the defining property of $\mathfrak{C}$, $\Delta_{m-1}$ is in the $G$-orbit of a triangle in $\mathfrak{C}$. By induction, we get that all $\Delta_i$ are in the $G$-orbit of a triangle in $\mathfrak{C}$. This, combined with the definition of $\mathfrak{C}$, ensures that $\mathfrak{C}$ contains exactly one triangle from each $G$-orbit of the triangles in $T^{\mathcal{O}}$. This completes the proof of the lemma.
\end{proof}

From this, it is easy to see that $S = \bigcup_{g \in G}g\mathfrak{F}$ and $s|G|$ is the number of triangles of $T^{\mathcal{O}}$. Now, the orbit space $S_G$ can be thought of as $\mathfrak{F}$ in which arcs in the same $G$-orbit are being identified; see Figure \ref{fig:hexagon2}. We will always make this identification from now on.

It follows from Lemma \ref{lem 22} that the points in $S$ with non-trivial isotropy are either punctures or points inside standard internal triangles of $\tau(T)$, and the collection of those latter points was denoted $\overline{\mathcal{O}}$.
The orbits $M_G:=M/G$ of $M$ then correspond to marked points in $S_G$. Lemma \ref{lem 22} together with Lemma \ref{lemmaboundary} below guarantees that the punctures of $S_G$ correspond to the orbits of the punctures of $M$; and the marked points on the boundary of $S_G$ correspond to the orbits of the marked points on $\partial S$. For each puncture $b$ of $S$, let $m_b$ be the order of its isotropy group $G_b$. This defines a
function $m: M_G \to \Z_{\ge 1}$ that associate to each puncture $b$ the number $m_b$ and to each marked point on the boundary of $S_G$ the number $1$.
Note that the $G$-orbits $\mathcal{O}$ of the points in $\overline{\mathcal{O}}$ are disjoint from $\partial S_G, M_G$. Therefore, $(S_G, M_G, \mathcal{O})$ is an orbifold. An $m$-puncture in $M_G$ is called an \emph{ordinary puncture}, if $m=1$; and a \emph{$G$-puncture}, if $m>1$.

As noticed above, as surfaces (without taking into account the marked points), the orbit space $S_G$ can be identified with $\mathfrak{F}$. On the one hand, $S_G$ is equipped with two disjoint sets of points: the marked points $M_G$ together with the orbifold points ${\mathcal{O}}$. On the other hand, $\mathfrak{F}$ has \emph{no orbifold} point but rather marked points, which are $M_\mathfrak{F}:=M_G \cup \mathcal{O}$. Therefore, an arc in $\mathfrak{F}$ is understood to be an arc of the marked surface $(\mathfrak{F}, M_\mathfrak{F})$.

\medskip

The next lemma guarantees that the boundary components of $S_G$ are in correspondence with the orbits of the boundary components of $S$ under the action of $G$. In particular, there is no new boundary component in $S_G$. Of course, the fact that $G$ consists only of orientation-preserving homeomorphisms is crucial. For instance, if $S$ is the sphere with all punctures and arcs on the equator and $G = \Z_2$ is the group generated by the reflection along the equator, then $S_G$ is a disk and hence, a new boundary component is created.

\begin{Lemma} \label{lemmaboundary}The $G$-orbits of the boundary components of $S$ correspond to the boundary components of $S_G$.
\end{Lemma}

\begin{proof}
We use the above notation and we identify $S_G$ with $\mathfrak{F}$. It is not hard to check that any boundary segment $\alpha$ which is an edge of a triangle of $\mathfrak{C}$ is also a boundary segment in the orbit space $\mathfrak{F}$. Let $\mathfrak{T}$ be the collection of all arcs of $T^{\mathcal{O}}$ that are edges of a triangle of $\mathfrak{C}$. It is sufficient to prove that no element of $\mathfrak{T}$ is a boundary segment in the orbit space $\mathfrak{F}$. Note that the arc $d$ in the right picture of Figure \ref{fig:hexagon2} is not a boundary segment of $\mathfrak{F}$ because of the gluing.  Let $\beta \in \mathfrak{T}$. If $\beta$ is the radius of a self-folded triangle, then the interior of the corresponding self-folded triangle is entirely contained in $\mathfrak{F}$, and hence $\beta$ is not a boundary segment in $\mathfrak{F}$. Assume now that $\beta \in \mathfrak{T}$ is not the radius of a self-folded triangle. Let $\Delta_1, \Delta_2$ be the two triangles in $T^{\mathcal{O}}$ adjacent to $\beta$. We may assume that only one of $\Delta_1, \Delta_2$, say $\Delta_1$, lies in $\mathfrak{C}$. By definition of $\mathfrak{C}$, there is $1 \ne g$ in $G$ such that $g\Delta_2$ lies in $\mathfrak{C}$. Since $G$ acts freely on arcs, we have $g\beta \ne \beta$. Therefore, in $\mathfrak{F}$, $\beta$ and $g\beta$ are glued along two distinct triangles $\Delta_1$ and $g\Delta_2$ of $\mathfrak{C}$. Hence, $\beta$ is not a boundary segment in $\mathfrak{F}$.
\end{proof}

A tagged arc or curve of $(S,M')$ is \emph{$\overline{\mathcal{O}}$-avoiding} if it is not incident to a point in $\overline{\mathcal{O}}$.
We say that the $G$-orbit of an arc $\zg$ in $(S,M')$ consists of compatible arcs if for all $g\in G$ the arcs $g\zg$ and $\zg$ are compatible. For example the orbit of the arc $\zg$ in the center of Figure \ref{fig:hexagon2} does not consist of compatible arcs.

\begin{Prop} \label{PropBijec}
There is a bijection between the $G$-orbits consisting of compatible tagged arcs in $(S,M)$ with the tagged arcs of $(S_G, M_G, \mathcal{O})$. Moreover, this bijection induces a bijection between $G$-stable tagged triangulations of $(S,M)$ and tagged triangulations of $(S_G, M_G, \mathcal{O})$.
\end{Prop}

\begin{proof}
We just prove the bijections for ordinary arcs and ordinary triangulations. The cases where there is a tagging can be checked without difficulties, and using the fact that $G$ respects the tagging. We start with the first bijection and we use the above notation, so we have a triangulation $T^{\mathcal{O}}$ of $(S,M')$ which is $G$-stable and we identify the orbit space with $\mathfrak{F}$. First, note that an arc $\alpha$ of $(S,M')$ can be thought of as a curve $\varphi(\alpha)$ of $(\mathfrak{F}, M_{\mathfrak{F}})$ between marked points in $M_{\mathfrak{F}} = M_G \cup \mathcal{O}$. This is done by folding $\alpha$ along the copies $g\mathfrak{F}$, $g \in G$, of $\mathfrak{F}$; see Figure \ref{fig:hexagon2}. Alternatively, $\varphi(\alpha)$ is obtained by restricting the curves $\{g \alpha \mid g \in G\}$ to $\mathfrak{F}$. It is clear that $\varphi(g\alpha) = \varphi(\alpha)$ for all $g \in G$. If $\alpha$ is $\overline{\mathcal{O}}$-avoiding, then $\varphi(\alpha)$ is a curve of $(\mathfrak{F}, M_{\mathfrak{F}})$ between marked points in $M_G$. Observe that if $\alpha, \beta$ are isotopic arcs of $(S,M')$, then $\varphi(\alpha), \varphi(\beta)$ are isotopic in $(\mathfrak{F}, M_{\mathfrak{F}})$, since when deforming two curves in
$(S,M')$, we cannot cross a point in $\overline{\mathcal{O}}$, as the latter are all included in $M'$. Conversely, take any curve $c$ in $(\mathfrak{F}, M_{\mathfrak{F}})$ between marked points of $M_G$. Its fiber is a 
$G$-orbit of $\overline{\mathcal{O}}$-avoiding curves in $(S,M')$. Any deformation of $c$ in $(\mathfrak{F}, M_{\mathfrak{F}})$ corresponds to deformations of the curves in the $G$-orbit. In particular, any curve without self-intersection in $(\mathfrak{F}, M_{\mathfrak{F}})$ between marked points of $M_G$ corresponds to a $G$-orbit of compatible arcs. This shows that the correspondence $\{g \alpha \mid g \in G\} \mapsto \varphi(\alpha)$ gives the first bijection.

Now, consider a partial triangulation $V$ of $(S,M)$ that is $G$-stable. In particular, all arcs are $\overline{\mathcal{O}}$-avoiding. By the alternative description of $\varphi$ above, it is clear that $\varphi$ sends $V$ to a partial triangulation of $(\mathfrak{F}, M_{\mathfrak{F}})$ between marked points in $M_G$. In other words, it is a partial triangulation of $(S_G, M_G, \mathcal{O})$. Now, a $G$-stable triangulation of $(S,M)$ with $n$ arcs has to be sent to a partial triangulation of $(S_G, M_G, \mathcal{O})$ with $n/|G|$ arcs. This has to be a triangulation of $(S_G, M_G, \mathcal{O})$. Conversely, given any triangulation $V'$ of $(S_G, M_G, \mathcal{O})$ with $m = n/|G|$ arcs, its fiber will consist of $|G|m = n$ curves between marked points in $M$. As already argued, no two curves are isotopic and they are pairwise compatible. Therefore, the fiber is indeed a triangulation of $(S,M)$.
\end{proof}

\begin{Cor} \label{G-flip}
There exists a notion of mutation of any tagged arc in a tagged triangulation of $(S_G, M_G, \mathcal{O})$ and, under the above bijection, this corresponds to changing a $G$-orbit $A$ of tagged arcs of a $G$-stable tagged triangulation $V$ of $(S,M)$ to another $G$-orbit $B$ of tagged arcs where $(V\backslash A)\cup B$ is a $G$-stable tagged triangulation of $(S,M)$.
\end{Cor}

Changing a $G$-orbit as in the previous corollary will be called an orbit mutation.

\begin{figure}
\begin{center}
\huge\scalebox{0.4}{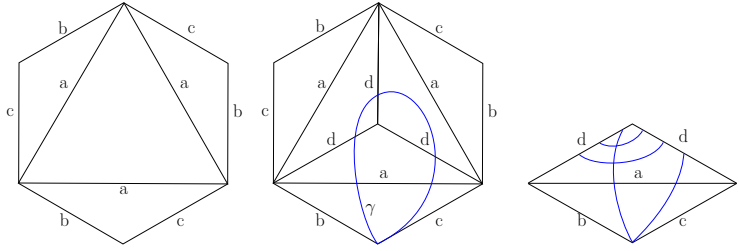}
 \caption{Left: the hexagon $(S,M)$ having a triangulation $T$ with the group of order $3$ acting by rotation. Middle: an arc $\gamma$ in $(S,M')$ with triangulation $T^{\mathcal{O}}$. Right: a fundamental domain $\mathfrak{F}$ with the corresponding curve $\varphi(\gamma)$.}
\label{fig:hexagon2}
\end{center}
\end{figure}

\begin{Exam} \label{ExampleOctahedron}
Consider the regular octahedron, seen as the sphere $S$ with $|M|=6$ punctures and the corresponding triangulation $T$ (without self-folded triangles and all arcs plain). This is a well known fact that there are $24$ orientation-preserving symmetries of the regular octahedron, so $24$ possible $T$-automorphisms of $(S,M)$. Among these symmetries, $6$ are not admissible since they fix two arcs. Take the subgroup $H$ of $G$ generated by rotations of order $2$ around punctures and the rotations of order $3$. Color the facets of the octahedron in two colors, black or white, in such a way that if two triangles share an arc, then they are colored in a different way, see the left picture in Figure \ref{fig octahedron}. The subgroup $H$ can be described as the orientation-preserving symmetries that preserve the colors of the triangles.
\begin{figure}
\begin{center}
\scalebox{0.6}{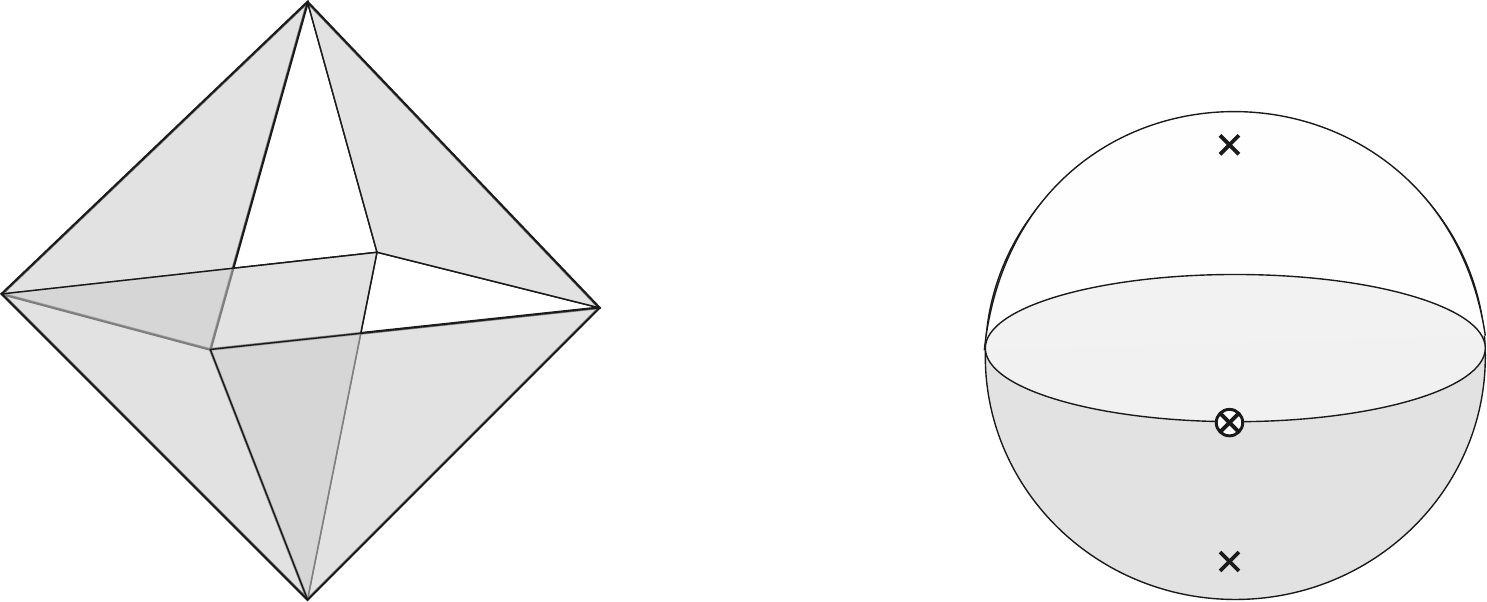}
 \caption{The octahedron of Example \ref{ExampleOctahedron} on the left, and its orbifold, a sphere with one 2-puncture and two orbifold points, on the right.}
\label{fig octahedron}
\end{center}
\end{figure}
This subgroup does not contain the rotations of order $4$ and is admissible. It is clearly non-abelian and every element has order $1,2$ or $3$. Therefore, $H$ is isomorphic to the alternating group $A_4$. Observe that every triangle and every puncture has non-trivial isotropy. Notice that there are two orbits of triangles for the action of $H$, only one orbit of arcs, and only one orbit of punctures for $H$.

The orbifold is a sphere with one $2$-puncture corresponding to the orbit of the punctures of the octahedron, and two orbifold points corresponding to the points fixed by $H$ other than the punctures, see the right picture in Figure \ref{fig octahedron}. One of these points is the center of a white triangle and the other the center of a black triangle. The white triangles become the northern hemisphere while the black triangles become the southern hemisphere. The two triangles of $T_G$ are orbifold triangles. A mutation of the unique arc in $T_G$ would be a change of tagging at both ends of $\gamma$, while the corresponding orbit mutation of the unique orbit of arcs in $T$ would be a simultaneous change of taggings at all ends of arcs.
\end{Exam}

\begin{Exam}\label{ex icosi}
Consider the modified icosidodecahedron illustrated in Figure \ref{fig icosi}.
\begin{figure}
\begin{center}
\Large\scalebox{0.8}{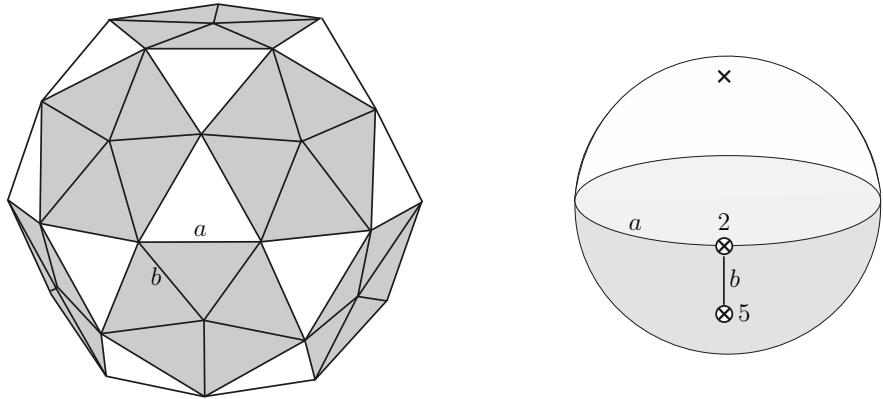}
\caption{The icosidodecahedron of Example \ref{ex icosi} on the left and its orbifold on the right.}
\label{fig icosi}
\end{center}
\end{figure}
 There are $60$ black triangles, $20$ white triangles, $42$ punctures and $120$ arcs. Consider the orientation preserving symmetries generated by rotations of order three at the center of the white triangles and rotations of order five at the center of the black pentagons (build from five black triangles). This generates the subgroup (of order $60$) of all orientation-preserving symmetries preserving the colors of the triangles. We get two orbits of triangles (black and white), two orbits of punctures (a center of a black pentagon and a vertex of a white triangle) and two orbits of arcs (a side of a white triangle denoted $a$ and a common side of two black triangles denoted $b$). Observe that the orbifold $(S_G, M_G, \mathcal{O})$ has one orbifold point, one $2$-puncture and one $5$-puncture, see the right picture in Figure \ref{fig icosi}.

Observe that in the original triangulation $T_1$ of $S$, there is a unique way to change the arcs in $Gb$ to get another triangulation $T_2$ such that the new arcs will form another single $G$-orbit. The same observation holds for the arcs in $Ga$. This orbit mutation at $Gb$ just produces a change of tags at the punctures corresponding to the centers of the black pentagons. The orbit mutation of $Ga$ is illustrated in the left picture in Figure \ref{fig icosi2}. The corresponding mutation in the orbifold is shown on the right of the figure.
\begin{figure}
\begin{center}
\Large\scalebox{0.8}{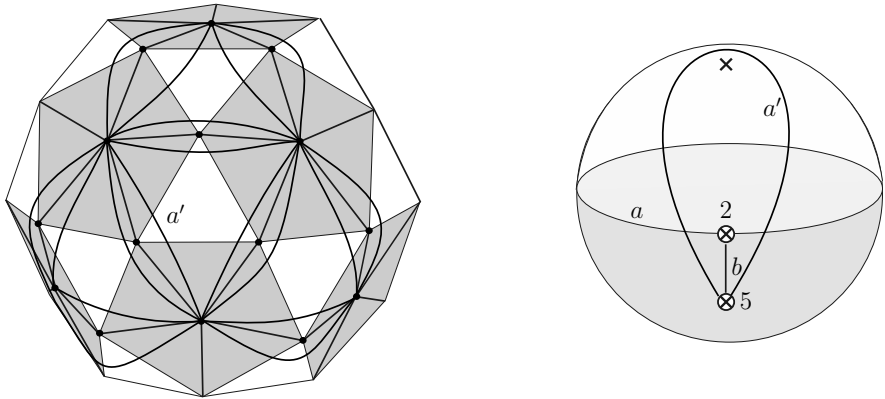}
\caption{The icosidodecahedron of Example \ref{ex icosi} after mutation in the orbit $Ga$ on the left and the corresponding mutation on the orbifold producing the loop $a'$ on the right.}
\label{fig icosi2}
\end{center}
\end{figure}
\end{Exam}

The following will be useful and is well known in case $\mathcal{O} = \emptyset$.

\begin{Prop} Let $(S,M,\mathcal{O})$ be an orbifold with a tagged triangulation $T$. Let $m$ be the number of marked points, $t$ the number of triangles of $\tau(T)$ (including the self-folded triangles and the orbifold triangles) and $a$ the number of arcs. Then the number $\chi(S,M,\mathcal{O}) = m + t - a$ does not depend on the triangulation and equals to $2-2g-b$ where $g$ is the genus of $S$ and $b$ is the number of boundary components in $S$.
\end{Prop}

\begin{proof}
For each orbifold point $x$ in $S$, there is a unique loop $a_x - a_x$ enclosing $x$, where $a_x$ denotes the base point of the loop. Take $M' = M \cup \mathcal{O}$ and consider the triangulation $T^{\mathcal{O}}$ of $(S,M')$ obtained from $T$ by adding, for each orbifold point $x$, the arc $a_x - x$. Clearly, the number $m + t - a$ is the same for $(S,M)$ and $(S,M')$. Since $(S,M')$ is an ordinary surface, this common number is $2-2g-b$.
\end{proof}

The number $\chi(S,M,\mathcal{O})$ of the proposition is called the \emph{Euler characteristic} of the orbifold $(S,M,\mathcal{O})$.

\section{The exchange polynomials for the orbit space}\label{TechnicalSection}
In this section, we determine the exchange polynomials for the generalized cluster algebra structure on the orbit space.

Let $T$ be a tagged triangulation of a surface $(S,M)$, and let $G$ be a non-trivial admissible group of $T$-automorphisms. Denote by $\cala$ the cluster algebra with trivial coefficients associated to $(S,M)$ with initial seed corresponding to the triangulation $T$. Recall that to each tagged arc $\tau$ in $(S,M)$, one can associate a Laurent polynomial, also denoted $\tau$, in $\mathbb{Z}[\mathbf{x}^{\pm 1}]$. This polynomial may not be a cluster variable. It will be convenient to label the arcs of $T$, and hence the initial cluster variables, according to the $G$-orbits as follows. Let $s$ be the number of orbits and let
\[T=\{\tau_{11},\ldots,\tau_{1r}\}\sqcup\{\tau_{21},\ldots,\tau_{2r}\}\sqcup\cdots\sqcup\{\tau_{s1},\ldots,\tau_{sr}\}\]
be the decomposition of $T$ into its $G$-orbits.
Denoting by $x_{ij}$ the cluster variable of $\tau_{ij}$, we obtain the following decomposition of the initial cluster
\[\mathbf{x}=(x_{11},\ldots,x_{1r},
x_{21},\ldots,x_{2r},\ldots
x_{s1},\ldots,x_{sr}).\]
The cluster algebra $\cala$ is a $\mathbb{Z}$-subalgebra of the field $\calf=\mathbb{Q}(\mathbf{x})$ of rational functions in the $x_{ij}$.

For the orbifold $(S_G,M_G,\calo)$, we have the tagged triangulation $T_G=\{\tau_1,\tau_2,\ldots,\tau_s\}$, the cluster $\mathbf{y}= (y_1,y_2,\ldots,y_s)$, and we will work in the field $\calf_G=\mathbb{Q}(\mathbf{y})$, where the arc $\tau_i$ and the variable $y_i$ represent the orbit of arcs $\tau_{i1},\ldots,\tau_{ir_i}$, respectively the orbit of variables $x_{i1},\ldots,x_{ir}$.
In order to determine the (generalized) cluster algebra structure of the orbifold, we must define mutations, which then will allow us to construct generators (generalized cluster variables) starting from the initial seed $\mathbf{y}=(y_1,y_2,\ldots,y_s)$. To this end, we will construct exchange polynomials $p_{y_i}\in\mathbb{Z}[y_1,y_2,\ldots,y_s]$.

In the cluster algebra $\cala$, let $x_{ij}'$ denote the cluster variable obtained by mutation the initial cluster in direction $ij$. Let $p_{x_{ij}}\in \mathbb{Z}[\mathbf{x}\setminus\{x_{ij}\}]$ denote the exchange polynomial of this mutation. Thus
\[x_{ij}x_{ij}'=p_{x_{ij}}.\]
Let $F\colon\mathbb{Z}[\mathbf{x}^{\pm 1}]\to
\mathbb{Z}[\mathbf{y}^{\pm 1}]$
be the ring homomorphism given by $F(x_{ij})=y_i$ and $F(a)=a$, for $a\in\mathbb{Z}$. Thus $F(p_{x_{ij}})$ is the polynomial in $\mathbb{Z}[y_1,y_2,\ldots,y_s]$ obtained by replacing the variables $x_{i1},\ldots,x_{ir}$ of each orbit by the variable $y_i$.

\begin{Remark}\label{rem 4.1}
 Since $G$ is an admissible group of $T$-automorphisms,  we have, for all $j,k\in\{1,\ldots,r\}$,
 $$F(p_{x_{ij}})=F(p_{x_{ik}}).$$\end{Remark}

Determining the exchange polynomials $p_{y_i}$ for the orbifold is not straightforward in general. In the simplest case, when $p_{x_{ij}}$ does not involve any variable of the same orbit $x_{i1},\ldots,x_{ir}$, we have $p_{y_i}=F(p_{x_{ij}})$.
However, if $p_{x_{ij}}$ does involve one of the variables $x_{i1},\ldots,x_{ir}$, the situation is more complex. In this case, it follows from Corollary \ref{G-flip} that there is a unique other tagged triangulation $T'=\left(T\setminus\{\tau_{i1},\ldots,\tau_{ir}\}\right)\cup\{\tau''_{i1},\ldots,\tau''_{ir}\}
$
such that $G$ is also an admissible group of $T'$-automorphisms and $\{\tau''_{i1},\ldots,\tau''_{ir}\}$ is a $G$-orbit. We will see that these tagged arcs correspond to Laurent polynomials $x''_{i1},\ldots,x''_{ir}$ in the initial cluster $\mathbf{x}$, such that $F(x''_{ij}) = F(x''_{ik})$ for $1 \le j,k \le r$. Therefore, it will make sense to define
\[p_{y_i}=F(x_{ij}x''_{ik}),\]
where $j,k\in\{1,\ldots,r\}$ are arbitrary.
We will see that $p_{y_i}$ actually is a polynomial in $\mathbb{Z}[y_1,y_2,\ldots,y_s]$. However this polynomial is not always a binomial and it also may have integer coefficients greater than 2. As a consequence, we do not obtain an honest cluster algebra structure for the orbifold but a \emph{generalized} cluster algebra structure.

\medskip
Consider the ordinary triangulation $\tau(T)$. We now fix $\gamma$ an arc of our tagged triangulation $T$ and denote its endpoints by $a$ and $b$. For simplicity, we identify $\gamma$ with $x_1$ and $G\gamma$ with $y_1$. If $\tau(\gamma)$ is not a radius of a self-folded triangle, then $\tau(\gamma)$ is a diagonal in a quadrilateral $\mathcal{Q}$ from $\tau(T)$ formed by the edges $\tau(\mu): a - c, \tau(\nu): c - b$, $\tau(\alpha): a - d, \tau(\beta): d - b$, which could be arcs or boundary segments, and may possibly be identified, see Figure \ref{fig:local}.
We adopt the convention that whenever $\tau(\gamma)$ is the loop of a self-folded triangle, then $\mu = \nu$ and $\tau(\mu) = \tau(\nu)$ is the radius of this self-folded triangle (and then $a = b$) and $\tau(\alpha), \tau(\beta)$ are the other arcs (or boundary segments) adjacent to $\tau(\gamma)$ in $\tau(T)$.
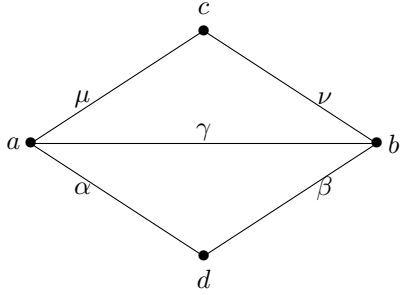
\begin{figure}
  \centering
  \begin{tikzpicture}[xscale=2.30,yscale=1.5]
\node at (0, 0.1) {$\gamma$};
\node at (-0.7, 0.4) {$\mu$};
\node at (-0.7, -0.4) {$\alpha$};
\node at (0.7, 0.4) {$\nu$};
\node at (0.7, -0.4) {$\beta$};
\node at (-1.1, 0) {$a$};
\node at (1.1, 0) {$b$};
\node at (0,1.2) {$c$};
\node at (0,-1.2) {$d$};
\node at (0,1) {$\bullet$};
\node at (0,-1) {$\bullet$};
\node at (-1,0) {$\bullet$};
\node at (1,0) {$\bullet$};
\draw (-1.00,0) -- (0,1);
\draw (1.00,0) -- (0,1);
\draw (-1.00,0) -- (0,-1);
\draw (1.00,0) -- (0,-1);
\draw (-1.00,0) -- (1,0);
\end{tikzpicture}
\caption{The quadrilateral $\mathcal{Q}$}
\label{fig:local}
\end{figure}

The triangle formed by the arcs $\tau(\gamma), \tau(\mu), \tau(\nu)$ is denoted $\Delta_1$ while the triangle formed by arcs $\tau(\gamma), \tau(\alpha), \tau(\beta)$ is denoted $\Delta_2$. As noted, $\Delta_1, \Delta_2$ are distinct triangles. Note that if one of $\Delta_1,\Delta_2$ is self-folded, then the other is not self-folded. Otherwise, the surface $S$ consists exactly of $\Delta_1, \Delta_2$ and therefore has to be the sphere with $3$ punctures, which is excluded.

If $\epsilon\in T$ is such that $\tau(\epsilon)$ is a loop (or radius, respectively) of a self-folded triangle in $\tau(T)$, then we denote by $\bar \epsilon$ the arc in $T$ with $\tau(\bar \epsilon)$ the radius (or loop, respectively) of that triangle. In particular $\epsilon^0 = \bar\epsilon^0$. If $\tau(\epsilon)$ is not a loop or radius of a self-folded triangle, then we set $\bar \epsilon = 1$, by convention. In case $\tau(\gamma)$ is the radius of a self-folded triangle, the quadrilateral $\mathcal{Q}$ does not make sense. We will rather consider the corresponding quadrilateral for $\tau(\bar \gamma)$ and still denote it by $\mathcal{Q}$.

\medskip

Any of $\{\mu, \nu, \alpha, \beta\}$ that is a boundary segment is identified with $1$ in $\cala$. For a marked point $e$ in $M$, we let $m_e$ denote its isotropy. By Lemma \ref{lem 22}, we have $m_e=1$ unless $e$ is a puncture, in which case $m_e \ge 1$. As before, we have an induced tagged triangulation $T_G$ in $S_G$ and identify it with the set of $G$-orbits of tagged arcs of $T$ in $(S, M)$. As seen previously, $T_G$ is a tagged triangulation of $(S_G, M_G, \mathcal{O})$. The triangulation $\tau(T_G)$ corresponds to the $G$-orbits of arcs in $\tau(T)$.

\medskip

To simplify the notions, we identify $\tau(T)$ with $T$ and $\tau(T_G)$ with $T_G$ in the following sense. Whenever we work in a geometric framework, we always refer to the geometric version $\tau(T), \tau(T_G)$ of $T, T_G$, respectively. Whenever we consider elements in $\mathcal{F}$ or in $\mathcal{F}_G$, we always mean the tagged triangulations $T$ or $T_G$. Therefore, we drop the $\tau$.

\begin{Lemma} \label{LemmaExchange}The exchange polynomial $p_{\gamma}$ for $\gamma$ is  $\frac{\mu\bar\mu\beta\bar\beta + \nu\bar\nu\alpha\bar\alpha}{{\rm gcd}(\mu\bar\mu\beta\bar\beta, \nu\bar\nu\alpha\bar\alpha)}$. The denominator is non-trivial in the following cases.
\begin{enumerate}
    \item The arc $\gamma$ is a loop or a radius of a self-folded triangle in $T$.
\item We have $\mu = \alpha$ and $\mathcal{Q}$ is a once-punctured bigon, that is, there are exactly two arcs of $T$ at $a$ and they are not loops.
\item We have $\nu = \beta$ and $\mathcal{Q}$ is a once-punctured bigon, that is, there are exactly two arcs of $T$ at $b$ and they are not loops.
\end{enumerate}
\end{Lemma}

\begin{proof}
First, $\gamma$ needs to be replaced by $\bar\gamma$ if $\gamma$ is a radius of a self-folded triangle. Indeed, it is well-known that the exchange polynomials for $\gamma$ or $\bar\gamma$ are the same, hence, we may assume already that if $\gamma$ is an arc of a self-folded triangle, then it is a loop. If the arc $\gamma$ is identified with one arc in $\{\alpha, \beta, \mu,\nu\}$, then one of the triangles $\Delta_1, \Delta_2$, say $\Delta_1$, is self-folded. In that case, $\gamma$ is identified with one of $\{\mu, \nu\}$. By our convention, we have an identification $\mu = \nu$. Therefore, we get $\gamma = \mu=\nu$, which yields a triangle having only one arc, a contradiction.

So $\gamma \not \in \{\alpha, \beta, \mu,\nu\}$. If the cardinality of $\{\alpha, \beta, \mu,\nu\}$ is $4$, then ${\rm gcd}(\mu\bar\mu\beta\bar\beta, \nu\bar\nu\alpha\bar\alpha)=1$ and $\mu\bar\mu\beta\bar\beta + \nu\bar\nu\alpha\bar\alpha$ is the usual Ptolemy relation taking radii of self-folded triangles into account. So we may assume that the cardinality of $\{\alpha, \beta, \mu,\nu\}$ is less than $4$. If $\Delta_1, \Delta_2$ are self-folded, then the surface $(S,M)$ is the sphere with three punctures and this is excluded. So assume, as a first case, that $\Delta_1$ is self-folded  but $\Delta_2$ is not, so that $\mu= \nu$ is the radius and $\gamma$ is the loop of $\Delta_1$, and we are in case (i) of the Lemma. In particular, $\alpha \ne \beta$. Also, $\alpha \ne \mu$ and $\beta \ne \mu$. Therefore, there is not other identification among $\alpha, \beta, \mu,\nu$. The expression $\frac{\mu\bar\mu\beta\bar\beta + \nu\bar\nu\alpha\bar\alpha}{{\rm gcd}(\mu\bar\mu\beta\bar\beta, \nu\bar\nu\alpha\bar\alpha)}$ becomes $\beta\bar\beta + \alpha\bar\alpha$. This is the known exchange polynomial for the loop $\gamma$ of a self-folded triangle. The case where $\Delta_2$ is self-folded  is similar.

Therefore, we may assume that none of $\Delta_1, \Delta_2$ is self-folded. This means that $\alpha \ne \beta$ and $\mu \ne \nu$, but not all four are distinct. As a first case, assume that $\alpha = \mu$. Using the orientability of $S$, the arcs $\alpha, \mu$ have to be identified in such a way that $c = d$. Observe that the triangles adjacent to $\mu$ are $\Delta_1, \Delta_2$. Consider a small oriented cycle $\sigma$ having $a$ as center and starting on $\gamma$ and going clockwise. Observe that $\sigma$ first traverses $\Delta_2$ and then, $\Delta_1$. With our identification of $\mu$ with $\za$, we see that $\sigma$ only crosses two ends of arcs. In particular, only $\alpha, \gamma$ have $a$ as endpoint and none of these arcs are loops.
Therefore, the arcs $\beta, \nu$ enclose a once-punctured bigon, and we are in case (ii) of the lemma. In this case, $\frac{\mu\bar\mu\beta\bar\beta + \nu\bar\nu\alpha\bar\alpha}{{\rm gcd}(\mu\bar\mu\beta\bar\beta, \nu\bar\nu\alpha\bar\alpha)}$ becomes $\beta\bar\beta + \nu\bar\nu$ which is the exchange polynomial for an arc $\gamma$ inside a once-punctured bigon.
The case (iii), where $\beta = \nu$, is similar.
We cannot have both $\za=\mu$ and $\zb=\nu$, since this would mean that $(S,M)$ is a sphere with 3 punctures, which is excluded. If $\za=\nu$ or $\zb=\mu$ then
 ${\rm gcd}(\mu\bar\mu\beta\bar\beta, \nu\bar\nu\alpha\bar\alpha) = 1$ and $\mu\bar\mu\beta\bar\beta + \nu\bar\nu\alpha\bar\alpha$ is the usual Ptolemy relation taking radii of self-folded triangles into account.
\end{proof}

According to the preceding result, a special attention has to be given to self-folded triangles and once-punctured bigons.
A self-folded triangle in $T_G$ around an $m$-puncture is called an \emph{$m$-self-folded} triangle.
\begin{Prop} The orbits of the self-folded triangles in $(S,M,T)$ corresponds bijectively to the $1$-self-folded triangles in $(S_G, M_G, \mathcal{O}, T_G)$.
\end{Prop}

\begin{proof}
Let $\sigma_1,\sigma_2$ be the arcs of a self-folded triangle in $(S,M,T)$ with $\sigma_1:a - a$ the loop and $\sigma_2: a - b$ the radius.
Lemma \ref{lem 22} implies that this self-folded triangle, and hence $b$, has a trivial isotropy group.
 Therefore, we see that $G\sigma_1,G\sigma_2$ is a $1$-self-folded triangle of $T_G$ in $(S_G, M_G, \mathcal{O})$. Conversely, assume that $G\rho_1, G\rho_2$ are the arcs in $T_G$ of a $1$-self-folded triangle in $(S_G, M_G, \mathcal{O})$ with $G\rho_1:Ga \to Ga$ the loop and $G\rho_2: Ga \to Gb$ the radius.
 Since $Gb$ is a 1-puncture in $S_G$, we see that $b$ in $S$ is a puncture with trivial isotropy group. Because $G$ is admissible, this implies that only one arc of $T$ is incident to $b$.
 This means that $b$ lies inside a self-folded triangle in $(S,M,T)$. This self-folded triangle corresponds to the self-folded triangle of $T_G$ given by $G\rho_1, G\rho_2$.
\end{proof}

A once-punctured bigon in $T_G$ containing
an $m$-puncture is called a \emph{once-punctured $m$-bigon.}

\begin{Prop} The orbits of the once-punctured bigons in $(S,M,T)$ correspond bijectively to the $2$-self-folded triangles and once-punctured $1$-bigons in $(S_G, M_G, \mathcal{O}, T_G)$.
\end{Prop}

\begin{proof} Consider a once-punctured bigon $Q$ in $(S,M,T)$ as shown in Figure \ref{fig bigon}.\begin{figure}
\begin{center}
\small\scalebox{1.2}{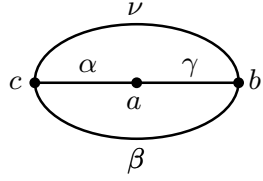}
\caption{A once-punctured bigon}
\label{fig bigon}
\end{center}
\end{figure}
Since the puncture $a$ is
 incident to precisely two arcs in $T$, its isotropy
 $m_a$ must be either 1 or 2. If $m_a=1$ then the 4 arcs of the bigon lie in 4 different $G$-orbits. Moreover, the puncture $a$ does not lie in the orbit of $b$ (or $c$), since there are at least 3 arcs incident to $b$ (and $c$). This shows that the orbit of $Q$ is a bigon in $S_G$.

 Assume now that $m_a = 2$. Then there is $1 \ne g \in G$ with $ga = a$. We must have $g\gamma = \za$ and $g\beta = \nu$. Hence, $Gb = Gc$ and as for the argument above, $Ga \ne Gb$. It follows that the orbit of $\mathcal{Q}$ is a $2$-self-folded triangle in $S_G$. The converse is clear.
\end{proof}

\medskip
We now define the exchange polynomials for $S_G$. We shall use the notation $p_{G,\zg}$ for the exchange polynomial of the variable associated to the $G$-orbit of $\zg$.
We need to distinguish several cases. In each case, we use the notation in Figure \ref{fig:local}.

\subsection{Case where $\gamma$ lies in a self-folded triangle or a once-punctured bigon}\label{sect 7.1}

Let $\gamma$ be the loop of the self-folded triangle, which we may assume to be $\Delta_1$. Then $\{\mu=\nu, \alpha,\beta,\gamma\}$ are four distinct arcs. Let $g \in G$. Observe that a self-folded triangle is sent to a self-folded triangle by $g$ and $g\gamma$ is a loop of a self-folded triangle. Since $\Delta_2$ is not self-folded, none of $G\mu, G\alpha, G\beta$ is equal to $G\gamma$.
Lemma \ref{LemmaExchange} implies that
$p_{\gamma} = \beta\bar\beta + \alpha\bar\alpha$, and since $p_{G,\zg} =F(p_\zg)$, we have
\[p_{G,\zg}=
 G\beta G\bar\beta + G\alpha G\bar\alpha.\]
This is either a sum of two distinct monomials or, if $G\alpha = G\beta$, a single monomial with coefficient $2$. If $\gamma$ is the radius of a self-folded triangle, then the exchange polynomials for $\gamma, \bar\gamma$ are the same.

\medskip

Assume now that $\mathcal{Q}$  is a once-punctured bigon, so we may assume $\alpha = \mu$. The exchange polynomial for $\gamma$ is $\beta\bar\beta + \nu\bar\nu$. Observe that $\beta \ne \nu$ as otherwise, $S$ is a sphere with three punctures. Also, since $\Delta_1, \Delta_2$ are not self-folded, we get that $\{\mu=\alpha, \nu,\beta,\gamma\}$ forms $4$ distinct arcs. If $m_a = 1$, then all $G\mu, G\nu, G\beta, G\gamma$ are distinct. Therefore, in this case,
\[p_{G,\gamma} = F(p_\gamma)= G\beta G\bar\beta + G\nu G\bar\nu.\]
 We get a sum of two distinct monomials. If $m_a = 2$, then we are still in the case where $p_{G, \gamma} = F(p_\gamma)$. Since $G\nu = G\beta$, we get \[p_{G,\gamma} = 2G\beta G\bar\beta.\]

\subsection{Case where $\gamma$ lies in the orbit of one of $\{\alpha, \beta, \mu, \nu\}$}\label{SubsectionIdentifiedOrbits}

Because of Section~\ref{sect 7.1}, we may assume that $\mathcal{Q}$ is not a bigon and none of $\Delta_1, \Delta_2$ are self-folded. By Lemma \ref{LemmaExchange},
the exchange polynomial for $\gamma$ is $\mu\bar\mu\beta\bar\beta + \nu\bar\nu\alpha\bar\alpha$.  We need the following lemma.

\begin{Lemma}\label{LemmaAll} If all arcs of $\mathcal{Q}$ lie in the same orbit, then all arcs in $T$ lie in the same orbit and $\partial S = \emptyset$.
\end{Lemma}

\begin{proof}
Assume that all arcs of $\mathcal{Q}$ lie in the same orbit. Assume to the contrary that $G\gamma \ne T$. Then there is a triangle $\Delta_3$ adjacent to a triangle in $G\Delta_1 \cup G\Delta_2$ having an arc $\epsilon$ not in $G\gamma$. We may assume that $\zD_3$ is adjacent to $\Delta_1$ or $\Delta_2$. Let $g \in G$ with $g\gamma = \mu$ and $g' \in G$ with $g'\gamma = \alpha$.

As a first case, assume that $ga = c$ and $g'a = d$. Then $g'\Delta_2 = \Delta_2$ and $g\Delta_1 = \Delta_1$, since $G$ is orientation-preserving.  By symmetry, we may assume that $\Delta_3$ is adjacent to $\Delta_1$.
However, each side of $\Delta_1$ is a side of a triangle in $G\Delta_2$.
Thus, $\Delta_3 \in G\Delta_1 \cup G\Delta_2$. But this mean that $\epsilon \in G\gamma$, a contradiction.

As a second case, assume that $ga = a$ and $g'a = d$. Then $g'\Delta_2 = \Delta_2$ and $g\Delta_2 = \Delta_1$. So again, we may assume that $\Delta_3$ is adjacent to $\Delta_1$ and we get the same contradiction. The case where $g'a=a$ is similar.
\end{proof}

We will assume now that the arcs of $T$ do not lie in a single orbit, when $\partial S = \emptyset$. This case is treated separately in Subsection \ref{sect 9}.

\begin{Lemma} \label{LemmaCases}
One and only one of the following situations occur.
\begin{enumerate}[$(1)$]
\item
There exists a non-trivial $g\in G$ such that $g\zD_1=\zD_1$. In this case, $G\gamma \ne G\alpha$ and $G\gamma \ne G\beta$.
\item
There exists a non-trivial $g\in G$ such that $g\zD_2=\zD_2$. In this case $G\gamma \ne G\mu$ and $G\gamma \ne G\nu$.
\item
There exists $g\in G$ such that $g\mu=\zg$ and $g\zg=\za$.  In this case, $g\nu=\zb$ and $G\zg\ne G\zb$ and $G\zg \ne G\nu$.
\item
There exists $g\in G$ such that $g\zb=\zg$ and $g\zg=\nu$.  In this case, $g\za=\mu$ and  $G\zg\ne G\za$ and $G\zg\ne G\mu$.
\end{enumerate}
\end{Lemma}

\begin{proof}
Assume first that we are in case (1). Thus $G\zg=G\mu=G\nu$. If there exists $h\in G$ such that $h\zg=\za$ then either $ha=d$ and then $h^2\zg=\zb$, or $ha=a$ and then $h\nu=\zb.$
In both cases, we get that all arcs of $\mathcal{Q}$ lie in the same orbit, and by Lemma \ref{LemmaAll} this contradicts our assumption. This shows that $G\zg\ne G\za$. Similarly, $G\zg\ne G\zb$. This proves the statement in (1).
The case (2) is proved by a similar argument.

Assume now we are in case (3). Then $g\zD_1=\zD_2$ and $g\nu=\zb$. If there exists $h\in G$ such that $h\zb=\zg$ then either $hb=a$ and then $h\zD_2=\zD_2$, or $hb=b$ and then
$h^2\zb=\nu$. In the former case, we are in case (2) which is impossible since $G\zg=G\mu$. In the latter case, all arcs of $\mathcal{Q}$ lie in the same orbit, and again Lemma \ref{LemmaAll} yields a contradiction to our assumption.
This proves that $G\zg\ne G\zb$. Similarly $G\zg \ne G\nu$. This proves (3), and (4) follows by a similar argument.

Since $\zg$ lies in one of the orbits of $\za,\zb,\mu,\nu$, the four cases of the lemma cover all possible situations. Clearly, the cases are mutually exclusive.
\end{proof}

The next two lemmas explain how to find the polynomial $p_{G,\zg}$ in the cases of Lemma \ref{LemmaCases}. Cases (1) and (2) are treated in Lemma \ref{LemmaChekovShapiro} while cases (3) and (4) are treated in Lemma \ref{LemmaPuncturedDisk}.

\begin{Lemma} \label{LemmaChekovShapiro} Let $T$ contain an unpunctured hexagon formed by the arcs
$\za_1,\ldots,\za_6$, $ \zg_1,\zg_2,\zg_3$ as in the left picture in Figure \ref{figlem10}.
Denote by $\Delta$ the triangle formed by $\zg_1, \zg_2,\zg_3$.
Suppose that there is a non-trivial $g\in G$ such that $g\zD=\zD$. Assume moreover that $\alpha_i \not \in G\zg_1$ for all $i$ and $G\alpha_i = G\alpha_j$ if $i \equiv j \mod 2$.
\begin{enumerate}[$(1)$]
\item If $\alpha_1 = \alpha_2$ or $\alpha_2 = \alpha_3$, then $p_{G, \zg_1} = 3G\alpha_1$.
    \item If $G\alpha_1 \ne G\alpha_2$, then $p_{G, \zg_1} = (G\alpha_1G\bar\alpha_1)^2 + G\alpha_1G\bar\alpha_1 G\alpha_2G\bar\alpha_2 + (G\alpha_2G\bar\alpha_2)^2$.
    \item If $G\alpha_1 = G\alpha_2$, then $p_{G,
    \zg_1} = 3(G\alpha_1G\bar\alpha_1)^2$.
\end{enumerate}
\end{Lemma}

\begin{figure}
\begin{center}
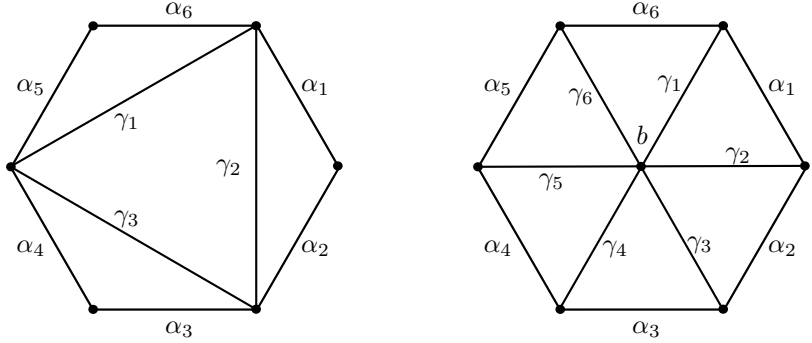
\end{center}
\caption{The triangulation of Lemma \ref{LemmaChekovShapiro} on the left and the triangulation of Lemma \ref{LemmaPuncturedDisk}, in the case where $m=6$, on the right.}\label{figlem10}
\end{figure}

\begin{proof}
Assume first that no arcs of $\{\alpha_1, \ldots, \alpha_6\}$ are identified. Let $T'=(T\setminus\{\zg_1,\zg_2,\zg_3\})\cup \{\zg_1'',\zg_2'',\zg_3''\}$ be the triangulation obtained by mutating in $\zg_1,\zg_2,\zg_3$ and then at $\zg_1'$, where $\zg_1'$ is the arc obtained by flipping $\zg_1$ at the first mutation. We get the following equations in the cluster algebra $\cala$.
$$\zg_1'' = \frac{\alpha_1\bar\alpha_1\alpha_3\bar\alpha_3\zg_2 + \alpha_2\bar\alpha_2\alpha_3\bar\alpha_3\zg_3 + \alpha_2\bar\alpha_2\alpha_4\bar\alpha_4\zg_1}{\zg_1\zg_2},$$ $$\zg_2'' = \frac{\alpha_1\bar\alpha_1\alpha_6\bar\alpha_6\zg_2 + \alpha_2\bar\alpha_2\alpha_6\bar\alpha_6\zg_3 + \alpha_1\bar\alpha_1\alpha_5\bar\alpha_5\zg_1}{\zg_1\zg_3},$$ $$\zg_3'' = \frac{\alpha_4\bar\alpha_4\alpha_6\bar\alpha_6\zg_2 + \alpha_3\bar\alpha_3\alpha_5\bar\alpha_5\zg_3 + \alpha_4\bar\alpha_4\alpha_5\bar\alpha_5\zg_1}{\zg_2\zg_3}\cdot$$
A straightforward check gives that $F(\gamma''_1)=F(\gamma''_2)=F(\gamma''_3)$ and $F(\zg_i''\zg_j)= (G\alpha_1G\bar\alpha_1)^2 + G\alpha_1G\bar\alpha_1G\alpha_2G\bar\alpha_2 +(G\alpha_2G\bar\alpha_2)^2$ for all $1 \le i,j \le 3$.
If $G\alpha_1 = G\alpha_2$, then $G\bar\alpha_1 = G\bar\alpha_2$ and we get the last case. It is not hard to check that if some arcs of $\{\alpha_1, \ldots, \alpha_6\}$ are identified, then we have two cases. Either $\alpha_{1}=\alpha_{2}, \alpha_3 = \alpha_4, \alpha_5 = \alpha_6$ and the left picture in Figure \ref{figlem10} contains three self-folded triangles. Otherwise, we have $\alpha_2 = \alpha_3, \alpha_4 = \alpha_5, \alpha_6 = \alpha_1$. In both cases, $(S,M)$ is the sphere with four punctures. These correspond to the cases in $(1)$ and are left to the reader, as the arguments are similar to the above arguments.
\end{proof}

\begin{Remark}
In the situation of Lemma \ref{LemmaChekovShapiro}, note that the orbit mutation of the arcs $\{g\gamma_1 \mid g \in G\}$ corresponds to rotating all triangles $g\Delta$, $g \in G$, about their respective centers by an angle of $\pi/3$.
\end{Remark}

\begin{Lemma} \label{LemmaPuncturedDisk}
Let $T$ contain a punctured polygon formed by the arcs
$\za_1,\ldots,\za_m$, $ \zg_1,\ldots,\zg_m$ as in the right picture in Figure \ref{figlem10}.
Let $b$ denote the puncture and assume that the isotropy group of $b$ is cyclic of order $m$ and $G\gamma_1 \ne G\alpha_1$.
\begin{enumerate}[$(1)$]
    \item  There exists a sequence of $2m-2$ mutations whose overall effect is a change of tag at the puncture $b$.
    \item We have $p_{G, \zg_1} = m G\alpha G\bar\alpha$.
\end{enumerate}
\end{Lemma}

\begin{proof} Observe that all $\gamma_i$ are in the same orbit and all $\alpha_i$ are in the same orbit. These two orbits are distinct.
Let $T'=(T\setminus\{\zg_1,\ldots,\zg_m\})\cup \{\zg_1'',\ldots,\zg_m''\}$ be the triangulation obtained by mutating in  $\gamma_1,\gamma_2, \ldots, \gamma_{m-1},\gamma_m, \gamma_{m-2}', \gamma_{m-3}', \ldots, \gamma_2', \gamma_1'$, where $\gamma_i'$ is the arc obtained after mutation at $\gamma_i$. In the cluster algebra $\cala$, we have the following identity
$$\gamma_{m-1}'' = \gamma_{m}\left(\frac{\alpha_1\bar\za_1}{\gamma_1\gamma_2} + \frac{\alpha_2\bar\za_2}{\gamma_2\gamma_3} + \cdots + \frac{\alpha_{m-1}\bar\za_{m-1}}{\gamma_{m-1}\gamma_m} + \frac{\alpha_m\bar\za_m}{\gamma_m\gamma_1}\right).$$
Observe that $F(\gamma_{m-1}''\gamma_i) = m G\alpha G\bar\alpha$ for all $1 \le i,j \le m$. By similar computations, we get arcs $\gamma_1'', \ldots, \gamma_m''$ and one can check that for $1 \le i,j \le m$, we have $F(\gamma''_i) = F(\gamma''_j)$ and $F(\gamma_{j}''\gamma_i) = m G\alpha G\bar\alpha$.  The arcs $\gamma_1'', \ldots, \gamma_m''$ clearly forms a $G$-orbit and $G$ is an admissible group of $T'$ automorphisms.
\end{proof}

\begin{Remark}
In the situation of Lemma \ref{LemmaPuncturedDisk}, note that the orbit mutation of the arcs $\{g\gamma_1 \mid g \in G\}$ corresponds to changing all tagging at $gb$, $g \in G$.
\end{Remark}

\subsection{Case of a single orbit}\label{sect 9}

Let $(S,M)$ be a surface with a tagged triangulation $T$ and assume that $G$ is an admissible group of $T$-automorphisms of $(S,M)$. In this section, we assume that all arcs of $T$ lie in the same orbit.

\begin{Lemma} If $\partial S \ne \emptyset$, then $(S,M,T)$ is one of the following surfaces illustrated in Figure \ref{figlem10}.
\begin{enumerate}[$(a)$]
    \item The disk with $6$ marked points on the boundary and one internal triangle, and $G$ is of order $3$.
\item  The once punctured disk where all arcs are connected to the puncture.\end{enumerate}

\end{Lemma}

\begin{proof} Let $C$ be a boundary component of $S$ and $m$ be the number of marked points on $C$. Let $\alpha_1: a_1 - a_2, \ldots, \alpha_{m-1}: a_{m-1} - a_m, \alpha_m : a_m - a_1$ be the boundary segments. Consider a triangle $\Delta$ having $\alpha_1$ as a side. If $\Delta$ is self-folded, then, since all arcs lie in the same orbit, we have $m=1$ and $S$ is the once-punctured disk with one marked point on the boundary. So assume that $\Delta$ is not self-folded.
Suppose first that the other two sides of $\zD$ are arcs and denote them by $\zb_1$ and $\zb_2$. Let $b$ be the common vertex of $\zb_1$ and $\zb_2$.
 Then there is $g \in G$ with $g \zb_1=\zb_2$ and such a $g$ sends $\alpha_1$ to a  boundary segment adjacent to $\alpha_1$ on $C$, say $\alpha_2$.
 Let $\zb_3=g\zb_2=g^2\zb_1$. Thus the triangles $\zD $ and $g\zD$ share one side $\zb_2$ and have two adjacent sides $\za_1,\za_2$ on $C$. Moreover, all three edges $\zb_1,\zb_2,\zb_3$ have a common vertex $b$. Repeating this argument, we obtain a sequence of triangles $\zD,g\zD,g^2\zD,\ldots , g^{m-1}\zD$ each of which contains exactly one boundary segment of $C$ and each contains two arcs from the boundary to the point $b$. Therefore, these triangles cover the entire surface and $b$ is a puncture. Thus
  we get a once-punctured disk and all arcs are connected to the puncture.

  Assume now that
  two sides $\za_1,\za_2$ of $\zD$ lie on the boundary and the third is an arc $\zg = \zg_1$.
   We claim that $m$ is even, that there are $m/2$ arcs $\gamma_i: a_i - a_{i+2}$ for all odd $i$ (where indices are taken modulo $m$), and that these arcs are all arcs having an endpoint on $C$. If there is an arc $\gamma'$ other than $\gamma$ having $a_3$ as endpoint, then there is $1 \ne g \in G$ with $g \gamma = \gamma'$. Since $a_3$ has isotropy one, $\gamma' = \gamma_3 : a_3 - a_5$. Since $\gamma'$ is in the $G$-orbit of $\gamma$, we see that $\gamma',\alpha_3, \alpha_4$ form the triangle $g\Delta \ne \Delta$. In particular, in this case, $m > 2, m \ne 3$ and $\gamma'$ is the only other arc adjacent to $a_3$. This yields the claim, by induction. Consider a triangle $\Delta'$ other than $\Delta$ having $\gamma_1$ as a side. We have no choice that this triangle has sides $\gamma_3$ and $\gamma_{m-1}$. Thus, $m = 6$ and $(S,M)$ is the disk with $6$ marked points on the boundary and one internal triangle.
\end{proof}

\begin{Lemma} \label{LemmaOneOrbit} Assume that $\partial S = \emptyset$ and that all arcs are in the same orbit. Then $(S,M,T)$ has exactly two orbits of triangles, one orbit of arcs and one orbit of punctures. In particular, $(S_G, M_G, \mathcal{O})$ is the once-punctured sphere with two orbifold points.
\end{Lemma}

\begin{proof}Clearly, there is no self-folded triangle in $(S,M,T)$. We claim that there are two orbits of triangles in $(S,M,T)$ for the action of $G$. Assume there is exactly one orbit of triangles. Consider an arc $\gamma$ with its two adjacent triangles as follows.

\begin{figure}[h]
  \centering
  \begin{tikzpicture}[xscale=2.30,yscale=1.5]
\node at (0, 0.1) {$\gamma$};
\node at (-0.7, 0.4) {$\mu$};
\node at (-0.7, -0.4) {$\alpha$};
\node at (0.7, 0.4) {$\nu$};
\node at (0.7, -0.4) {$\beta$};
\node at (-1.1, 0) {$a$};
\node at (1.1, 0) {$b$};
\node at (0,1.1) {$c$};
\node at (0,-1.1) {$d$};
\node at (0,1) {$\bullet$};
\node at (0,-1) {$\bullet$};
\node at (-1,0) {$\bullet$};
\node at (1,0) {$\bullet$};
\draw (-1.00,0) -- (0,1);
\draw (1.00,0) -- (0,1);
\draw (-1.00,0) -- (0,-1);
\draw (1.00,0) -- (0,-1);
\draw (-1.00,0) -- (1,0);
\end{tikzpicture}
\end{figure}

Let $g \in G$ sending the upper triangle to the lower triangle. Since $G$ is admissible, either $g \gamma = \alpha$ or $g \gamma = \beta$. With no loss of generality, assume the first case occurs. Since $g$ is orientation-preserving, we get $ga=a$ and $g \mu = \gamma$.
This implies that no non-trivial element of $G$ maps $\zD$ to itself. Indeed, if $g'\zD=\zD$, say $g'\za=\zg,g'\zg=\zb$ and $g'\zb=\za$, then the element $(g'g)$ fixes $\zg$. Note that $g'g$ sends $\mu$ to $\zb$ and $\nu$ to $\alpha$. Since $S$ is not the once-punctured torus, this yields that $g'g \ne 1$, contradicting that $G$ is admissible.
Now, since all triangles lie in one $G$-orbit and no non-trivial element of $G$ maps a triangle to itself, we see that there are exactly $|G|$ triangles in $(S,M,T)$. But three times the number of triangles should be twice the number of arcs, since $T$ has no self-folded triangles and $\partial S = \emptyset$. This is a contradiction.

Thus, the two triangles in the above figure lie in distinct orbits. Since all arcs of $T$ are in the orbit of $\gamma$, we have exactly two orbits of triangles in $(S,M,T)$. Let $g' \in G$ with $g' \mu = \gamma$. Since the upper triangle is not in the orbit of the lower triangle and since $g'$ is orientation preserving, we see that $g'$ maps the upper triangle to itself. Similarly, there is a non-identity element of $G$ that maps the lower triangle to itself. In particular, we have $|\mathcal{O}| = 2$ and $a,b$ lie in the same orbit, and thus, all punctures lie in the same orbit. Since $S$ has no boundary, so is $S_G$. The Euler characteristic of $S_G$ is $2 + 1 - 1 = 2$, so $S_G$ has to be a sphere.
\end{proof}

Observe that in the situation of the above lemma, the Euler characteristic of $S$ is
$$ \frac{|G|}{m_a} - \frac{|G|}{3},$$
where $a$ is any puncture in $M$.
So if $S$ is a sphere, $m_a=1,2$. In the first case, $|G|=3$, and we have $3$ arcs, $3$ punctures and this is the sphere with three punctures and three arcs on the equator. This is excluded. In the second case, $|G|=12$, and we have $12$ arcs, $8$ triangles and $6$ punctures. This is the octahedron with the regular triangulation.

\begin{Exam}
Consider the torus with four punctures as follows.
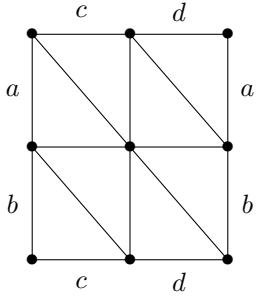
\begin{figure}[h]
  \centering
  \begin{tikzpicture}[xscale=1.30,yscale=1.5]

    \draw (0,2) -- (2,2);
\draw (0,0) -- (2,0);
\draw (0,2) -- (0,0);
\draw (2,2) -- (2,0);
\draw (0,2) -- (2,0);
\draw (0,1) -- (1,0);
\draw (1,2) -- (2,1);
\draw (0,1) -- (2,1);
\draw (1,0) -- (1,2);

\node at (0,0) {$\bullet$};
\node at (0,1) {$\bullet$};
\node at (0,2) {$\bullet$};
\node at (1,0) {$\bullet$};
\node at (1,1) {$\bullet$};
\node at (1,2) {$\bullet$};
\node at (2,0) {$\bullet$};
\node at (2,1) {$\bullet$};
\node at (2,2) {$\bullet$};

\node at (-0.2,1.5) {$a$};
\node at (2.2,1.5) {$a$};
\node at (-0.2,0.5) {$b$};
\node at (2.2,0.5) {$b$};

\node at (1.5, 2.2) {$d$};
\node at (1.5, -0.2) {$d$};
\node at (0.5, 2.2) {$c$};
\node at (0.5, -0.2) {$c$};

  \end{tikzpicture}
\caption{A torus with $4$ punctures}
\label{fig:torus3}
\end{figure}
Consider the group $G$ generated by all rotations of $2\pi/3$ about the punctures and centers of the triangles. It is not hard to check that $G$ has order $12$ and is admissible. All arcs are in the same orbit and we are in the situation of the above lemma.
\end{Exam}


\begin{Remark}In the situation of Lemma \ref{LemmaOneOrbit}, if $(S,M)$ has at least two punctures, then the orbit mutation corresponds to changing all taggings at all punctures. On the orbifold, the mutation changes the taggings (which are necessarily the same) at both ends of the unique tagged arc.

If $(S,M)$ has exactly one puncture, then there is no way to go from a triangulation $T$ to the triangulation $T'$ obtained from $T$ by changing the tag at the puncture, using only finitely many flips.
\end{Remark}

\begin{Lemma}
Assume that $\partial S = \emptyset$ and all arcs of $T$ lie in the same orbit.
If $S$ is not a once-punctured surface, we have $p_{G,\gamma} = (2m_a)^2$.
\end{Lemma}

\begin{proof}
Assume that $S$ is not a once-punctured surface. We claim that $T$ cannot consist of loops only. Indeed, assume it is the case. Consider a triangle from $T$. Then this triangle has a single vertex $a$. Take any arc $\alpha$ of this triangle. Then $\alpha$ is a side of another triangle, which then is also a triangle having only vertex $a$. By continuing this process, we see that all of the triangles from $T$ have only vertex $a$. So $S$ is once-punctured, a contradiction. This proves our claim. Since all arcs are in the same orbit, there is no loop in $T$.

Observe that there exists a sequence of mutations such that the overall effect is changing all tags at the punctures. We just need to apply Lemma \ref{LemmaPuncturedDisk} successively for each puncture. Fix an arc $\gamma: a - b$ in $T$. Then there are exactly $m:=2m_a$ arcs of $T$ having $a$ as endpoint. Let us denote these arcs by $\gamma_1, \ldots, \gamma_{m}$, in clockwise orientation around $a$ such that $\gamma = \gamma_{m}$. Let $\alpha_i$ be such that $\gamma_i, \gamma_{i+1}, \alpha_i$ is a triangle of $T$ (where $\gamma_{m+1}$ means $\gamma_1$). By applying a sequence of mutations at $\gamma_1, \gamma_2, \ldots, \gamma_{m-1}$, the arc $\gamma_{m-1}$ becomes
$$\gamma_{m}\left(\frac{\alpha_1}{\gamma_1\gamma_2} + \frac{\alpha_2}{\gamma_2\gamma_3} + \cdots + \frac{\alpha_{m-1}}{\gamma_{m-1}\gamma_{m}} + \frac{\alpha_{m}}{\gamma_{m}\gamma_1}\right).$$
This gives the arc $\gamma^{a}$ which is obtained from $\gamma$ by changing the tag at $a$. After identifying all arcs in $G\gamma$ to a single variable $x$, this arc $\gamma^{a}$ becomes $m$. Similarly, the arc $\gamma^{b}$ obtained from $\gamma$ by changing the tag at $b$ becomes $2m_b = 2m_a = m$ after identifying all arcs of $G\gamma$ by $x$. Now, using \cite[Theorem 12.9]{MSW}, the arc $\gamma^{ab}$ obtained from $\gamma$ by changing both tags is such that $\gamma^{ab}\gamma = \gamma^{a}\gamma^{b}$. Therefore, after identifying all arcs of $G\gamma$ to the variable $x$, we get $xx' = (m)^2$.
\end{proof}

\subsection{Remaining cases}

We may assume that no triangle in $\mathcal{Q}$ is self-folded also that $\mathcal{Q}$ does not form a once-punctured bigon. We know from Lemma \ref{LemmaExchange} that the exchange polynomial $p_{\zg}$ is $\mu\bar\mu\beta\bar\beta + \nu\bar\nu\alpha\bar\alpha$. Also, we may assume that none of $\mu, \nu, \alpha, \beta$ lie in $G\gamma$. Since $\gamma$ is not an arc of a self-folded triangle, none of $\bar\mu, \bar\nu, \bar\alpha, \bar\beta$ lie in $G\gamma$. Therefore, we have
\[p_{G, \zg}= F(p_\gamma)= G\mu G\bar\mu G\beta G\bar\beta + G\nu G\bar\nu G\alpha G\bar\alpha.\]

\subsection{Exchange polynomials and cluster algebra structure revisited}

As promised at the beginning of this section, the results collected so far yield the following.
\begin{Prop}
Let $G$ be an admissible group of $T$-automorphisms of $(S,M)$ where $T$ is a tagged triangulation. Let $\{\tau_1, \ldots, \tau_r\}$ be a $G$-orbit of tagged arcs and $\{\tau_1', \ldots, \tau_r'\}$ be the orbit mutation, where $T' = (T \backslash \{\tau_1, \ldots, \tau_r\}) \cup (\{\tau_1'', \ldots, \tau_r''\})$ is such that $G$ is an admissible group of $T'$-automorphisms. Then $F(\tau_i'') = F(\tau_j'')$ and $F(\tau_i\tau_j'') \in \mathbb{Z}[\mathbf{y}]$ for all $1 \le i,j \le r$. The polynomial $F(\tau_i\tau_j'') = F(\tau_1)F(\tau_1'')$ is the exchange polynomial $P_{G,\tau_1}$.
\end{Prop}

\begin{Remark} These exchange polynomials allow us to define generalized cluster variables through mutations, and hence a generalized cluster algebra in $\mathcal{F}_G$. A priori, this algebra depends on $(S,M,T)$ and on $G$, however, we will see in Section \ref{sect 7} that it only depends on the orbifold $(S_G, M_G, \mathcal{O})$ with induced triangulation $T_G$.
\end{Remark}

\section{Generalized cluster algebra of an orbifold}
\label{sect 6}
Let $(S, M, \mathcal{O})$ be an orbifold with a tagged triangulation $T$. Consider the function $m: M \to \Z_{\ge 1}$ such that $m_b : = m(b)$ is one whenever $b$ is not a puncture. Let $\{\tau_1, \ldots, \tau_s\}$ denote the set of arcs of $T$. We also identify these arcs with  indeterminates  $y_1, \ldots, y_n$. The boundary segments are identified with $1$ and for each arc $\alpha$, we have $\bar \alpha \in \Z[y_1^{\pm 1}, \ldots, y_s^{\pm 1}]$, which is $1$ unless $\tau(\alpha)$ is an arc of a $1$-self-folded triangle in $\tau(T)$. In the latter case, $\bar \alpha$ is the element corresponding to the unique arc of $T$, also denoted $\bar \alpha$, with $\bar \alpha^{0}=\alpha$.
For $1 \le i \le n$, the \emph{mutation} $\mu_i(T)$ \emph{in direction} $i$ of $T$ is the tagged triangulation $(\{\tau_1, \ldots,\tau_n\} \backslash \{\tau_i\}) \cup \{\tau_i'\}$ of $(S, M, \mathcal{O})$ where $\tau_i'$ is not isotopic to $\tau_i$. Such an arc always exists an is uniquely determined. Now, we explain how to perform the corresponding mutation in  $\mathbb{Q}(y_1, \ldots, y_s)$.
\medskip

For each $\tau\in T$, let $p^{-}_{\tau}$ (respectively $p^{+}_{\tau}$) be the product of all $\za\bar{\za}$ where $\za$ is an arc of $T\setminus\{\tau\}$ or a boundary segment such that $\za,\tau$ are sides of a triangle in $T$ and $\za$ is following $\tau$ in the counter-clockwise (respectively clockwise) direction.
 Observe that $p^{-}_{\tau}$ and $p^{+}_{\tau}$ are not always relatively prime. For example, in the once-punctured bigon of Figure \ref{fig bigon} we have $p_\zg^-=\za\bar\za\nu\bar\nu$ and $p_\zg^+=\za\bar\za\zb\bar\zb$.

\begin{Defn}\label{def ca}
  For each $\tau\in T$, define a polynomial $p_{\tau}$ in $\Z[y_1, \ldots, y_s]$ as follows.
 \begin{enumerate}[$(a)$]
\item If  $S$ is the  sphere with one $m$-puncture with $m \ge 1$ and two orbifold points, then $T$ has only one arc $\tau$ and \[p_{\tau}=(2m)^2.\]
\item Let $\tau: a - a$ enclose a monogon $\Delta$ with an orbifold point $o$, and assume we are not in case $(a)$. Let $\Delta'$ be the other triangle adjacent to $\tau$ (which cannot be an orbifold triangle). Let $\za:a - b, \zb: a - b$ be the other arcs of this triangle.
\begin{enumerate}[$(i)$]
\item If $\Delta$ is $m$-self-folded with $m=m_b=1$ or $m_a =1$, then $\za = \zb = \bar \tau$ and $S$ is a sphere with two punctures, one orbifold point, and  $T$ has precisely two arcs $\tau$ and $\za$. We have \[p_{\tau} = 3\za.\]
\item Otherwise, we have \[p_{\tau} = \za^2 + \za\zb + \zb^2.\]
\end{enumerate}
\item Let $\tau$ be a loop at $a$ of a $1$-self-folded triangle such that $\tau$ is not as in case (b). Then \[p_{\tau} = \frac{p^{-}_{\tau} + p^+_{\tau}}{\tau\bar \tau}.\]
\item Let $\tau$ be a radius of a $1$-self-folded triangle with loop $\bar \tau$ at $a$. Then
$ p_{\tau} = p_{\bar \tau}$, unless $S$ is the once-punctured monogon, in which case we set $p_{\tau}=2$.
\item Let $\tau: a - b$ be a radius of a once-punctured $1$-bigon with radii $\tau: a - b, \za: a - c$ where $a \ne c, a \ne b$. Then \[p_{\tau} = \frac{p^{-}_{\tau} + p^+_{\tau}}{\za\bar \za}.\]
\item Let $\tau: a - b$ be a radius of an $m$-self-folded triangle where $m > 1$ and with loop $\za$. Then \[p_{\tau} = m\za\bar \za.\]
\item Otherwise, we let $p_{\tau} = p^{-}_{\tau} + p^{+}_{\tau}$ (with the possibility that $p^{-}_{\tau}, p^{+}_{\tau}$ have common factors)
\end{enumerate}
\end{Defn}

\begin{Defn} Let $\tau$ be an arc in the triangulation $T$ and let $y\in \{y_1,\ldots,y_s\}$ be the corresponding cluster variable. Let $\tau'$ be the arc obtained by flipping $\tau$ and let $y'$ denote the Laurent polynomial $p_\tau/y$
in $\mathbb{Z}[y_1^{\pm1}, \ldots, y^{\pm1}_s]$. It is not hard to check that $(\{y_1, \ldots, y_s\} \backslash \{y\}) \cup \{y'\}$ are again algebraically independent in $\mathbb{Q}(y_1, \ldots, y_s)$. We call $y_1, \ldots, y_s$ the \emph{initial cluster variables}. Any arc $\zg$ lying in a triangulation that can be obtained from $T$ by a finite sequence of mutations gives rise to a Laurent polynomial $y_\zg$. Such a $y_\zg$ is called a \emph{cluster variable}.
We define an algebra $\cala(S,M, \mathcal{O}) \subseteq \mathbb{Q}(y_1, \ldots, y_s)$ to be the $\Z$-subalgebra of $\mathbb{Q}(y_1, \ldots, y_s)$ generated by all cluster variables. We call it the \emph{generalized cluster algebra of the orbifold} $(S,M, \mathcal{O})$.
\end{Defn}

Some cases of the mutation rules are pictured in Figure \ref{figmut}.  The first column represents a local configuration in the tagged triangulation of the orbifold. The configuration in the second column is obtained by flipping the arc $y_i$ and the third column show the exchange relation in the cluster algebra $\mathcal{A}(S,M,\mathcal{O})$. In the last two cases, $S$ is a sphere and the picture represents the entire triangulation.
\begin{figure}[h]
\begin{center}
\input{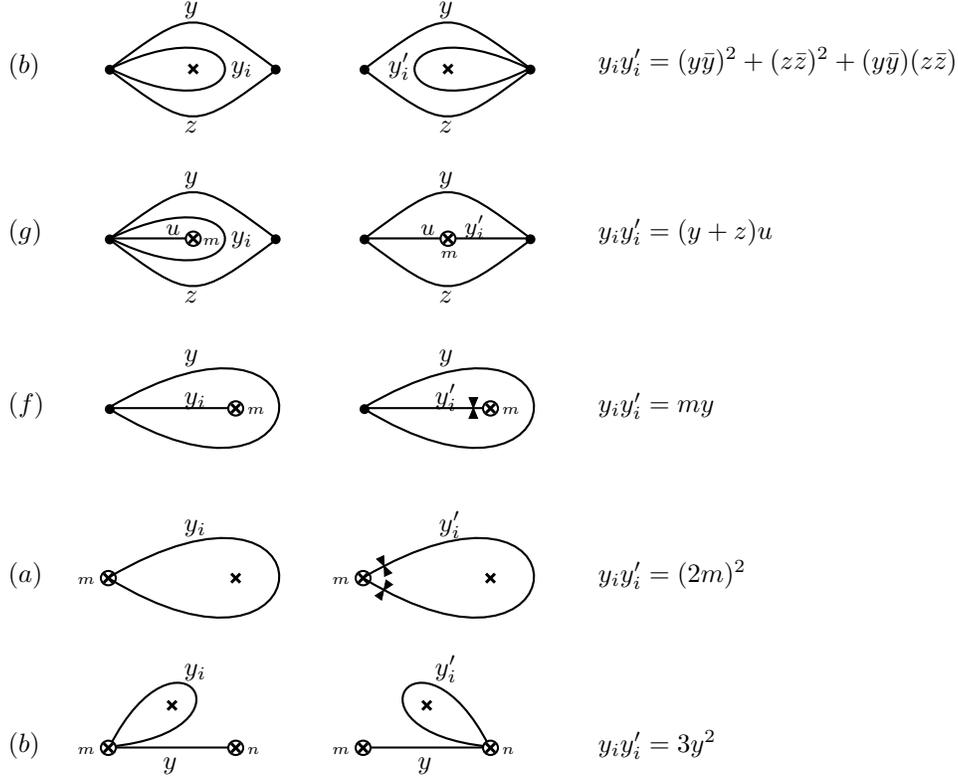}
\caption{Some cases of the mutation rules.
}
\label{figmut}
\end{center}
\end{figure}

\begin{Remark}
 \label{rem LP}
 This computation shows that our notion of generalized cluster algebra is different from the one of Chekhov-Shapiro \cite{ChSh} and Lam-Pylyavskyy \cite{LP}. Indeed in the second row of Figure \ref{figmut}, the two summands of the exchange polynomial have a non-trivial common factor, which is not allowed in loc.cit.
\end{Remark}

\medskip
Now, let us classify the generalized cluster algebras of orbifolds with one or two arcs.

\subsection{Rank n=1}\label{sect 61}
By Lemma \ref{LemmaRank}, we have $1=6(g-1) + 3b +3p+2x+c$. If $g\ge1$ this equation has no solution, because if $b=0$ then $c=0$. Thus $g=0 $ and the equation becomes
\[ 7=3b +3p+2x+c.\]
This equation has the following four solutions.
\subsubsection{The sphere with 1 puncture and 2 orbifold points}
If $b=0$, then $c=0$ and $p=1,x=2$, and we have a sphere with one puncture and two orbifold points. The two cluster variables are
\[ y \quad\textup{and}\quad 4m^2/y\]
where $m$ is the isotropy of the puncture.

\medskip
If $b=1$ our equation becomes
\[ 4=3p+2x+c, \textup{ with } c\ge 1,\]
which has three solutions.

\subsubsection{The square}
If $p=0,x=0$ and $c=4$, we have the disk with 4 marked points on the boundary. The generalized cluster algebra is the honest cluster algebra of rank 1 (type $\mathbb{A}_1$) with cluster variables
\[ y \quad\textup{and}\quad 2/y.\]

\subsubsection{The bigon with 1 orbifold point}
If $p=0,x=1$ and $c=2$, we have the disk with 2 marked points on the boundary and one orbifold point in the interior. The two cluster variables are
\[ y \quad\textup{and}\quad 3/y.\]

\subsubsection{The once-punctured monogon}
If $p=1,x=0$ and $c=1$, we have the disk with 1 puncture and 1 marked point on the boundary.
If the isotropy $m$ of the puncture is one, we obtain the honest cluster algebra of rank 1 again (case (d) of Definition \ref{def ca}). If $m>1$, the two cluster variables are
\[ y \quad\textup{and}\quad m/y.\]

In rank 1 all 4  cases  can be obtained from a triangulation $T$ of a surface $(S,M)$ and an admissible group $G$ of $T$-automorphisms. For case $(1)$, one  takes for $(S,M,T)$ the octahedron as in Example \ref{ExampleOctahedron} (the isotropy of the puncture is then $2$). The group $G$ is the alternating group $A_4$. Case (2) is a surface. For case $(3)$, one takes for $(S,M,T)$ the disk with six marked points on the boundary and a single internal triangle. The group $G$ is of order $3$.
 Finally, case (4) is obtained from the once-punctured disk with $m$ marked points under the action of the group of order $m$ given by rotations.

\subsection{Rank n=2}\label{sect 62}

Now Lemma \ref{LemmaRank} implies
$2=6(g-1) + 3b +3p+2x+c$. Again there is no solution if $g\ge1$. Thus $g=0 $ and the equation reads
\[ 8=3b +3p+2x+c.\]
This equation has the following 6 solutions.

\subsubsection{The sphere with 2 punctures and 1 orbifold point}
If $b=0$, then $c=0$ and $p=2,x=1$, and we have a sphere with 2 punctures and 1 orbifold point. Let $r, s$ be the isotropies of the punctures. When none of $r,s$ is one, the generalized cluster algebra has 8 cluster variables
$$x_1, x_2, \frac{3x_2^2}{x_1}, \frac{3sx_2}{x_1}, \frac{9s^2}{x_1}, \frac{3rs}{x_2}, \frac{3r^2x_1}{x_2^2}, \frac{rx_1}{x_2}.$$
Otherwise, when for instance $s=1$, we get $6$ cluster variables $$x_1, x_2, \frac{3x_2}{x_1}, \frac{3r}{x_1}, \frac{3r}{x_2}, \frac{rx_1}{x_2}.$$
For $r=3$ and $s=1$ this orbifold is obtained from the triangulation of the sphere with $4$ punctures and three self-folded triangles
shown in Figure \ref{fig sphere} under the action of rotations about $\pi/3$ and $2\pi/3$ degrees centered at the common puncture.
\begin{figure}
\begin{center}
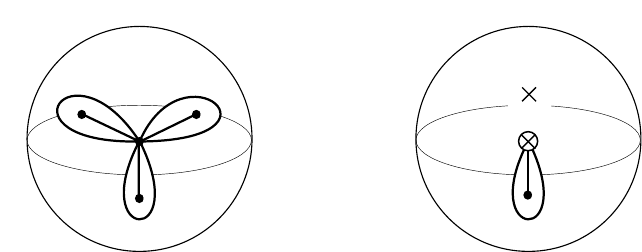
\caption{The sphere with two punctures and one  orbifold point as orbit space of a sphere with four punctures.}
\label{fig sphere}
\end{center}
\end{figure}

\medskip
If $b=1$ our equation becomes
\[ 5=3p+2x+c, \textup{ with } c\ge 1,\]
which has four solutions.

\subsubsection{The pentagon}
If $p=0,x=0$ and $c=5$, we have a disk with 5 marked points on the boundary. The generalized cluster algebra is the honest cluster algebra of the pentagon (type $\mathbb{A}_2$) and has 5 cluster variables
\[x_1,x_2,\frac{x_2+1}{x_1},\frac{x_1+x_2+1}{x_1x_2},\frac{x_1+1}{x_2}.\]

\subsubsection{The triangle with 1 orbifold point}
If $p=0,x=1$ and $c=3$, we have a disk with 3 marked points on the boundary and one orbifold point in the interior. The generalized cluster algebra has 6 cluster variables
\[x_1,x_2,\frac{x_1^2+x_1+1}{x_2},
\frac{x_1^2+x_1+x_2+1}{x_1x_2},
\frac{x_1^2+x_2^2+x_1x_2+x_1+2x_2+1}{x_1^2x_2},\frac{x_2+1}{x_1}.
\]
This orbifold is obtained from the triangulation of the disk with 9 marked points
shown in Figure \ref{fig triangle} under the action of rotations about $\pi/3$ and $2\pi/3$ degrees.
\begin{figure}
\begin{center}
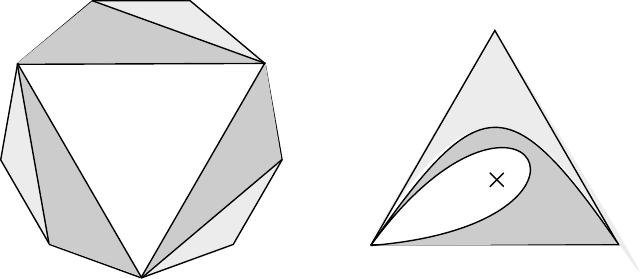
\caption{The triangle with one orbifold point as orbit space of a disk with 9 marked points.}
\label{fig triangle}
\end{center}
\end{figure}

\subsubsection{The monogon with 2 orbifold points}
If $p=0,x=2$ and $c=1$, we have a disk with one marked point on the boundary and two orbifold points in the interior. The generalized cluster algebra has infinitely many cluster variables
\[\cdots,\frac{x_1^2+x_1+1}{x_2}, x_1,x_2,\frac{x_2^2+x_2+1}{x_1},
\frac{({x_1^2+x_1+x_2+1})^2+({x_1^2+x_1+x_2+1})+1}{x_1^2x_2},\ldots
\]
This orbifold is obtained from the triangulation of the sphere with 3 boundary components and 3 marked points
shown in Figure \ref{fig monogon} with a group of order 3 acting by cyclically shifting the boundary components. The north and south pole are fixed by this action and give rise to the two orbifold points in orbit space.
\begin{figure}
\begin{center}
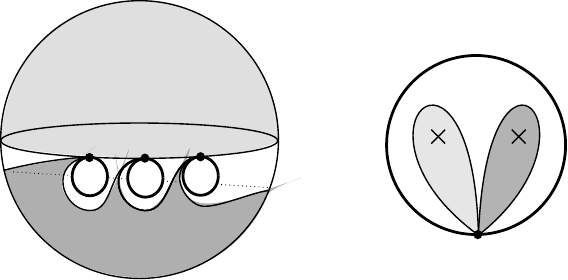
\caption{The monogon with two orbifold points as orbit space of a sphere with 3 boundary components.}
\label{fig monogon}
\end{center}
\end{figure}

\subsubsection{The once-punctured bigon}
If $p=1,x=0$ and $c=2$, we have a disk with one puncture and two marked points on the boundary. If the isotropy $m$ of the puncture is one, the generalized cluster algebra is the honest cluster algebra of type $\mathbb{A}_1\times\mathbb{A}_1$, with 4 cluster variables
\[ x_1,x_2,\frac{2}{x_1},\frac{2}{x_2}.\]
If $m>1$, the generalized cluster algebra has 6 cluster variables
\[x_1,x_2,\frac{2x_2}{x_1},\frac{2m}{x_1}, \frac{2m}{x_2}, \frac{2x_1}{x_2}.\]
The clusters and triangulations are shown in Figure \ref{fig bigonmutation}.

\begin{figure}
\begin{center}
\scalebox{1}{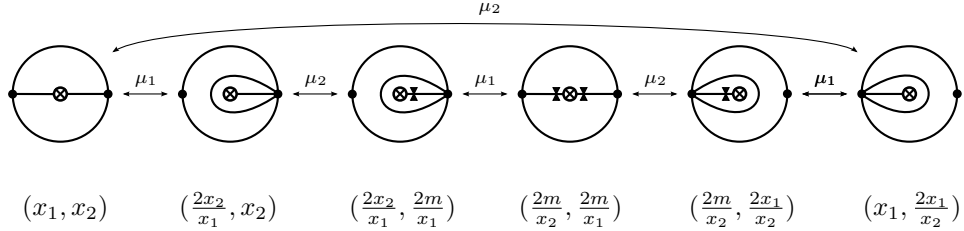}
\caption{The exchange graph of the once-punctured $m$-bigon.}
\label{fig bigonmutation}
\end{center}
\end{figure}

\medskip
Finally, if $b=2$, our equation becomes
\[ 2=3p+2x+c, \textup{ with } c\ge 2,\]
which has one solution.
\subsubsection{The annulus with 2 marked points} If $p=0,x=0$ and $c=2$, we have the annulus with one marked point on each boundary component. The generalized cluster algebra is the honest cluster algebra of type $\widetilde{\mathbb{A}}_{1,1}$ (Kronecker) with infinitely many cluster variables
\[
\ldots \frac{x_1^2+1}{x_2}, x_1,x_2,
\frac{x_2^2+1}{x_1},
\frac{x_2^4+2x_2^2+x_1^2+1}{x_1^2x_2}\ldots\]

\begin{Exam}
Consider the torus in the left picture in Figure \ref{fig:torus} with $15$ vertices, $30$ triangles and $45$ arcs. We consider the group generated by the rotations at the center of the triangles $s_i$, and the rotation around the punctures $p_i$ and $q_i$. It is not hard to check that all triangles $s_i$ are in the same orbit, all triangles $r_i$ are in the same orbit, all triangles adjacent to a $p_i$ are in the same orbit, and all triangles adjacent to a $q_i$ are in the same orbit. There are exactly four orbits of triangles. There are $5$ orbits of arcs (arcs having an $p_i$ as an endpoint, arcs having an $q_i$ as an endpoint, arcs adjacent to a triangle $s_i$, the others). Therefore, the group $G$ has order $9$. Observe also that the only triangles that are mapped to themselves by a non-trivial element of  $G$ are the triangles $s_1, s_2, s_3$. Finally, observe that there are three orbits of punctures, the orbit of the $p_i$, the orbit of the $q_i$ and the orbit of the other punctures. The orbifold has two $3$-punctures, one $1$-puncture, four triangles and $5$ arcs. By computing the Euler characteristic, we get $3 + 4 - 5 = 2$. Since $S$ has no boundary, so does $S_G$. Therefore, $S_G$ is the sphere shown in the right picture in Figure \ref{fig:torus}. It has one orbifold point and two 3-punctures.
\begin{figure}
\begin{center}
\LARGE
\scalebox{0.5}{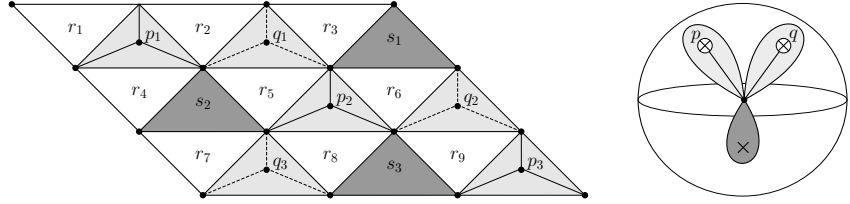}
\caption{A torus with 15 vertices and its orbit space, a sphere with two 3-punctures and one orbifold point.}
\label{fig:torus}
\end{center}
\end{figure}
\end{Exam}

\begin{Remark} \label{RemDP}
Observe that in the rank two case, there are two situations where we have exactly $6$ cluster variables. In these cases, all the cluster variables are Laurent monomials, and this, no matter what seed we use to express them. Consider now the quasi-cluster algebra as defined in \cite[Section 6.1]{DP} obtained by taking the quasi-triangulations of the M$\ddot{o}$bius strip with two marked points on the boundary. There are exactly $6$ quasi-cluster variables and they are not all Laurent monomials in the two initial quasi-cluster variables. Thus, this quasi-cluster algebra cannot be obtained as a generalized cluster algebra of an orbifold. On the other hand, a quasi-cluster algebra of rank two arising from a connected non-orientable surface has to the M$\ddot{o}$bius strip with two marked points; see \cite[Prop. 3.7]{DP}. Therefore, our generalized cluster algebras are not quasi-cluster algebras, and vice versa.
\end{Remark}

\begin{Remark} Not every orbifold is of the form $(S_G, M_G, \mathcal{O})$. The orbifolds of the form $(S_G, M_G, \mathcal{O})$ are called \emph{good orbifolds}, following W. Thurston's terminology. For instance, the sphere with one orbifold point and with all punctures of isotropy one is not a good orbifold. Proposition \ref{PropBijec} guarantees that if our orbifold is of the form $(S_G, M_G, \mathcal{O})$ for a surface $(S,M)$ and an admissible group $G$ of $T$-automorphisms for some tagged triangulation $T$, then any tagged triangulation of the orbifold comes from a ($G$-stable) tagged triangulation of the surface $(S,M)$.
\end{Remark}

\section{Relationship between the cluster algebras}\label{sect 7}

Let $(Q,W)$ be a Jacobi-finite quiver with potential. Denote by $\calc=\calc(Q,W)$ the cluster category and by $B=J(Q,W)$ the Jacobian algebra. Let $\cala=\cala(Q)$ be the cluster algebra (without coefficients), and denote by $U=\zG(Q,W)$ the cluster-tilting object in $\calc$ corresponding to the initial seed.

Let $G$ be an admissible group of automorphisms of $(Q,W)$ and denote by $\calc_G=\calc(Q_G,W_G)$, $B_G=J(Q_G,W_G)$ the cluster category and the Jacobian algebra determined by the action of $G$, respectively.
We denote by $U_G$ the basic cluster-tilting object in $\calc_G$ corresponding to $\zG(Q_G,W_G)$.

We decompose $U$ according to its $G$-orbits as follows
\[U=U_1\oplus\cdots\oplus U_s,\]
where $U_i=\oplus_{g\in G} \,gU_i'$ with $U_i'$ indecomposable.
The initial cluster variables of $\cala$ are denoted accordingly by $x_{i,j}$, $1\le i \le s$, $1\le j\le |G|$, where the variables $x_{i,1},\ldots,x_{i,|G|}$ correspond to the indecomposable summands of $U_i$.

As before, we let $\mathcal{F}_G:=\mathbb{Q}(y_1, \ldots, y_s)$. Let $F\colon\mathbb{Z}[\mathbf{x}^{\pm 1}]\to
\mathbb{Z}[\mathbf{y}^{\pm 1}]$ be the homomorphism  such that $F(x_{i,j})=y_i$.

\smallskip

 We let $\mathcal{G}$ denote the exchange graph of all cluster-tilting objects of $\C$.
 By definition, the vertices of $\mathcal{G}$ are the cluster-tilting objects of $\mathcal{C}$ and the edges are given by mutations. Note that $\mathcal{G}$ does not need to be connected. Let $\mathcal{G}(U)$ be the connected component of $\mathcal{G}$ containing $U$. We denote by $\mathcal{X}$
 the \emph{reachable} indecomposable rigid objects of $\C$. In other words, $\calx$ corresponds to the indecomposable direct summands of the objects from $\mathcal{G}(U)$. If $\C$ is of acyclic type or is the Amiot cluster category of a surface without punctures, then $\mathcal{X}$ is the set of all indecomposable rigid objects in $\C$.

Following \cite{BIRS}, we say that  $\mathcal{C}$ has a \emph{cluster structure} if one of the following equivalent conditions hold.
 \begin{enumerate}
\item Whenever two cluster-tilting objects $T,T'$ in $\mathcal{G}(U)$ are related by a mutation $T'=\mu_i(T)$ then the quivers $Q_T,Q_{T'}$ of the endomorphism algebras ${\rm End}_\C(T)$, ${\rm End}_\C(T')$ are related by the Fomin-Zelevinsky mutation,  $Q_{T'}=\mu_i(Q_T)$.
\item The quiver of any cluster-tilting object in  $\mathcal{G}(U)$  has no loop and no $2$-cycle.
\item The potential $W$ is non-degenerate.
\end{enumerate}

It follows from Proposition \ref{PropOneDim} and the multiplication formula from \cite{Palu} that if $\calc$ has a cluster structure then the cluster character $X$ commutes with mutations in $\mathcal{G}(U)$ and mutations in $\cala$. In particular, the $X_M$ for $M \in \mathcal{X}$ are exactly the cluster variables of $\cala$. Moreover, $\mathcal{G}(U)$ is isomorphic to the exchange graph of $\cala$.

 Let $\mathcal{G}_G$ be the  graph whose vertices are the $G$-stable cluster-tilting object of $\C$ that can be obtained from $U$ by a sequence of  Iyama-Yoshino mutations  of $G$-orbits, and whose edges are the  Iyama-Yoshino mutations.
 Note that  $\mathcal{G}_G$ is connected. We let $\mathcal{X}_G$ denote the set of indecomposable direct summands of the vertices of $\mathcal{G}_G$. In general, $\mathcal{X}_G$ does not need to be a subset of $\mathcal{X}$.

 In terms of the cluster algebra,  when $W$ is non-degenerate and $\calx_G \subseteq \calx$, the set $\calx$ is the set of all cluster variables of $\cala$ and $\calx_G$ is the subset of those cluster variables that lie in the $G$-stable clusters obtained from the initial cluster by $G$-orbit mutations.

\subsection{The $G$-mutation connected case}

\begin{Defn}
The cluster category $\C$ is called $G$-\emph{mutation connected} if any finite sequence of mutations from $U_G$ in $\C_G$ is given by a finite sequence of mutations from $U$ in $\C$.
\end{Defn}

\begin{Remarks}(1) This definition is equivalent to the following. Any vertex  of $\mathcal{G}_G$ can be obtained from $U$ by a finite sequence of mutations in $\C$.

\medskip

\noindent (2) If $\C$ is $G$-mutation connected, then $\mathcal{X}_G$ is a subset of $\mathcal{X}$.\end{Remarks}

Let $\C$ be $G$-mutation connected. For each cluster variable $x$ in $\cala$, we have that $F(x) \in \mathcal{F}_G$. We define the \emph{cluster algebra of orbits} $\cala_G$ associated to $\C_G$ to be the $\mathbb{Z}$-subalgebra of $\mathcal{F}_G$ generated by the set of all $F(x)$ with $x$ a cluster variable of $\mathcal{X}_G$.

\begin{Prop}
Let $\C$ be $G$-mutation connected and let $\cala'$ denote the $\Z$-subalgebra of $\cala$ generated by the cluster variables in $\mathcal{X}_G$. Then we have a commutative diagram of algebras and their generating sets
$$\xymatrix{
\cala' \ar[d] \ar[r] & \cala\ar[d] &\quad& \calx_G\ar[r]\ar[d] &\calx\ar[d] \\
\cala_G \ar[r]  & \cala/\langle x_{i,j}- x_{i, j'}\rangle && F(\calx_G)\ar[r]&F(\calx)}
$$
where the horizontal maps are inclusions and the vertical maps are surjective and induced by $F$.
\end{Prop}

\begin{proof} First note that  $ \cala/\langle x_{i,j}- x_{i, j'}\rangle $ is generated by $F(\calx)$. This follows from the fact that $\mathcal{X}_G$ is a subset of $\mathcal{X}$, thanks to $\C$ being $G$-mutation connected. Also note that $\cala/\langle x_{i,j}- x_{i, j'}\rangle$ is well-defined since $\cala \subseteq \mathbb{Z}[\mathbf{x}^{\pm 1}]$.
\end{proof}

Recall that every cluster variable in $\cala$ is a Laurent polynomial in the variables of any given cluster. This is the Laurent phenomenon and was proven for cluster algebras by Fomin and Zelevinsky in \cite{FZ}. It follows from this that the generalized cluster variables in $\cala_G$ also satisfy the Laurent phenomenon, in the $G$-mutation connected case. Therefore, we can define the \emph{upper-cluster algebra} $U(\cala_G)$ to be
$$U(\cala_G) = \bigcap_{{\bf x} \in \mathcal{G}_G}\mathbb{Z}[F({\bf x})^{\pm 1}] = \bigcap_{{\bf y} \; \text{cluster in} \; \cala_G}\mathbb{Z}[{\bf y}^{\pm 1}]$$
where for a set $S = \{a_1, \ldots, a_r\}$ of rational functions, we write $S^{\pm 1}$ for $\{a_1^{\pm 1}, \ldots, a_r^{\pm 1}\}$.
Now, the Laurent phenomenon guarantees that $\cala_G \subseteq U(\cala_G)$.

\begin{Prop} Let $\C$ be $G$-mutation connected. Assume that $\cala_G = U(\cala_G)$. Then $\cala_G = \cala/\langle x_{i,j}- x_{i, j'}\rangle$.
\end{Prop}

\begin{proof} Take any cluster variable $x$ in $\cala$. Since $\C$ is $G$-mutation connected and by the Laurent phenomenon, we see that $F(x)$ is a Laurent polynomial in each cluster of $\cala_G$. By assumption, $F(x) \in U(\cala_G) = \cala_G$. This implies the equalities of algebras of the statement.
\end{proof}

\subsection{The surface type}
We let $(S,M)$ be a surface with an admissible group $G$ of $T$-automorphisms of $(S,M)$ where $T$ is a given tagged triangulation of $(S,M)$. As usual, we exclude the sphere with 1, 2, or 3 punctures and the once-punctured torus. We will also exclude the case of a once-punctured closed surface such that all arcs belong to the same $G$-orbit. Indeed, in this case, an orbit mutation corresponds to changing the taggings at the unique puncture. However, there does not exist a sequence of standard mutations that will have this overall effect.

\begin{Prop} \label{PropositionMutationConnected} Let $(S,M)$ be a surface with triangulation $T$ and assume that $(Q,W)$ is Jacobi-finite, where $W$ is the Labardini potential. Let $G$ be an admissible group of $T$-automorphisms of $(S,M)$. Assume that if $(S,M)$ is a once-punctured closed surface, then there are at least two orbits of tagged arcs. Then the category $\C$ is $G$-mutation connected.
\end{Prop}

\begin{proof} Let $\mathcal{T}$ denote the set of tagged arcs of $(S,M)$ that can be obtained from $T$ by a finite sequence of standard mutations. For instance, if $\partial S \ne \emptyset$, then $\mathcal{T}$ contains all tagged arcs. First, recall from \cite[Section 3.4]{AmiotGen} that since $W$ is non-degenerate, there is a bijection $\Psi: \mathcal{X} \rightarrow \mathcal{T}$ that commutes with the standard mutations. Moreover, for indecomposable objects $X,Y \in \mathcal{X}$, we have that $\Ext^1_\C(X,Y) = 0$ if and only if $\Psi(X), \Psi(Y)$ are compatible tagged arcs. Let $X$ be an indecomposable summand of the initial cluster-tilting object $U$. Then, clearly, $\Psi(gX) = g \Psi(X)$ for all $g \in G$. It follows from the results of Section \ref{sect cto} that if $Y = \mu_X(U)$ is the indecomposable rigid object obtained by mutating $U$ in direction $X$, then for $g \in G$, we have $gY = g \mu_X(U) = \mu_{gX}(gU) = \mu_{gX}(U)$. Since $\Psi$ and $\mu$ commute, we get $\Psi(gY) = \Psi(\mu_{gX}(U))=\mu_{\Psi(gX)}(\Psi(U)) = \mu_{\Psi(gX)}(T)$. On the other hand, we have $g\Psi(Y)= g\Psi(\mu_X(U))=g \mu_{\Psi(X)}(T) = \mu_{g\Psi(X)}(gT)= \mu_{\Psi(gX)}(T)=\Psi(gY)$.  Now, let $U'$ be a reachable cluster-tilting object, $X'$ and indecomposable summand of $U'$ and $U'/X'\oplus Y'$ the cluster-tilting object obtained by mutation in $X'$. Let $Y' = \mu_{X'}(U')$. Assume we know that $\Psi(gX') = g\Psi(X')$ for all $g \in G$. Then, we have $gY' = g \mu_{X'}(U') = \mu_{gX'}(gU')$. Thus in terms of the corresponding arcs in $(S,M)$ we have $\Psi(gY') = \Psi(\mu_{gX'}(gU'))=\mu_{\Psi(gX')}(\Psi(gU'))$. On the other hand, we have $g\Psi(Y')=g\Psi(\mu_{X'}(U'))= g \mu_{\Psi(X')}(\Psi(U')) = \mu_{g\Psi(X')}(g\Psi(U'))= \mu_{\Psi(gX')}(\Psi(gU'))=\Psi(gY')$. This shows, by induction, that $\Psi$ commutes with the action of $G$. Therefore, the $G$-stable cluster-tilting objects from $\mathcal{X}$ correspond to the $G$-stable tagged triangulation under $\Psi$.

Let $H$ in $\C$ be a $G$-stable cluster-tilting object, and let $H_1$ be an indecomposable direct summand of $H$. We let $GH_1 = \oplus_{g \in G}gH_1$. Let $F: \C \to \C_G$ be the $G$-precovering functor. Recall that $F(H)$ is a (non-basic) cluster-tilting object of $\C_G$ and $F(GH_1) \cong F(H_1)^{|G|}$ where $F(H_1)$ is indecomposable. It follows from \cite{IY} that there is a unique indecomposable rigid object $Z$ in $\C_G$ with $Z \not \cong F(H_1)$ such that $F(H/GH_1)\oplus Z$ is cluster-tilting. Therefore, if $V, V'$ are $G$-stable having each $|G|$ indecomposable direct summands, with both $(H/GH_1)\oplus V$ and $(H/GH_1)\oplus V'$ cluster-tilting, then each of $F(V), F(V')$ is isomorphic to $F(H_1)^{|G|}$ or $Z^{|G|}$. Assume that $F(V) \cong F(H_1)^{|G|}$. Applying the adjoint functor $\bar F$ and using Lemma \ref{Adjoint}, we get $V^{|G|} \cong \oplus_{g \in G}gV \cong \bar F F (V) \cong \bar F(F(H_1)^{|G|}) \cong (GH_1)^{|G|}$, so $GH_1 \cong V$. Therefore, if $V \not \cong GH_1$ and $V' \not \cong GH_1$, then $F(V) \cong F(V')$. Thus, $(V)^{|G|} \cong \oplus_{g \in G}gV \cong \bar F F (V) \cong \bar F F (V') \cong (V')^{|G|}$. This yields $V \cong V'$. This shows that orbit mutation in $\C$ is unique up to isomorphism. As we have seen in Section \ref{TechnicalSection}, $\mathcal{T}$ is closed under orbit mutation of tagged arcs. This means that the tagged arcs in $\Psi(GH_1)$ can be replaced by a $G$-orbit $\mathcal{Z}$ of tagged arcs in $\mathcal{T}$ such that $\Psi(H/GH_1)\cup \mathcal{Z}$ is a $G$-stable tagged triangulation. Lifting through $\Psi$, this corresponds to a $G$-stable cluster-tilting object and has to coincide with $(H/GH_1)\oplus V$ by uniqueness of orbit mutation in $\C$.
\end{proof}

\begin{Theo}
Assume that $(S,M)$ is as above. The algebra $\cala_G$ generated by $F(\mathcal{X}_G)$ coincides with the generalized cluster algebra $\cala(S_G, M_G, \mathcal{O})$.
\end{Theo}

\begin{proof}
Let $\mathbf{x}(T) = \{x_{ij} \mid 1 \le i \le s, 1 \le j \le |G|\}$ be the initial cluster corresponding to the tagged triangulation $T$ of $(S,M)$ and consider $\{x_{11}, \ldots, x_{1,|G|}\}$ corresponding to the $G$-orbit of the arc $x_{11}$ in $T$. Recall that $y_i = F(x_{ij})$ for all $1 \le j \le |G|$.  We have seen that there is a finite sequence of flips going from $T$ to another tagged triangulation $T'$ such that $G$ is an admissible group of $T'$-automorphisms of $(S,M)$.
The corresponding cluster is given  by $\mathbf{x}(T')=\mathbf{x}(T)\backslash \{x_{11}, \ldots, x_{1,|G|}\} \cup \{x_{11}', \ldots, x_{1,|G|}'\}$, and $F(x_{1i}x_{1j}') \in \mathbb{Z}[y_1, \ldots, y_s]$ is independent of the chosen $i,j$. It follows from the results of Section \ref{TechnicalSection} that $F(x_{1i}x_{1j}')$ is the polynomial $p_{y_1}$ of Definition \ref{def ca}. This implies the statement.
\end{proof}

The above theorem allows one to perform mutations in $\cala_G$ directly using the exchange relations listed in Definition~\ref{def ca}. For a general $G$-mutation connected category $\C$, we do not know how to mutate the cluster variables in $\cala_G$.

\subsection{Cluster characters}
Recall from \cite{Palu} that we have a cluster character $X$ in $\C$, that is, a function $X: \C \to \mathbb{Z}[x_1^{\pm 1}, \ldots, x_n^{\pm 1}]$ such that $X_{M \oplus N} = X_{M} X_N$, $X_M=X_{M'}$ if $M \cong M'$, and if ${\rm Hom}_\C(M,N[1])$ is one dimensional, then $X_MX_N = X_B + X_{B'}$ where $B,B'$ are the two middle terms of the two non-split distinguished triangles with end-terms $M,N$.

\begin{Prop} \label{PropOneDim} Suppose $(Q,W)$ is non-degenerate and let $T$ be a basic cluster-tilting object of $\C$ obtained by a finite sequence of mutations from $U$. Let $Y$ be an indecomposable direct summand of $T$ and let $Y^*$ be indecomposable non-isomorphic to $Y$ such that $(T/Y)\oplus Y^*$ is cluster-tilting. Then ${\rm Hom}_\C(Y,Y^*[1])$ is one dimensional.
\end{Prop}

\begin{proof} Since $k$ is algebraically closed, the endomorphism algebra of $Y$, modulo its radical, is isomorphic to $k$. Since $(Q,W)$ is non-degenerate, the quiver of ${\rm End}_\C(T)$ has no loop and no $2$-cycle. Now, the result follows from the argument of the proof of Proposition 6.14 in \cite{BMRRT}.
\end{proof}

Now, let $\calc_G^F$ be the set of all objects  of $\C_G$ having all its direct summands in the image of $F\colon   \C \to \C_G$. Let $\bar M \in \C$ be such that $F(\bar M) = M$ and define a function $X^G\colon \calc_G^F\to \Z[y_1^{\pm 1}, \ldots, y_s^{\pm 1}] $
by $X^G_M = F(X_{\bar M}) $.

\begin{Prop} \label{PropClusterChar}
The  function $X^G$ is well-defined and constant within each isomorphism class.
\end{Prop}

\begin{proof} Let $X = X_1 \oplus \cdots \oplus X_r$ where all $X_i$ are indecomposable.  Let $Y_1, Y_2 \in\C$ such that $F(Y_1)\cong F(Y_2) \cong X$.
 By using the Krull-Remak-Schmidt property in $\C_G$ and the right adjoint $\bar F: \C_G \to \C$, it is not hard to show that indecomposability is preserved by $F$. It follows, for $i = 1,2$, that $Y_i = Y_{i1}\oplus \cdots \oplus Y_{ir}$ such that $F(Y_{ij})\cong X_{ij}$. Now, for each $1 \le j \le r$, we have $$\oplus_{g \in G}gY_{1j} \cong \bar F F(Y_{1j}) \cong \bar F F (Y_{2j}) \cong \oplus_{g \in G}gY_{2j}.$$
By the Krull-Remak-Schmidt property in $\C$, we get that $Y_{1j} \cong g_j Y_{2j}$ for some $g_j \in G$. Observe that if $Y,Y'$ are indecomposable with $Y \cong Y'$ in $\C$, then $X_Y = X_{Y'}$, hence $F(X_Y) = F(X_{Y'})$. Therefore, we need to prove that for $g \in G$ and $Y$ an indecomposable object in $\C$, we have $F(X_{Y}) = F(X_{gY})$.
  Consider the cluster-tilting object $U = \Gamma(Q,W)$ in $\C$ and let $\Lambda = {\rm End}_\C(U)$. Now, $g$ induces an auto-equivalence of $\C$ and of ${\rm mod}\Lambda$. Let $K_0^{\rm sp}({\rm mod}\Lambda)$ denote the split Grothendieck group of $\mmod \Lambda$. Let $\langle -, - \rangle$ denote the bilinear form $K_0^{\rm sp}({\rm mod}\Lambda) \to \Z$ such that for $M,N \in {\rm mod}\Lambda$, we have $\langle M, N \rangle = {\rm dim}\Hom_\C(M,N)-{\rm dim}\Ext^1_\C(M,N)$ and $\langle -, - \rangle_a$ the antisymmetric bilinear form such that $\langle M, N \rangle_a = \langle M, N \rangle - \langle N, M \rangle$. It follows from a result of Palu \cite{Palu} that $\langle -,- \rangle_a$ descends to the usual Grothendieck group $K_0({\rm mod}\Lambda)$. For $M,N \in {\rm mod}\Lambda$, we have
\begin{eqnarray*}\langle M, N \rangle & = & {\rm dim}\Hom_\C(M,N)-{\rm dim}\Ext^1_\C(M,N)\\
& = &{\rm dim}\Hom_\C(g M,g N)-{\rm dim}\Ext_\C(g M,g N)\\
& = & \langle gM, gN \rangle.
\end{eqnarray*}
In a similar way, we have $\langle {\rm dim}M, {\rm dim}N \rangle_a = \langle {\rm dim}gM, {\rm dim}gN \rangle_a$. Let $U = \oplus_{i \in I}U_i$ be a decomposition of $U$ into indecomposable direct summands. Observe that each $g$ induces a permutation $I \to I$ with no fixed point. For each $U_i$, let $S_i$ be the simple top of the projective $\Lambda$-module $\Hom_\C(U,U_i)$. Let also $x_i$ denote the initial cluster variable associated to $U_i$. Observe that $gU_i \cong U_{gi}$, $g S_i \cong S_{gi}$ and $F(x_i) = F(x_{gi})$. Let $Z = \Hom_\C(U,Y)$ and $gZ = \Hom_\C(gU,gY)\cong \Hom_\C(U,gY)$, where the last isomorphism is an isomorphism of $\Lambda$-modules. Assume first that $Z$ is non-zero, so that $Y$ is not isomorphic to a shift of an indecomposable object in $U$. Observe that $g$ induces an isomorphism of projective varieties

\medskip

\noindent $g: {\rm Gr}_{\bf e}(Z) = \{L \subset Z \mid L \; \text{submodule of}\; Z, {\rm dim} L = {\bf e}\}$

\medskip

\noindent  $\qquad \qquad \qquad\qquad \to {\rm Gr}_{g {\bf  e}}(gZ) = \{L \subset gZ \mid L \; \text{submodule of}\; gZ, {\rm dim} L = {g{\bf e}}\}$

\noindent and hence, $\chi({\rm Gr}_{\bf e}(Z)) = \chi({\rm Gr}_{g \bf e}(gZ))$. Now, we have
\begin{eqnarray*}X_Y &=& \sum_{\bf e}\chi({\rm Gr}_{\bf e}(Z))\prod_ix_i^{\langle {\rm dim}S_i, {\bf e}\rangle_a - \langle S_i, Z\rangle}\\
& = &\sum_{g\bf e}\chi({\rm Gr}_{g{\bf e}}(gZ))\prod_{i}x_{i}^{\langle {\rm dim}S_{gi}, g{\bf e}\rangle_a - \langle S_{gi},gZ\rangle}\\
& = &\sum_{\bf e}\chi({\rm Gr}_{\bf e}(gZ))\prod_{i}x_{g^{-1}i}^{\langle {\rm dim}S_{i}, {\bf e}\rangle_a - \langle S_{i}, gZ\rangle},
\end{eqnarray*}
as one can identify the dimension vectors of the submodules of $gZ$ as the $g{\bf e}$ where ${\bf e}$ runs through the dimension vectors of the submodules of $Z$. Now,
$$X_{gY} = \sum_{\bf e}\chi({\rm Gr}_{\bf e}(gZ))\prod_{i}x_{i}^{\langle {\rm dim}S_{i}, {\bf e}\rangle_a - \langle S_{i}, gZ\rangle}$$
and since $F(x_i) = F(x_{g^{-1}i})$ for all $g \in G$, it follows that $F(X_Y) = F(X_{gY})$. Finally, if $Y=U_i[1]$ is a shift of an indecomposable direct summand of $U$, then $X_Y=x_i$ while $X_{gY} = gx_i$. Clearly, $F(X_Y) = F(X_{gY})$ in this case as well.
\end{proof}

\begin{Remark} One can also define a cluster character $X': \C_G \to \Z[y_1^{\pm 1},\ldots,y_s^{\pm 1}]$ directly.  However, we have $X'_M=X^G_M$ only if, for $F(\bar M) = M$, we have
\begin{equation}
\label{eq chi}
\sum_{F({\bf e'})=\bf e}\chi({\rm Gr}_{\bf e'}(\Hom_{\C}(U,\bar M))) = \chi({\rm Gr}_{\bf e}(\Hom_{\C_G}(FU,M))),
\end{equation}
for all ${\bf e}$; but this is not always true. Indeed, the module $\Hom_{\C_G}(F(U),M)$ may have a submodule of dimension vector ${\bf e}$ such that $\bar M$ has no submodule of dimension vector ${\bf e'}$ with $F({\bf e'}) = {\bf e}$. Even when $F$ is dense, that is, when $F$ is a $G$-covering, we do not know whether the above equality on the Euler characteristics of Grassmannians always holds.
\end{Remark}

\begin{Exam}
Let $Q$ be the quiver $$\xymatrix@R0pt{&2\\ 1\ar[ru]^\za\ar[rd]_\zb &&1'\ar[lu]_{\zb'}\ar[ld]^{\za'}\\&2'}$$ with group $G=\mathbb{Z}/2\mathbb{Z} $ acting by rotation. Then the quiver $Q_G$ is the Kronecker quiver
$\xymatrix{1\ar@<1pt>[r]^{\za}\ar@<-1pt>[r]_{\zb}&2}$. Here both potentials $W,W_G$ are zero. Let $M$ be the representation
$$\xymatrix@R0pt{&k\\ k\ar[ru]^1\ar[rd]_1&&k\ar[lu]_{1}\ar[ld]^{1}\\&k}$$
and let $M_G$ denote its image under $F$. Then
$$M_G =
\xymatrix@R55pt{k^2\ar@<1pt>[r]^-{ \left[\begin{smallmatrix}
        1 & 0 \\ 0 & 1
    \end{smallmatrix}\right]} \ar@<-1pt>[r]_-{ \left[\begin{smallmatrix}
        0 & 1 \\ 1 & 0
    \end{smallmatrix}\right]}&k^2},$$
which is isomorphic to the direct sum
$$M_G \cong
\xymatrix@R55pt{k\ar@<1pt>[r]^1 \ar@<-1pt>[r]_{1}&k} \oplus
\xymatrix@R55pt{k\ar@<1pt>[r]^1 \ar@<-1pt>[r]_{-1}&k}.$$
In particular, $M_G$ has two subrepresentations of dimension vector $e=(1,1)$. On the other hand, $M$ has no subrepresentation with a dimension vector $e'$ such that $F(e')=e$. Thus for $e=(1,1)$, the left hand side of equation (\ref{eq chi}) is zero, while the right hand side is not. Therefore the cluster characters $X'$ and the function $X^G$ are not equal in this example. Moreover, $F$ is not dense.
\end{Exam}

The following yields a third way to get the generalized cluster algebra structure on an orbifold $(S_G, M_G, \mathcal{O})$, using Proposition \ref{PropositionMutationConnected}.

\begin{Prop} Let $\C$ be $G$-mutation connected and let $W$ be non-degenerate. Then the $\mathbb{Z}$-subalgebra of $\mathcal{F}_G$ generated by all $X^G_N$, where $N$ is an indecomposable direct summand of a cluster-tilting object obtained from $U_G$ by a finite sequence of standard mutations, coincides with $\mathcal{A}_G$.
\end{Prop}

\begin{proof}
Since $\C$ is $G$-mutation connected, one can identify the indecomposable direct summands of the cluster-tilting objects obtained from $U_G$ by a finite sequence of standard mutations by the $F(N)$ where $N \in \mathcal{X}_G$. Now, $X^G_{F(N)} = F(X_N)$, and $F(\mathcal{X}_G)$ is the set of generalized cluster variables for $\mathcal{A}_G$.
\end{proof}

Let $\tau$ denote the Auslander-Reiten translation in $\C$. When $k = \mathbb{C}$, the cluster character $X: \C \to \mathbb{Z}[x_1^{\pm 1}, \ldots, x_n^{\pm 1}]$  is such that if $$M_1 \to M_2 \to M_3 \to M_1[1]$$
is an Auslander-Reiten triangle in $\C$, then $X_{M_1}X_{M_3} = 1 + X_{M_2}$; see \cite{DG}. When $F$ is dense, we get a function $X^G: \C_G \to \mathbb{Z}[y_1^{\pm 1}, \ldots, y_s^{\pm 1}]$ as defined in the previous section. A natural question arises here. Is it a cluster character? The next results answers the latter question affirmatively.

\begin{Prop} Assume that $F$ is dense. Then the function $X^G: \C_G \to \mathbb{Z}[y_1^{\pm 1}, \ldots, y_s^{\pm 1}]$ is a cluster character. If $k = \mathbb{C}$ and $L \to M \to N \to L[1]$ is an Auslander-Reiten triangle in $\C_G$, then $X^G_LX^G_N = X^G_M+1$.
\end{Prop}

\begin{proof} Notice now that since $F$ is dense, the construction of $X^G$ extends to any object of $\C_G$. It is clear that $X^G$ is constant within an isomorphism class. Let $M_1, M_2 \in \C_G$. Then $M_i=F(N_i)$ for $N_1,N_2 \in \C$. Therefore, $X^G_{M_1 \oplus M_2} = F(X_{N_1 \oplus N_2}) = F(X_{N_1}X_{N_2}) = F(X_{N_1})F(X_{N_2})=X^G_{M_1}X^G_{M_2}$. Assume now that $\Hom_{\C_G}(M_1, M_2[1])$ one dimensional. We have
$$\oplus_{g \in G}\Hom_\C(N_1,gN_2[1]) \cong \Hom_{\C_G}(M_1, M_2[1]).$$
Therefore, there is exactly one $g \in G$ with $\Hom_\C(N_1,gN_2[1])$ one dimensional. By the $2$-Calabi-Yau property of $\C$, we have that $\Hom_\C(gN_2,N_1[1])$ is one dimensional and $\Hom_\C(g'N_2,N_1[1])=0$ if $g' \ne g$. Consider the non-split exact triangles
$$gN_2 \to B \to N_1 \to gN_2[1]$$
$$N_1 \to B' \to gN_2 \to N_1[1]$$
in $\C$. We know that $X_{N_1}X_{gN_2} = X_B + X_{B'}$. Since $F$ is exact, we get exact triangles
$$\eta_1: \quad F(gN_2) \to F(B) \to F(N_1) \to F(gN_2)[1]$$
$$\eta_2: \quad F(N_1) \to F(B') \to F(gN_2) \to F(N_1)[1]$$
where $F(gN_2) \cong F(N_2) = M_2$ and $F(N_1) = M_1$. Therefore \begin{eqnarray*}X^G_{M_1}X^G_{M_2} &=& F(X_{N_1})F(X_{N_2})\\ &=& F(X_{N_1})F(X_{gN_2})\\ &=& F(X_{N_1}X_{gN_2})\\ & = & F(X_B + X_{B'})\\ & = & F(X_B) + F(X_{B'})\\& =& X^G_{F(B)} + X^G_{F(B')}.\end{eqnarray*} Since $\eta_1, \eta_2$ are clearly non-split, we get that $X^G$ is a cluster character.

The second part of the proposition about Auslander-Reiten triangles follows from Proposition \ref{PropARQUivers} and the remark above this proposition.
\end{proof}

\subsection{The finite representation type}
We assume that $G$ is an admissible group of automorphisms of $(Q,W)$ where $(Q,W)$ is Jacobi-finite. We have seen that we have an induced functor $F: \C \to \C_G$ which is a $G$-precovering. We call a Hom-finite Krull-Schmidt $k$-category $\mathcal{B}$ of \emph{finite type} if $\mathcal{B}$ has finitely many indecomposable objects, up to isomorphism.

\begin{Prop} \label{FiniteType}
The category $\C$ is of finite type if and only if the category $\C_G$ is of finite type.
\end{Prop}

\begin{proof} By \cite[Cor. 4.4]{KZ} and \cite{BMR}, the category ${\rm mod}B$ of finite dimensional representations of $B$ is equivalent to $\C/U[1]$ where $U$ is the cluster-tilting object corresponding to $\Gamma(Q,W)$. Hence, $B$ is of finite type if and only if $\C$ is of finite type. Similarly, $B_G$ is of finite type if and only if $\C_G$ is of finite type. Now, by Corollary~\ref{CovJacobian} we have a $G$-covering $B \to B_G$. Since the characteristic of $k$ does not divide $|G|$, it follows from a result of Gabriel \cite[Lemma 3.4]{Gabriel} that if $B$ is of finite type, then $B_G$ is of finite type. Finally, by \cite[Lemma 3.3]{Gabriel}, if $B_G$ is of finite type, then the algebra $B$ is of finite type.
\end{proof}

\begin{Prop} \label{PropARQUivers}
Assume that one of $\C, \C_G$ is of finite type. Then $F : \C \to \C_G$ is a $G$-covering that preserves indecomposability and Auslander-Reiten triangles. In particular, $F$ induces a $G$-covering of Auslander-Reiten quivers of $\C$ and $\C_G$.
\end{Prop}

\begin{proof} By Proposition \ref{FiniteType}, we know that $\C$ is of finite type. As seen in the proof of Proposition \ref{FiniteType}, this implies that both $B, B_G$ are of finite type.
 Let $U$ be the cluster-tilting object of $\C$ corresponding to $\Gamma(Q,W)$.
 Let $\mathcal{I}$ be the ideal of $\C$ of the morphisms which factorize through $U[1]$ and $\mathcal{J}$ be the ideal of $\C_G$ of the morphisms which factorize through $FU[1]$. One can check that for any $M,N \in \C$, we have isomorphisms $$\Hom_\mathcal{J}(FM,FN) \to \oplus_{g \in G}\Hom_\mathcal{I}(M,gN)$$ and $$\Hom_\mathcal{J}(FM,FN) \to \oplus_{g \in G}\Hom_\mathcal{I}(gM,N).$$
Thus, according to the definition in Section \ref{Covering} and the fact that $\textup{mod}\,B=\C/\mathcal{I}$ and $\textup{mod}\,B_G=\C_G/\mathcal{J}$, we see that
 $F$ induces a $G$-precovering $\tilde F: {\rm mod}B \to {\rm mod}B_G$. By Theorem 4 in \cite{dlPMV}, the functor $\tilde F$ sends indecomposable objects to indecomposable objects.
 Consequently, if $M$ is an indecomposable object in $\calc\setminus \textup{add}\, U[1]$  then $ F M$ is indecomposable in $\calc_G$. On the other hand, if $M$ is an indecomposable summand of $U[1]$ then $FM$ is an indecomposable summand of $F(U)[1]$. This shows that all indecomposable objects of $\calc_G$ are isomorphic to an object in the image of $F$. Thus the $G$-precovering $F$ is dense and hence a $G$-covering.

Now
assume that $$\eta: \quad L \stackrel{u}{\to} M \stackrel{v}{\to} N \to L[1]$$ is an Auslander-Reiten triangle in $\C$. The exact functor $F$ sends this distinguished triangle to the distinguished triangle $$F(\eta): \quad FL \stackrel{Fu}{\to} FM \stackrel{Fv}{\to} FN \to FL[1].$$ We know that $FL, FN$ are indecomposable from what was shown above. Let $Z$ be any indecomposable object in $\C_G$ and $\bar Z$ be such that $F(\bar Z) = Z$. Let $f: Z \to FN$ be a non-isomorphism. By the adjunction property of Lemma \ref{Adjoint}, we get an isomorphism
$$\Hom_{\C_G}(Z,FN) \to \oplus_{g \in G}\Hom_{\C}(\bar Z,gN).$$
Therefore, there exists $(f_g)_{g \in G}$ with $f=\sum_{g \in G}F(f_g)$. Now, recall that the $gN$ for $g \in G$ are pairwise non-isomorphic since $G$ acts freely on the indecomposable objects of $\C$. Therefore, there is at most one $f_g$ that is an isomorphism. Since $F$ is exact and $FX$ is non-zero whenever $X$ is non-zero, we see that a non-isomorphism is sent to a non-isomorphism through $F$. Now if one $f_g$ is an isomorphism, then the morphism $\sum_{g \in G}F(f_g) = f$ is the sum of an isomorphism and a nilpotent endomorphism, thus $f$ is an isomorphism, a contradiction. Therefore, no $f_g$ is an isomorphism. Since for $g \in G$, we have that $f_g: \bar Z \to gN$ is a non-isomorphism between indecomposable objects and $g\eta$ is an Auslander-Reiten triangle, we get that $f_g$  factors through $gv$, meaning that $F(f_g)$ factors through $F(gv)=F(v)$. Since each $F(f_g)$ factors through $F(v)$, we see that $f$ factors through $F(v)$. This proves that $F(\eta)$ is an Auslander-Reiten triangle. Therefore, $F$ sends Auslander-Reiten triangles to Auslander-Reiten triangles and the second part of the statement follows.
\end{proof}

\begin{Prop} \label{PropLift}Let $\C$ be of finite type
and let $V \in \C_G$ be cluster-tilting. Then there exists a cluster-tilting object $Z \in \C$ such that $F Z = V$.
\end{Prop}

\begin{proof} Since $F$ is dense, there exists $\bar V \in \C$  such that $F \bar V = V$. We need to prove that $Z:=\oplus_{g \in G}g\bar V$ is cluster-tilting. We have $\Hom_{\C_G}(V, V[1])=0$. This means $\oplus_{g \in G}\Hom_{\C}(\bar V,g \bar V[1])=0$. Similarly, we get $$\oplus_{g \in G}\Hom_{\C}(g'\bar V,g \bar V[1])=0$$ for any $g' \in G$. In particular, $Z$ is rigid. Let $Y \in \C$ be indecomposable with $\Hom_{\C}(Z,Y[1])=0$. Thus, $\oplus_{g \in G}\Hom_{\C}(g\bar V,Y[1])=0$. Then $\Hom_{\C_G}(F\bar V,FY[1])=0$. Hence, we get $\Hom_{\C_G}(V,FY[1])=0$. Since $V$ is cluster-tilting, we know that $FY$ is a summand of $V$. By applying the adjoint $\bar F: \C_G \to \C$ to $F$, we get that $\bar F F Y \cong \oplus_{g \in G}gY$ is a direct summand of $\bar F V = \bar F F \bar V \cong \oplus_{g \in G}g \bar V = Z$. In particular, $Y$ is a direct summand of $Z$. This proves that $Z$ is cluster-tilting.
\end{proof}

In what follows, we call $\C$ of \emph{acyclic type} if there is a cluster-tilting object $M$ of $\C$ such that the quiver of ${\rm End}_{\C}(M)$ has no oriented cycles. By \cite{KR}, this means that $\C$ is equivalent to the (classical) cluster category of a finite quiver without oriented cycles. Observe also that if $(Q,W)$ is non-degenerate and $\C$ is of finite type, then $\C$ is just the (classical) cluster category of a quiver of Dynkin type.

\begin{Prop} Assume that $\C$ is of acyclic and of finite type. Then the indecomposable rigid objects in $\C_G$ are precisely the $\{F(V_i) \mid i \in I\}$, where the $\{V_i \mid i \in I\}$ form a complete set of the representatives of the $G$-orbits of those indecomposable rigid objects $V$ in $\C$ with $\Hom_{\C}(V,gV[1])=0$ for all $g \in G$. Therefore, the generalized cluster variables in $\cala_G$ can be obtained by the following methods.
\begin{enumerate}[$(1)$]
    \item The $F(X_{V_i})$ for $i \in I$.
    \item The $X^G_Y$ where $Y$ is rigid in $\C_G$.
\end{enumerate}
\end{Prop}

\begin{proof}
Since $\C$ is of acyclic type, every indecomposable rigid object in $\C$ is a summand of a cluster-tilting object that can be obtained from $U$ by finitely many mutations. Therefore, $\C$ is $G$-mutation connected. It follows from the argument of the proof of Proposition \ref{PropLift} that for $X$ indecomposable in $\C$, $F(X)$ is rigid in $\C_G$ if and only if $\Hom_{\C}(X,gX[1])=0$ for all $g \in G$. Moreover, all indecomposable rigid objects of $\C_G$ can be obtained this way. Clearly, for $X_1, X_2$ rigid in $\C$, we have $F(X_1) \cong F(X_2)$ if and only if $X_1, X_2$ lie in the same $G$-orbit, up to isomorphism. This yields the main part of the proposition. This also shows that $(1)$ and $(2)$ give the same elements.
Now, part $(1)$ gives the description of the generalized cluster variables in $\cala_G$ by definition, since it is well known in this case that the cluster variables of $\cala$ are given by the $X_V$ where $V$ is indecomposable rigid in $\C$.
\end{proof}

We present two examples for illustration.

\begin{Exam}
Let $S$ be the disk with $6$ marked points on the boundary represented by a regular hexagon. Let $T$ be a triangulation of $(S,M)$ such that a rotation of $2\pi/3$ fixes $T$,  see Figure \ref{figlem10}. Let $G = \Z_3$ be the cyclic group of order $3$ generated by a rotation of $2\pi/3$. We let $x_1, x_2, x_3$ be the initial cluster variables corresponding to the arcs of $T$. The quiver $Q$ is an oriented cycle of length $3$  and the potential is this cycle.  In the following picture, we put the Auslander-Reiten quiver of $\C$ where each indecomposable $M$ of $\C$ is replaced by its cluster variable $X_M$.
$$\xymatrixrowsep{12pt}\xymatrixcolsep{10pt}\xymatrix{&&x_3 \ar[dr]&&\frac{x_1 + x_2 + x_3}{x_1x_3}\ar[dr]&&& \\ &\frac{x_1 + x_3}{x_2}\ar[dr] \ar[ur]&&\frac{x_2 + x_3}{x_1}\ar[dr]\ar[ur]&&\frac{x_1 + x_2}{x_3}\ar[dr]\ar@{.>}[ur]&& \\ x_1\ar[ur]&&\frac{x_1 + x_2 + x_3}{x_1x_2}\ar[ur]&&x_2\ar[ur]&&\frac{x_1 + x_2 + x_3}{x_2x_3}\ar@{.>}[ur]&}$$
We know that the set $\calx$ of cluster variables of $\cala$ consists of the  $X_M$ where $M$ is any indecomposable object of $\C$.  On the other hand, the set $\calx_G$ contains only the 6 cluster variables of the top row and the bottom row. Setting $x_1,x_2,x_3$ equal to $y_1$ and making the appropriate identifications in the quiver, we obtain the following picture of the Auslander-Reiten quiver of $\C_G$ where each indecomposable $Y$ of $\C_G$ is replaced by $X^G_Y$.
$$\xymatrixrowsep{12pt}\xymatrixcolsep{10pt}\xymatrix{&2\ar@/_0.5pc/[dr] \ar@/_0.5pc/[dl]&&\\ y_1\ar@/_0.5pc/[ur]&&3/y_1\ar@/_0.5pc/[ul]&}$$

The set of cluster variables $F(\calx_G)$ is  $\{y_1,3/y_1\}$ while the set $F(\calx) $ is $\{y_1,3/y_1,2\}$. Note that both sets generate the same algebra, thus $\cala_G=\cala/ \langle x_{i,j}- x_{i, j'}\rangle$.

If $V$ denotes the indecomposable object of $\C_G$ labeled by a $2$ and $\bar V$ is a lift of it, then $\oplus_{g \in G}g\bar V$ is not rigid in $\C$. So the object with character $2$ is not rigid, even though it comes from a rigid object in $\C$. Observe that an Auslander-Reiten triangle $L \to M \to N \to L[1]$ of $\C_G$ satisfies $X^G_LX^G_N = X^G_M+1$.
\end{Exam}

\begin{Exam}
Let $S$ be the once-punctured disk with $4$ marked points on the boundary. Let $T$ be the  triangulation of $(S,M)$ such that a rotation of $\pi/4$ fixes $T$. Let $G = \Z_4$ be the cyclic group of order $4$ generated by a rotation of $\pi/4$. We let $x_1, x_2, x_3, x_4$ be the initial cluster variables corresponding to the arcs of $T$. The quiver $Q$ is an oriented cycle of length $4$ with arrows $\alpha, \beta, \gamma, \delta$ and the potential is $W = \alpha\beta\gamma\delta$. In the following picture, we put the Auslander-Reiten quiver of $\C$ where each indecomposable $M$ of $\C$ is replaced by $X_M$.
$$\xymatrixrowsep{8pt}\xymatrixcolsep{8pt}\xymatrix{&&x_1 \ar[ddr]&&f_2\ar[ddr]&&x_3 \ar[ddr]&&f_4 \ar@{.>}[ddr] \\
&&f_1 \ar[dr]&&x_2\ar[dr]&&f_3 \ar[dr]&&x_4\ar@{.>}[dr] \\
&g_2\ar[dr]\ar[uur] \ar[ur]&&g_3\ar[dr]\ar[ur]\ar[uur]&&g_4\ar[dr]\ar[ur]\ar[uur]&&g_1\ar@{.>}[dr]\ar[ur]\ar[uur]&&g_2 \\ \frac{x_1 + x_3}{x_2}\ar[ur]&&\frac{x_2 + x_4}{x_3}\ar[ur]&&\frac{x_1 + x_3}{x_4}\ar[ur]&&\frac{x_4 + x_2}{x_1}\ar[ur]&&\frac{x_1 + x_3}{x_2}&}$$
where
$$f_{i} = \frac{x_4x_1 + x_1x_2 + x_2x_3 + x_3x_4}{x_1x_2x_3x_4}x_i$$
and
$$g_i = \frac{x_4x_1 + x_1x_2 + x_2x_3 + x_3x_4 - x_ix_{i+1}}{x_ix_{i+1}}$$
and where indices are taken modulo $4$.
We know that the set $\calx$ of cluster variables of $\cala$ consists of  the 16 variables $X_M$ where $M$ is any indecomposable object of $\C$.  On the other hand, the set $\calx_G$ contains only the 8 cluster variables $x_i,f_i$.
Again setting $x_1,x_2,x_3,x_4$ equal to $y_1$ and making the appropriate identifications in the quiver, we obtain the following picture of the Auslander-Reiten quiver of $\C_G$ where each indecomposable $Y$ of $\C_G$ is replaced by $X^G_Y$.
$$\xymatrixrowsep{12pt}\xymatrixcolsep{10pt}\xymatrix{&y_1 \ar@/_0.5pc/[dl]&\\ 3\ar@/_0.5pc/[ur] \ar@/_0.5pc/[rr] \ar@/_0.5pc/[dr]
&&4/y_1\ar@/_0.5pc/[ll]\\&2\ar@/_0.5pc/[ul] &}$$
The set of cluster variables $F(\calx_G)$ is $\{y_1,4/y_1\}$ whereas the set $F(\calx) $ is $\{y_1,4/y_1,2,3\}$. Again both sets generate the same algebra, thus $\cala_G=\cala/ \langle x_{i,j}- x_{i, j'}\rangle$.
\end{Exam}

\end{document}